\documentclass[a4paper, 12pt]{article}

\usepackage{xcolor}

\usepackage{mathrsfs}
\usepackage{amsmath,amsthm,amssymb,amsfonts,amscd,amstext}
\usepackage{latexsym}
\usepackage{amssymb}
\usepackage{xspace}
\usepackage[utf8]{inputenc}
\usepackage{hyperref}
\usepackage{enumerate}
\usepackage{bm}
\usepackage{bbm}
\usepackage{verbatim}
\usepackage{mathpazo}
\usepackage[sort&compress,numbers,square]{natbib}
\theoremstyle{plain}
\newtheorem{theorem}{Theorem}[section]
\newtheorem{corollary}[theorem]{Corollary}
\newtheorem{proposition}[theorem]{Proposition}
\newtheorem{lemma}[theorem]{Lemma}
\theoremstyle{definition}
\newtheorem{definition}[theorem]{Definition}
\newtheorem{remark}[theorem]{Remark}
\newtheorem{observation}[theorem]{Observation}
\newtheorem{example}[theorem]{Example}
\newtheorem{examples}[theorem]{Examples}
\numberwithin{theorem}{section}
\numberwithin{equation}{section}

\def\R{\mathbb{R}}
\def\C{\mathbb{C}}
\def\N{\mathbb{N}}

\def\Q{\mathbb{Q}}

\newcommand{\scrS }{\mathscr{S}}

\renewcommand{\k}{\kappa}
\newcommand{\tk}{\tilde{\kappa}}

\newcommand{\Bquad}{{Q}}
\newcommand{\bquad}{{q}}

\newcommand{\eps}{\varepsilon}

\let \Re \relax
\DeclareMathOperator{\Re}{Re}
\let \Im \relax
\DeclareMathOperator{\Im}{Im}

\newcommand{\LS}{Lopatinski\u{\i}-\v{S}apiro\xspace}
\newcommand{\SbarpN}{\ensuremath{\overline{\mathscr{S}}(\RNp)}}

\newcommand{\y}{\varrho}
\newcommand{\tchi}{\tilde{\chi}}
\newcommand{\nhd}{neighborhood\xspace}

\newcommand{\U}{\mathscr U}

\newcommand{\br}{\ensuremath{_{|x_N=0^+}}}
\newcommand{\bd}{\ensuremath{_{|\partial\Omega}}}

\newcommand{\ovl}[1]{\overline{#1}}

\newcommand{\Rdp}{\R^d_+}

\newcommand{\Op}{\ensuremath{\mathrm{Op}}}

\newcommand{\Opt}{\ensuremath{\mathrm{Op}_{\mbox{\tiny ${\mathsf T}$}}}}

\newcommand{\suff}{sufficiently\xspace}

\DeclareMathOperator{\trace}{tr}

\DeclareMathOperator{\range}{Ran}

\newcommand{\transp}{\ensuremath{\phantom{}^{t}}}

\newcommand{\Norm}[2]{{\| #1 \|}_{#2}}

\newcommand{\norm}[2]{{|#1|}_{#2}}

\newcommand{\Normsc}[2]{{\| #1 \|}_{#2, \tilde{\tau}}}
\newcommand{\normsc}[2]{{| #1 |}_{#2, \tilde{\tau}}}

\newcommand{\inp}[2]{( #1, #2 )} 
\newcommand{\biginp}[2]{\big( #1, #2 \big)}

\newcommand{\Normsctau}[2]{{\| #1 \|}_{#2, \tau}}
\newcommand{\normsctau}[2]{{| #1 |}_{#2, \tau}}

\newcommand{\lscttau}{\lambda_{\mbox{\tiny ${\mathsf T}$},\tau}}

\newcommand{\Lscttau}{\Lambda_{\mbox{\tiny ${\mathsf T}$},\tau}}

\newcommand{\Ssctau}{S_\tau}
\newcommand{\Psisctau}{\Psi_\tau}

\newcommand{\Sscttau}{S_{\mbox{\tiny ${\mathsf T}$},\tau}}
\newcommand{\Psiscttau}{\Psi_{\mbox{\tiny ${\mathsf T}$},\tau}}

\newcommand{\RNp}{\R^N_+}
\newcommand{\lsc}{\lambda_{\tilde{\tau}}}
\newcommand{\lsct}{\lambda_{\mbox{\tiny ${\mathsf T}$},\tilde{\tau}}}

\newcommand{\Lsct}{\Lambda_{\mbox{\tiny ${\mathsf T}$},\tilde{\tau}}}

\newcommand{\Ssc}{S_{\tilde{\tau}}}
\newcommand{\Psisc}{\Psi_{\tilde{\tau}}}

\newcommand{\un}[1]{\Sigma_{#1}}
\newcommand{\z}{\mathbf{z}}

\newcommand{\Con}{\ensuremath{\mathscr C}}

\newcommand{\Cinf}{\ensuremath{\Con^\infty}}
\newcommand{\Cinfc}{\ensuremath{\Con^\infty_c}}
\newcommand{\Cbarc}{\ensuremath{\overline \Con^\infty}_c}

\newcommand{\Pell}{\mathsf P}

\newcommand{\scrO}{\mathscr O}

\DeclareMathOperator{\id}{Id}
\DeclareMathOperator{\supp}{supp}
\DeclareMathOperator{\rank}{rank}

\DeclareMathOperator{\Span}{span}

\usepackage{algorithm, algpseudocode} 
\usepackage{mathrsfs} 

\usepackage{lipsum}

\title{Null-controllability for a fourth order parabolic equation under general boundary conditions}
\author{Emmanuel Wend-Benedo Zongo$^1$ and Luc Robbiano$^2$}
\date{
	$^1$Universit\'e Paris-Saclay, Laboratoire de Math\'ematiques d'Orsay (LMO), CNRS, UMR 8628, France.\\ \texttt{emmanuel.zongo361@gmail.com}\\%
	$^2$Universit\'e Versailles Saint-Quentin-en-Yvelines, Laboratoire de Math\'ematiques de Versailles (LMV), CNRS, UMR 8100, France.  \\ \texttt{luc.robbiano@uvsq.fr}\\[2ex]%
}
\begin{document}
\maketitle
\begin{abstract}
 In this paper, we consider a fourth order inner-controlled parabolic equation on an open bounded subset of $\R^d$, or a smooth compact manifold with boundary, along with general boundary operators fulfilling the \LS condition. We derive a spectral inequality for the solution of the parabolic system that yields a null-controllability result. The spectral inequality is a consequence of an interpolation inequality obtained via a Carleman inequality for the bi-Laplace operator under the considered boundary conditions.
\end{abstract}

\tableofcontents

\section{Introduction}
In this paper, our aim is to study the interior null-controllability for a fourth order parabolic equation. A kind of fourth order parabolic equations describes the epitaxial growth of nanoscale thin films which has recently received increasing interest in materials science because compositions like \text{$YBa_2Cu_3O_{7-\delta}$} (YBCO) are expected to be high-temperature super-conducting and could be used in the design of semi-conductors (see for instance \cite{KBB} and the references therein). Therefore, studying the features of fourth order parabolic equations has realistic meanings.\\
Let $\Omega$ be a bounded connected open subset in $\R^d$, or a smooth bounded connected $d$-dimensional manifold, with smooth boundary $\partial\Omega$.
We consider the following controlled parabolic equation
\begin{equation}\label{eq: parabolic}
\begin{cases}
\partial_t y +\Delta^2y =\mathbbm{1}_{\Sigma}v
~~~~\text{in}~(t,x) \in  \R_+ \times\Omega,\\
B_1 y_{|\R_+ \times\partial \Omega}= B_2y_{|\R_+ \times\partial \Omega}=0,\\
y_{|t=0}=y^0\in L^2(\Omega).
\end{cases}
\end{equation}
Here, $\Sigma$ is a nonempty open subset of $\Omega$, $B_1$ and $B_2$ denote two boundary differential operators. The function $v$ is the control and lies in $L^2((0,\infty)\times \Omega).$ It only acts on the solution $y$ in $\Sigma.$ The function $y$ may represent a scalled film height and the term $\Delta^2 y$ represents the capillarity-driven surface diffusion (see \cite{SGK}).\\
\indent The two boundary  operators $B_1$ and $B_2$ are of order $k_j$, $j=1,2$ respectively, yet at most of order $3$ in the direction
 normal to the boundary $\partial\Omega$. They 
are chosen such that the two following properties are fulfilled: 
\begin{enumerate}
  \item the \LS boundary condition holds (this condition is fully
    described in what follows);
  \item along with the homogeneous boundary conditions given above the
    bi-Laplace operator is self-adjoint and nonnegative. 
\end{enumerate}
We are concerned with the question of null-controllability for this controlled parabolic equation which is the following:
\begin{center}
For a given initial data $y^0\in L^2(\Omega),$ and a given time $T>0,$ can one find $v\in L^2((0,T)\times \Sigma)$ such $y(T)=0?$
\end{center}
The answer to this question rely on the derivation of a spectral inequality. With the second property, associated with the operator is then a Hilbert basis $(\varphi_j)_{j\in\mathbb{N}}$ of $L^2(\Omega).$ In the case of ``clamped'' boundary condition ($B_1 y_{|\partial \Omega}=y_{|\partial\Omega}$ and $B_2 y_{|\partial \Omega}=\partial_{\nu}y_{|\partial\Omega}$) the following spectral inequality was proved by J. Le Rousseau and L. Robbiano in \cite{JL}.
\begin{theorem}\textup{(spectral inequality for the ``clamped'' bi-Laplace operator)}.\\
Let $\varnothing$ be an open subset of $\Omega$. There exists $C>0$ such that 
$$\|y\|_{L^2(\Omega)}\leq C e^{C\mu^{1/4}}\|y\|_{L^2(\varnothing)},~~\mu>0,~~
y\in\Span\{\varphi_j;~\mu_j\leq \mu\},~~~\text{with}~~\Delta^2\varphi_j=\mu_j\varphi_j.$$
\end{theorem}
\noindent 
The proof of this theorem is based on a Carleman inequality for the augmented fourth-order elliptic operator 
$D^4_s+\Delta^2$, that is, after the addition of a variable $s.$ We are interested in proving such a spectral 
inequality for general boundary condition, that is, the operator $\Delta^2$ along with the general boundary 
operators $B_1$ and $B_2$ satisfy the \LS condition. 
By considering for instance, the boundary operators $B_1(x,D)=D_d^2+\alpha R(x,D')$ and $B_2(x,D)=D_d$, 
where $D_d=-i\partial_d,$ $\alpha\in\R$ and $R(x,D')$ a tangential elliptic operator of order $2,$ we remark 
that having the \LS condition for $(\Delta^2,B_1,B_2)$ does not imply that the \LS condition holds true 
$(D^4_s+\Delta^2, B_1, B_2).$\\
\\
With this approach, some boundary conditions that satisfy the \LS condition cannot be used to prove the Carleman inequality required to obtain the spectral inequality. However, some boundary conditions allow this, for instance the clamped boundary condition. So, we have determined the boundary conditions allowing to prove such Carleman inequality and thus the spectral inequality.\\
\indent Our first goal is thus the derivation of the Carleman inequality for
the operator $Q=D^4_s+\Delta^2$ near the boundary under the boundary conditions given by $B_1$ and $B_2$.

Then, from the Carleman estimate one deduces an interpolation inequality
for the operator $Q$ in the case of the prescribed boundary
conditions.  The spectral inequality then follows from this interpolation inequality. 

Finally as an application of the spectral inequality one deduces a null-controllability result. The method we used is called the Lebeau-Robbiano method which is a quantitative unique continuation property for the sum of eigenfunctions of the operator $\Delta^2.$
Null-controllability of a fourth order parabolic equation in dimension greater or equal to two was studied in \cite{SGK, HY} under the boundary conditions $B_1 y_{|\partial \Omega}=y_{|\partial\Omega}$ and $B_2 y_{|\partial \Omega}=\Delta y_{|\partial\Omega},$ see also the references therein for some interesting results concerning the null-controllability of this kind of equations. In \cite{SGK}, the authors proved the null-controllability of system (\ref{eq: parabolic}) (with homogeneous Dirichlet boundary conditions on the solution and the laplacian of the solution), by establishing an observability inequality for the non-homogeneous adjoint system associated to (\ref{eq: parabolic}), obtained from a Carleman estimate. The null interior controllability for a fourth order parabolic equation under the boundary conditions $B_1 y_{|\partial \Omega}=y_{|\partial\Omega}$ and $B_2 y_{|\partial \Omega}=\Delta y_{|\partial\Omega},$ obtained in \cite{HY} is based on Lebeau-Robbiano method.

\subsection{On Carleman estimate for higher-order elliptic operators}
If $B$ is an elliptic operator of even order $k,$ and $\varphi$ is a Carleman weight function such that the couple $(B,\varphi)$ satisfies the so-called sub-ellipticity condition, then a Carleman estimate can be obtained, even at the boundary, for instance from the results of \cite{ML}.\\
If $k\geq 4, $ it is however quite natural not to have sub-ellipticity, in particular if $B$ is a product, say $B=B_1B_2.$ Denote by $b, b_1$ and $b_2$ the principal symbols of $B, B_1$ and $B_2$ respectively. The conjugated operator $B_\varphi=e^{\tau\varphi}Be^{-\tau\varphi}$ reads $B_\varphi=B_{1,\varphi}B_{2,\varphi}$ with $B_{\ell,\varphi}=e^{\tau\varphi}B_\ell e^{-\tau\varphi}$ for $\ell=1,2.$ If the semi-classical characteristic set $Char(B_{1,\varphi})\cap Char(B_{2,\varphi})\neq \emptyset,$ then the sub-ellipticity property fails to holds. We recall that the semi-classical characteristic set of a differential operator $B$ with principal symbol $b(\y)$ is defined as
$$Char(B)=\{~\y=(z,\zeta,\tau)~:~(\zeta,\tau)\neq (0,0)~;~b(\y)=0\}.$$
In fact, if $b_\varphi, b_{1,\varphi}$ and $b_{2,\varphi}$ are respectively the semi-classical symbols of $B_\varphi, B_{1,\varphi}$ and $B_{2,\varphi},$ that is,
$$b_\varphi=b(z,\zeta+i\tau d\varphi(z))~~\text{and}~~b_{\ell,\varphi}=b_\ell(z,\zeta+i\tau d\varphi(z)),~\ell=1,2,$$ we can write
$$\{\ovl b_\varphi, b_\varphi\}=\{\ovl{b_{1,\varphi}b_{2,\varphi}}, b_{1,\varphi}b_{2,\varphi}\}=|b_{1,\varphi}|^2\{\ovl b_{2,\varphi}, b_{2,\varphi} \}+|b_{2,\varphi}|^2 \{\ovl b_{1,\varphi}, b_{1,\varphi} \}+f|b_{1,\varphi}||b_{2,\varphi}|$$
for some function $f$ and where $\{\cdot,\cdot\}$ denotes the Poisson bracket.\\
\indent Hence $\{\ovl b_\varphi, b_\varphi\}$ vanishes if $b_{1,\varphi}=b_{2,\varphi}=0$ and thus the sub-ellipticity property cannot hold for $B.$ Observe that in this example, we have $d_{z,\zeta}b(z,\zeta+i\tau d\varphi(z))=0$ if $b_1(z,\zeta+i\tau d\varphi(z))=b_2(z,\zeta+i\tau d\varphi(z))=0.$ In such cases of symbol ``flatness", the Carleman estimate we can derive for $B$ exhibits at least a loss of one full derivative.
\subsection{Geometrical setting}
On $\Omega$ we consider a Riemannian metric $\mathsf{g}_x = (\mathsf{g}_{ij}(x))$, with
associated cometric $(\mathsf{g}^{ij}(x))  = (\mathsf{g}_x)^{-1}$.
It stands as a bilinear form that act on vector fields,
\begin{align*}
  \mathsf{g}_x (u_x, v_x) = \mathsf{g}_{ij}(x) u_x^i v_x^j,
  \quad u_x = u_x^i \partial_{x_i}, \
  v_x = v_x^i \partial_{x_i}.
\end{align*}

\medskip
For $x \in \partial \Omega$ we denote by $\nu_x$ the unit outward
pointing {\em normal vector} at $x$, unitary in the sense of the
metric $\mathsf{g}$, that is
\begin{align*}
  \mathsf{g}_x(\nu_x, \nu_x)=1
  \ \ \text{and} \ \
  \mathsf{g}_x(\nu_x, u_x)=0 \quad \forall u_x \in T_x \partial \Omega.
  \end{align*}
We denote by $\partial_\nu$ the associated derivative at the boundary,
that is, $\partial_\nu f (x) = \nu_x (f)$. 
We also denote by $n_x$ the unit outward
pointing {\em conormal vector} at $x$, that is,
$n_x = \nu_x^\flat$, that is, $(n_x)_i = \mathsf{g}_{i j} \nu_x^j$. 

Near a boundary point, we shall often use normal geodesic coordinates
where $\Omega$ is locally given by $\{ x_d >0\}$ and the  
metric $\mathsf{g}$ takes the form
\begin{align*}
  \mathsf{g}=dx^d\otimes dx^d
  +\sum\limits_{1\leq i,j\leq d-1} \mathsf{g}_{ij} dx^i\otimes  dx^j.
\end{align*}
Then, the vector field $\nu_x$ is locally given by $(0, \dots, 0,
-1)$. The same for the  one form $n_x$.  

Normal geodesic coordinates allow us to locally formulate 
boundary problems in a half-space  geometry. We write 
$$\Rdp:=\{ x\in\R^d,\ x_d>0\}\qquad \text{where}\ 
x=(x',x_d) \ \text{with}\ x'\in\R^{d-1}, x_d\in\R.$$ 
We shall naturally denote its closure by $\ovl{\Rdp}$, that is,
$\ovl{\Rdp} = \{ x\in \R^d; x_d\geq 0\}$.

The Laplace-Beltrami operator
is given by
\begin{align*}
\label{eq: laplace-Beltrami}
  (\Delta_\mathsf{g} f) (x) =(\det \mathsf{g}_x)^{-1/2} 
  \sum_{1\leq i, j\leq d} \partial_{x_i}\big(
  (\det \mathsf{g}_x)^{1/2} \mathsf{g}^{i j}(x)
  \partial_{x_j}  f \big)(x).
\end{align*}
in local coordinates. Its principal part is given by
$\sum\limits_{1\leq i, j\leq d} \mathsf{g}^{i j}(x) \partial_{x_i} \partial_{x_j}$ and
its principal symbol by $\sum\limits_{1\leq i, j\leq d} \mathsf{g}^{i j}(x) \xi_i
\xi_j$.
The bi-Laplace operator is $P= \Delta^2_\mathsf{g}$.\\ In the main text of the
article we shall write $\Delta$, $\Delta^2$ in place of $\Delta_\mathsf{g}$, $\Delta_\mathsf{g}^2$.
%
%
\subsection{Main results}
We state the main Carleman estimate for the operator $Q=D^4_s+\Delta^2$ in normal geodesic coordinates as presented in Section \ref{normal-geodesic}. We note that $V$ is an open bounded neighborhood of $z_0 = (s_0, x_0) \in \R^{N}$, with $N=d+1$.
\begin{theorem}
	\label{th: ineq Carleman}
Let $Q=\Delta^2+D_s^4$ and $(s^0,x^0) \in (0,S_0)\times\partial \Omega$, with
  $\Omega$ locally given by $\{ x_d>0\}$ and $S_0>0$. Assume that $(Q,B_1,B_2,\varphi)$ satisfies the \LS condition of Definition \ref{def: LS after conjugation} at $\y'=(s,x,\sigma,\xi',\tau,\gamma,\eps)$ for all $(\sigma,\xi',\tau,\gamma,\eps)\in\R\times\R^{d-1}\times (0,+\infty)\times [1,+\infty)\times(0,1].$
  Let $\varphi(z)=\varphi_{\gamma,\eps}(z)$ be define as in section \ref{SB}. There exists an open neighborhood $W$ of $z_0$ in $(0,S_0)\times\R^d$, $W\subset V$, and there exist $\tau_0\geq \tau_*,\gamma\geq 1,\eps_0\in (0,1]$ such that
$$\gamma\Normsc{\tilde{\tau}^{-1}e^{\tau\varphi}u}{4,0}+\normsc{\trace(e^{\tau\varphi}u)}{3,1/2} \lesssim \Norm{e^{\tau\varphi}Qu}{+}+\sum\limits_{j=1}^{2}
\normsc{e^{\tau\varphi}B_{j}u\br}{7/2-k_{j}}, $$
for $\tau\geq \tau_0,$ $\gamma\geq \gamma_0,$ $\eps\geq \eps_0,$ and  $u\in\Cbarc(W_+).$
\end{theorem}
\smallskip
The general boundary differential operators $B_1$ and $B_2$ are given by
\begin{equation*}
  B_\ell(x,D)
  = \sum\limits_{0\leq j\leq\min(3,k_\ell)}
  B_\ell^{k_\ell-j}(x,D') (i \partial_\nu) ^j,
  \qquad \ell=1,2,
\end{equation*}
with $B_\ell^{k_\ell-j}(x,D')$ differential operators acting in the
tangential variables. 
\noindent Let $(\Pell, D(\Pell))$ be the unbounded operator on $L^2(\Omega),$ with
\begin{align}
  D(\Pell) = \big\{ u \in H^4(\Omega); \ B_1
  u_{|\partial \Omega}
  = B_2 u_{|\partial \Omega}=0\big\},
\end{align}
and given by $\Pell u = \Delta^2 u$ for $u \in D(\Pell)$.\\
\\
\indent The following spectral inequality quantifies how linear combinations of the eigenfunctions of $\Pell$ can be observed from a subdomain.
\begin{theorem}\textup{(spectral inequality)} \label{spectineg}.   
\\
Let $P$ and $B_1, B_2$ be such that the \LS condition holds on $\partial\Omega$ and $(\Pell,D(\Pell))$ is self-adjoint. 
We assume furthermore that $Q=D_s^4+\Delta^2$ and $B_1, B_2$ satisfy the \LS on 
$ (0,S_0)\times \partial\Omega$.
Let $(\phi_j)_{j\in\N}$ be a Hilbert basis of $L^2(\Omega)$ made of eigenfunctions of $(\Pell,D(\Pell))$ associated with the sequence $\mu_0 \leq \mu_1 \leq \cdots \leq \mu_k \leq \cdots$ of eigenvalues.
Let $\omega$ be an open set in $\Omega$. There exists $K>0$ such that for all $\mu\geq\max\{\mu_0,0\}$ one has
\begin{equation}\label{SINE}
    \|y\|_{L^2(\Omega)}\leq K e^{K\mu^{1/4}}\|y\|_{L^2(\omega)},~~
y\in\Span\{\phi_j;~\mu_j\leq \mu\}.
\end{equation}
\end{theorem}
\begin{remark}
If $\mu_0\leq 0,$ we replace $P$ by $\mathsf P_{\lambda}=\Pell+\lambda\id$ with $\lambda>-\mu_0.$ Then, the eigenvalues of $\mathsf P_{\lambda}$ are simply $0<\mu_0+\lambda\leq \mu_1+\lambda\leq\dots .$ In addition, observe that if one derives a spectral inequality as (\ref{SINE}) for $\mathsf P_{\lambda}$ then one obtains also such an inequality for $\Pell,$ since $(\mu+\lambda)^{1/4}\lesssim\mu^{1/4}$ for $\mu>0.$ So, in the proof of the spectral inequality, we may therefore assume that $\mu_0>0.$
\end{remark}
\begin{remark}
In the statement of Theorem~\ref{spectineg} we assume \LS condition on $(Q,B_1,B_2)$ but we prove 
below in Proposition~\ref{prop: LS after conjugation} there exists $\varphi$ such that $(Q,B_1,B_2,\varphi)$ 
satisfies \LS condition. This allows to  apply Theorem~\ref{th: ineq Carleman}
\end{remark}
The proof of Theorem \ref{spectineg} is given in Section \ref{spectral}.
\subsection{Some perspectives}
In the light of the results obtained here, it would be of interest to consider more general polyharmonic operators on $\Omega$ with more general \LS conditions. More precisely, one can replace $\Delta^2$ by $(-\Delta)^k$ for $k\geq 3$ along with $k$ boundary conditions $B_1, B_2,\dots, B_k$ that satisfy the \LS boundary condition. It would also be very interesting to investigate the boundary null-controllability of system (\ref{eq: parabolic}) under a specify boundary conditions, let say in the case $B_1y=y$ and $B_2y=\Delta y$ or $\partial_ny$ on $\partial\Omega.$
\subsection{Some notation}

The canonical inner product in $\C^m$ is denoted by $(z,z')_{\C^m}=\sum\limits_{k=0}^{m-1}z_k
\bar{z'}_k,$ for $z=(z_0,\cdots,z_{m-1})\in\C^m,$
$z'=(z'_0,\cdots,z'_{m-1})\in\C^m.$ The associated norm will be
denoted $|z|^2_{\C^m}=\sum\limits_{k=0}^{m-1}|z_k|^2.$

We shall use the notations $a\lesssim b$ for $a\leq C b$ and
$a\gtrsim b$ for $a\geq C b$, with a constant $C>0$ that may change
from one line to another. We also write $a\asymp b$ to denote
$a\lesssim b\lesssim a.$

For an open set $U$ of $\R^N$ we set $U_+ = U \cap \R^N_+$ and 
\begin{equation}
   \label{eq: notation Cbarc}
 \Cbarc  (U_+)=\{ u=v_{|\R^N_+}; v\in\Cinfc(\R^N) \text{ and }\supp (v)\subset U\}.
\end{equation}

We set $\ovl{\mathscr S}(\R^N_+)=\{u_{|\R^N_+};\ u\in\scrS(\R^N)\}$ with
$\scrS(\R^N)$ the usual Schwartz space in $\R^N$:
\begin{align*}
  u \in \scrS(\R^N)
  \ \ \Leftrightarrow \ \ 
  u\in \Cinf(\R^N) \ \text{and} \ 
  \forall\alpha,\beta\in\N^N 
  \sup\limits_{x\in\R^N} |x^{\alpha}D_x^{\beta} u(x)|<\infty.
\end{align*}

\section{Tangential semi-classical calculus and associated Sobolev norm}\label{tangential-semi}
Dealing with boundary problems, we shall locally use coordinates so that the geometry is that of the half-space
$$\ovl{\RNp}=\{\theta\in\R^N~~;~\theta_N\geq 0\},~~\theta=(\theta',\theta_N)\in \R^{N-1}\times\R.$$ We shall use the notation $\upsilon=(\theta,\vartheta,\tau)$ and $\upsilon'=(\theta,\vartheta',\tau)$ in this section.\\
Let $a(\upsilon)\in\Cinf(\ovl{\RNp}\times\R^{N-1}),$ $\tau$ a parameter in $[1,+\infty)$ and $m\in\R$ be such that, for all multi-indices $\alpha,\beta,$ we have
$$|\partial_\theta^{\alpha}\partial^{\beta}_{\vartheta'}a(\upsilon')|\leq
  C_{\alpha,\beta}\lscttau^{m-|\beta|}, \  \
  \text{where} \ \ \lscttau^2=\tau^2+|\vartheta'|^2,~~\theta\in\ovl{\RNp},~\vartheta'\in\R^{N-1}.$$ We then write $a\in\Sscttau^m.$ We also define $\Sscttau^{-\infty}=\bigcap_{r\in\R}\Sscttau^r.$ For $a\in\Sscttau^m$ we define the principal symbol $\sigma(a)$ to be the equivalent class of $a$ in $\Sscttau^m/\Sscttau^{m-1}.$ We recall that $\lscttau^m\in\Sscttau^m.$\\ If $a(\upsilon')\in\Sscttau^m,$ we set
  $$\Opt (a)u(\theta)=(2\pi)^{-(N-1)}\int_{\R^{N-1}}e^{i\theta'\cdot \vartheta'}a(\upsilon')\hat{u}(\vartheta',\theta_N)d\vartheta',~~~~~~u\in\SbarpN,$$
  where $\hat{u}$ is the partial Fourier transform of $u$ with respect to the tangential variables $\theta'.$ We denote by $\Psiscttau^m$ the set of those pseudo-differential operators. We also set $\Lscttau=\Opt(\lscttau^m)$ for $m\in\R.$\\
  Let $m\in\N$ and $m'\in\R.$ If we consider $a$ of the form
$$a(\upsilon)=\sum\limits_{j=0}^{m}a_j(\upsilon')\vartheta_N^j,
  \quad a_j\in \Sscttau^{m+m'-j},$$ we define
  $\Op (a)=\sum\limits_{j=0}^{m}\Opt (a_j)D^j_{\theta_N}.$ We then
  write $a\in \Ssctau^{m,m'}$ and $\Op (a)\in \Psisctau^{m,m'}.$\\
  We define the following norm for $m\in\N$ and $m'\in\R:$
  \begin{align*}
  & \Normsctau{u}{m,m'}
    \asymp \sum\limits_{j=0}^{m}
    \Norm{\Lscttau^{m+m'-j} D_{\theta_N}^j u}{+},\\
  & \Normsctau{u}{m}
    =\Normsctau{u}{m,0}
    \asymp\sum\limits_{j=0}^{m}
    \Norm{\Lscttau^{m-j} D_{\theta_N}^j u}{+},~~\qquad u\in\SbarpN
\end{align*}
where $\Norm{\cdot}{+}:=\|\cdot\|_{L^2(\RNp)}.$ We have
$$\Normsctau{u}{m}\asymp\sum_{\substack{|\alpha|\leq m\\ \alpha\in\N^N}}\tau^{m-|\alpha|}\Norm{ D^\alpha u}{+},$$ and for $m'\in\N$ we have 
$$\Normsctau{u}{m,m'}\asymp\sum_{\alpha_N\leq m}\sum_{\substack{|\alpha|\leq m+m'\\ \alpha=(\alpha',\alpha_N)\in\N^N}}\tau^{m+m'-|\alpha|}\Norm{ D^\alpha u}{+}.$$ 
At the boundary $\{\theta_N=0\}$ we define the following norms, for $m\in\N$ and $m'\in\R$,
\begin{align*}
  \normsctau{\trace(u)}{m,m'}^2
  =\sum\limits_{j=0}^{m}
  \norm{\Lscttau^{m+m'-j} D^j_{\theta_N}u_{{|\theta_N=0^+}}}{L^2(\R^{N-1})}^2,
  \qquad u\in\SbarpN.
\end{align*}
\subsection{Semi-classical calculus with three parameters}\label{SB}
We set $N=d+1$ and $\mathcal{W}=\R^{N}\times \R^{N},$ often referred as \emph{phase space}. A typical element of $\mathcal{W}$ will be $X=(s,x,\sigma,\xi),$ with $s\in\R, x\in\R^d, \sigma\in \R, \xi\in\R^d.$ We also write $x=(x',x_d), x'\in\R^{d-1}, x_d\in\R$ and $\xi=(\xi',\xi_d)\in\R^{d-1}\times\R.$ With $s$ and $x$ playing a very similar role in the definition of the calculus, we set $z=(s,x)\in\R^{N}, z'=(s,x')\in\R^d,$ and $z_N=x_d.$ We also set $\zeta=(\sigma,\xi)\in\R^{N},$ $\zeta'=(\sigma,\xi')\in\R^{N-1},$ and $\zeta_N=\xi_d.$ We shall consider a weight function of the form
\begin{equation}\label{weight-function}
\varphi_{\gamma,\eps}(z)=e^{\gamma\psi_{\eps}(z)},~~~~\psi_{\eps}(z)~=\psi(\eps z', z_N),
\end{equation}
with $\gamma$ and $\eps$ as  parameters satisfying $\gamma\geq 1,$ $\eps\in [0,1],$ and $\psi\in\mathscr{C}^{\infty}(\R^{N}).$ To define a proper pseudo-differential calculus, we assume the following properties for $\psi:$
\begin{equation}\label{prop1-weight-function}
\psi\geq C>0,~~~\|\psi^{(k)}\|_{L^{\infty}}<\infty,~~k\in\N.
\end{equation}
In particular, there exists $k>0$ such that
\begin{equation}\label{prop2-weight-function}
  \sup\limits_{\R^{N}}\psi\leq (k+1)\inf\limits_{\R^{N}}\psi.  
\end{equation}
\subsubsection{A class of semi-classical symbols}
We introduce the following class of tangential symbols depending on $z\in\R^{ N},$ $\zeta'\in\R^{d}$ and $\hat{t}\in\R^{N}.$ We set $\hat{\lambda}_{\mathsf{T}}=|\zeta'|^2+|\hat{t}|^2.$
\begin{definition}
    Let $m\in\R.$ We say that $a(z,\zeta',\hat{t})\in\mathscr{C}^{\infty}(\overline{\R^{N}_+}\times\R^{N-1}\times\R^{d+1})$ belongs to the class $S^m_{\mathsf{T},\hat{t}}$ if for all multi-indices $\alpha\in \N^{N}, \beta\in \N^{N-1},$ there exists $C_{\alpha,\beta,\delta}>0$ such that
    $$|\partial_z^{\alpha}\partial^{\beta'}_{\zeta'}\partial^{\delta}_{\hat{t}}a(z,\zeta',\hat{t})|\leq
C_{\alpha,\beta,\delta}\hat{\lambda}_{\mathsf{T}}^{m-|\beta|-|\delta|},~~~~~~(z,\zeta',\hat{t})\in\overline{\R^{N}_+}\times\R^{N-1}\times\R^{N},~~|\hat{t}|\geq 1.$$
\end{definition}
If $\U$ is a conic open set in $\overline{\R^{N}_+}\times\R^{N-1}\times \R^{N},$ we say that $a\in S^m_{\mathsf{T},\hat{t}}$ in $\U$ if the above property holds for $(z,\zeta',\hat{t})\in\U.$ \\
As opposed to the usual semi-classical symbols, some regularity is required with respect to the semi-classical parameter that is a vector of $\R^{N}.$\\
\indent The above class of symbols will not be used as such to define a class of pseudo-differential operators but rather to generate other classes of symbols and associated operators in a more refined semi-classical culculus that we present below.
\subsubsection{Metrics}

For $\tau_0\geq 2.$ we set $$\mathcal{M}=\R^{N}\times \R^{N}\times [\tau_0,+\infty)\times[1,+\infty)\times[0,1] $$
$$\mathcal{M}_{\mathsf{T}}=\ovl{\R^{N}_+}\times \R^{d}\times [\tau_0,+\infty)\times[1,+\infty)\times[0,1] .$$ We denote by $\y=(z,\zeta, \tau,\gamma,\eps)$ a point in $\mathcal{M}$ and $\y'=(z,\zeta', \tau,\gamma,\eps)$ a point in $\mathcal{M}_{\mathsf{T}}.$ We recall that $z=(s,x),$ $\zeta=(\sigma,\xi)$ and $\zeta'=(\sigma,\xi').$\\
\indent We set $\tilde{\tau}=\tau\gamma\varphi_{\gamma,\eps}(z)\in\R_+.$ For simplicity, even though $\tilde{\tau}$ is independent of $\zeta',$ we shall write
$\tilde{\tau}=\tilde{\tau}(\y')$ when we wish to keep in mind that $\tilde{\tau}$ is not just a parameter but rather a function. As $\psi>0,$ $\tau>\tau_0,$ and $\gamma\geq 1,$ we note that $\tilde{\tau}\geq\tau_0.$ We then set
$$\lambda^2_{\tilde{\tau}}=\lambda^2_{\tilde{\tau}}(\y')=|\zeta|^2+\tilde{\tau}(\y')^2,~~~~~~\lsct^2=\lsct^2(\y')=|\zeta'|^2+\tilde{\tau}(\y')^2.$$ To simplify notation we dropped the explicit dependence of $\lsc$ and $\lsct$ upon $\y$ and $\y'.$ Similarly, we shall write $\varphi(z)$ or simply $\varphi$ in place of $\varphi_{\gamma,\eps}(z).$\\
We consider the following metric on the phase space $\mathcal W=\R^N\times\R^N:$
\begin{equation}\label{meticg}
g=(1+\gamma\eps)^2|dz'|^2+\gamma^2|dz_N|^2+\lsc^{-2}|d\zeta|^2
\end{equation}
for $\tau\geq \tau_0,\gamma\geq 1,$ and $\eps\in [0,1]$ (We remind that this metric is not to be confused with the Riemannian metric $\mathsf{g}$ on $\Omega$).\\
On the phase space $\mathcal W=\R^N\times\R^{N-1}$ adapted to the tangential calculus, we consider the metric
\begin{equation}\label{meticgt}
g_{\mathsf T}=(1+\gamma\eps)^2|dz'|^2+\gamma^2|dz_N|^2+\lsct^{-2}|d\zeta'|^2
\end{equation}
for $\tau\geq \tau_0,\gamma\geq 1,$ and $\eps\in [0,1]$.\\
As shown in \cite[Section 4.1.2]{JL}, the metric $g$ on $\mathcal{W}$ defines a Weyl-H\"ormander pseudo-differential calculus, and both $\varphi$ and $\lsc$ have the properties required of proper order functions. See for instance \cite[Sections 18.4-6]{LarsH1} for a presentation of the Weyl-H\"ormander calculus.
\\
\indent In \cite{JL}, the proof of the uncertainty principle uses the fact that $\tau_0\geq 2.$ The condition $\tau_0\geq 1$ would suffice if one choses $\psi\geq \ln(2).$ \\
Consequently, $\tilde{\tau}(\y')$ is also an admissible order function for both calculi.
\\
The proofs of the following two propositions can be found in \cite[Appendix A.2.1]{JL}.
\begin{proposition}
The metric $g$ and the order functions $\varphi_{\gamma,\eps}, \lsc$ are admissible, in the
sense that the following properties hold (uniformly with respect to the parameters $\tau,\gamma$ and $\eps$).
\begin{enumerate}
    \item $g$ satisfies the uncertainty principle, that is, $h^{-1}_g=\gamma^{-1}\lsc\geq 1,$
    \item $\varphi_{\gamma,\eps}, \lsc$ and $g$ are slowly varying,
\item $\varphi_{\gamma,\eps}, \lsc$ and $g$ are tempered.
\end{enumerate}
\end{proposition}
Similarly, we have the following proposition.
\begin{proposition}
The metric $g_{\mathsf T}$ and the order functions $\varphi_{\gamma,\eps}, \lsct$ are admissible. For
tangential calculus we have
$h^{-1}_{g_\mathsf T}=(1+\eps\gamma)^{-1}\lsct\geq 1.$
\end{proposition}
\subsubsection{Symbols}

Let $a(\y)\in \mathscr C^{\infty}(\R^N\times\R^N)$, and let $\tau,\gamma$ and $\eps$ act as parameters, and let $m,r\in\R$ be such that for all multi-indices $\alpha,\beta\in\N^N$ with $\alpha=(\alpha',\alpha_N)$ we have 
$$|\partial^{\alpha}_z\partial^{\beta}_\zeta a(\y)|\leq C_{\alpha,\beta}\gamma^{|\alpha_N|}(1+\eps\gamma)^{|\alpha'|}\tilde{\tau}^r\lsc^{m-|\beta|},~~~~\y'\in\mathcal{M}. $$
With the notation of \cite[Sections 18.4-18.6]{LarsH1} we then have $a(\y)\in S(\tilde{\tau}^r\lsc^m,g).$
Similarly, let $a(\y')\in\mathscr C^{\infty}(\ovl{\R^N_+}\times\R^{N-1}),$ let $\tau,\gamma,$ and $\eps$ as parameters, and let $m\in\R.$ If for all multi-indices $\alpha=(\alpha',\alpha_N)\in\N^N,~~\beta'\in\N^{d-1},$ we have
$$|\partial^{\alpha}_z\partial^{\beta'}_{\zeta'} a(\y')|\leq C_{\alpha,\beta}\gamma^{|\alpha_N|}(1+\eps\gamma)^{|\alpha'|}\tilde{\tau}^r\lsct^{m-|\beta'|},~~~~\y'\in\mathcal{M}_{\mathsf T}, $$
we write $a(\y')\in S(\tilde{\tau}^r\lsct^m,g_\mathsf T).$\\
The principal symbol associated with $a(\y')\in S(\tilde{\tau}^r\lsct^m,g_\mathsf T)$ is given by the equivalence class of $a(\y')$ in $S(\tilde{\tau}^r\lsct^m,g_\mathsf T)/S((1+\eps\gamma)\tilde{\tau}^r\lsct^{m-1},g_\mathsf T).$\\
 We define the following class of symbols that are polynomial with respect to $\zeta_N=\xi_d:$
$$\Ssc^{m,m'}=\sum\limits_{j=0}^mS(\lsct^{m+m'-j},g_\mathsf T)\zeta^j_N.$$ For $a(\y)\in \Ssc^{m,m'}$ with $a(\y)=\sum\limits_{j=0}^ma_j(\y')\zeta^j_N$ and $a_j(\y')\in S(\lsct^{m+m'-j},g_\mathsf T),$ we denote the principal part of $a(\y)$ by $\sigma(a)(\y)=\sum\limits_{j=0}^m\sigma(a_j)(\y')\zeta^j_N.$ For this calculus with three parameters to make sense, it is important to check that $\lsc\in S(\lsc,g), \lsct\in S(\lsct,g_\mathsf T)$ and $\tilde{\tau}\in S(\tilde{\tau},g)\cap S(\tilde{\tau},g_\mathsf T).$
In fact the have the following property which implies the first two.
\begin{lemma}
We have $\tilde{\tau}=\tau\gamma\varphi_{\gamma,\eps}\in S(\tilde{\tau},g)\cap S(\tilde{\tau},g_\mathsf T).$
\end{lemma}
\begin{proof}
We first remark that only differentiation with respect to $z=(s,x)$ of $\tilde{\tau}$ have to be considered since $d_\zeta\tilde{\tau}=0.$ Then, for any $\alpha=(\alpha',\alpha_N)\in\N^N,$ we can write $\partial^{\alpha}_z\tilde{\tau}(\y')$ as a linear combination of terms of the form
$$\partial^{\alpha}_z\tilde{\tau}(\y')=\tau\gamma^{1+n}\varphi_{\gamma,\eps}\prod\limits_{j=0}^n\partial^{\alpha^{(j)}}_z\psi_\eps(z)=\tau\gamma^{1+n}\eps^{|\alpha'|}\varphi_{\gamma,\eps}\prod\limits_{j=0}^n\partial^{\alpha^{(j)}}_z\psi(\eps z',z_N),$$ with $\sum\limits_{j=1}^n\alpha^{(j)}=\alpha,$ $|\alpha^{(j)}|\geq 1,$ $j=1,\dots, n$ and $n\leq |\alpha|.$ As, $\gamma\geq 1,$ this implies 
$$|\partial^{\alpha}_z\tilde{\tau}(\y')|\leq \tilde{\tau}(\y')\gamma^{\alpha_N}(\eps\gamma)^{|\alpha'|},$$ since $\|\psi\|_{L^{\infty}}\leq C$ for any $k\in\N,$ which ends the proof. Observing that $\tilde{\tau}=\gamma\tau\varphi_{\gamma,\eps}\lesssim\lsc$ (resp. $\tilde{\tau}=\gamma\tau\varphi_{\gamma,\eps}\lesssim\lsct$), we conclude that  $\lsc\in S(\lsc,g), \lsct\in S(\lsct,g_\mathsf T).$
\end{proof}
\subsubsection{Operators and Sobolev bounds} For $a\in S(\tilde{\tau}^{r}\lambda^m_{\tilde{\tau}},g)$ we define the following pseudo-differential operator in $\R^{N}:$
\begin{equation}
\Op (a)u(z)=(2\pi)^{-N}\int_{\R^{N}}e^{iz\cdot \zeta}a(z,\zeta,\tau,\gamma,\eps)\hat{u}(\zeta)d\zeta,~~~~~~u\in\SbarpN,
\end{equation}
where $\hat{u}$ is the Fourier transform of $u$. The associated class of pseudo-differential operators is denoted by $\Psi(\tilde{\tau}^{r}\lambda^m_{\tilde{\tau}},g).$ If $a$ is a polynomial in $\zeta$ and $\hat{\tau}(\y')=\tilde{\tau}d_z\psi_{\eps}(z),$ we write $\Op(a)\in\mathscr{D}(\tilde{\tau}^{r}\lambda^m_{\tilde{\tau}},g).$\\
\indent Similarly we define the tangential operators. For $a\in S(\tilde{\tau}^{r}\lsct^m,g_{\mathsf{T}})$ we set 
\begin{equation}
\Op (a)u(z)=(2\pi)^{-(N-1)}\iint_{\R^{2(N-1)}}e^{i(z'-y')\cdot \zeta'}a(z,\zeta',\tau,\gamma,\eps)u(y',z_d)d\zeta'dy',~~~~~~u\in\SbarpN, 
\end{equation}
and $z\in \R^{N}.$ We write $A=\Opt(a)\in\Psi(\tilde{\tau}^{r}\lsct^m,g_{\mathsf{T}})$ and $\Lsct^m=\Opt(\lsct^m).$\\
\indent We also introduce the following class of operators that act as differential operators in the $z_N=x_d$ variable and as tangential pseudo-differential operators in the $z'=(s,x')$ variables:
$$\Psisc^{m,r}=\sum\limits_{j=0}^m\Psi(\lsct^{m+r-j},g_\mathsf T)D^j_{z_N},~~~~m\in\N,~~~r\in\R,$$
that is, $\Op(a)\in\Psisc^{m,r}$ if $a\in \Ssc^{m,r}.$ Operators of this class can be applied to functions that are only defined on the half-space $\ovl{\R^N_+}.$\\

\indent For functional norms we use the notation $\|.\|$ for functions
defined in the interior of the domain and $|.|$ for functions defined on the boundary.
In that spirit, we shall use the notation
\begin{align*}
  \Norm{u}{+}=\Norm{u}{L^2(\R^N_+)}, \quad
  \inp{u}{\tilde{u}}_+=\inp{u}{\tilde{u}}_{L^2(\R^N_+)},
\end{align*}
for functions defined in $\R^N_+$ and
\begin{align*}
  \norm{w}{\partial}=\Norm{w}{L^2(\R^{N-1})}, \quad
  \inp{w}{\tilde{w}}_{\partial}=\inp{w}{\tilde{w}}_{L^2(\R^{N-1})},
\end{align*}
for functions defined on $\{ x_d =0\}$, such as traces. We define the following semi-classical Sobolev norms (here with $\tilde{\tau})$:
\begin{align*}
& |u|_{m,\tilde{\tau}}=|\Lsct^m u_{|z_d=0^+}|_{L^2(\R^{N-1})}, ~~m\in\R, u\in\SbarpN,\\
  & \Normsc{u}{m}
    =\Normsc{u}{m,0}
    \asymp\sum\limits_{j=0}^{m}
    \Norm{\Lsct^{m-j} D_{x_d}^j u}{+},~~~~m\in\N, u\in \SbarpN.
\end{align*}
We also set, for $m\in\N$ and $m'\in\R$,
$$\Normsc{u}{m,m'}
    \asymp \sum\limits_{j=0}^{m}
    \Norm{\Lsct^{m+m'-j} D_{x_d}^j u}{+},~~~ u\in \SbarpN.$$
At the boundary $\{x_N=0\}$ we define the following norms, for $m\in\N$ and $m'\in\R$,
\begin{align*}
  \normsc{\trace(u)}{m,m'}^2
  =\sum\limits_{j=0}^{m}
  \norm{\Lsct^{m+m'-j} D^j_{x_d}u\br}{\partial}^2,
  \qquad u\in\SbarpN.
\end{align*}

\subsection{Boundary  differential quadratic forms}
\label{sec: Boundary  differential quadratic forms}
A differential quadradic form acts on a function and involves
differentiations of various degrees of the function. One can associate
to these forms a symbol and positivity inequality that can be obtained in
the form of a G{\aa}rding inequality. 
Such forms appear in proofs of Carleman estimates in the pioneering work
of H\"ormander \cite[Section 8.2]{Hoermander:63}. 

Here differential quadradic forms are
defined at the boundary. The result we
present here without proof can be found in \cite{ML} and \cite{JGL-vol2}. 

\begin{definition}\label{conic}
A set $\scrO\subset\R^N\setminus\{0\}$ is called a conic set if, together with any point $\xi,$ it contains all the points $\lambda\xi$ where $\lambda>0.$ By a conic neighborhood of a point $\xi\in\R^N\setminus\{0\},$ we mean an open conic set that contains $\xi.$ 
\end{definition}
\begin{definition}
\label{def: boundary quadratic form}
Let $u\in  \SbarpN$. We say that
\begin{equation*}
  \Bquad(u)=\sum_{k=1}^n\inp{A^k u\br}{B^k u\br}_\partial, 
  \ \  A^k = a^k(s,x,D,\tau,\gamma,\eps),  \ B^k = b^k(s,x,D,\tau,\gamma,\eps),  
\end{equation*}
is a boundary differential quadratic form of type $(m-1,r)$ with $\Cinf$
coefficients, if for each $k=1,\dots n$, we have $a^k(\y)\in
\Ssc^{m-1,r'}(\ovl{\R^N_+} \times\R^N)$, $b^k(\y)\in
\Ssc^{m-1,r''}(\ovl{\R^N_+}  \times\R^N)$ with $r'+r''=2
r$, $\y=(\y',\zeta_N)$ with $\y'=(s,x,\xi',\tau,\gamma,\eps)$. The symbol of the
boundary differential quadratic form $\Bquad$ is defined by
\begin{equation*}
  \bquad(\y',\xi_N,\tilde{\zeta}_N)=\sum_{k=1}^n
  a^{k}(\y',\zeta_N)\ovl{b^k}(\y',\tilde{\zeta}_N). 
\end{equation*}
\end{definition}

For $\z=(z_0,\dots,z_{\ell-1})\in\C^\ell$ and $a(\y)\in
\Ssc^{\ell-1,t}$, of the form 
$a(\y',\xi_N) = \sum_{j=0}^{\ell-1} a_j(\y') \zeta_N^j$ with $a_j(\y') \in S(\lsct^{\ell-1+t-j},g_\mathsf T)$
we set
\begin{equation}
  \label{eq: symbol z}
  \un{a}(\y',\z)=\sum_{j=0}^{\ell-1}a_j(\y')z_j.
\end{equation}
From the boundary differential quadratic form $\Bquad$ we introduce the  following
bilinear symbol $\un{\Bquad}:\C^m\times\C^m\to \C $
\begin{equation}
  \label{eq: bilinear symbol-boundary quadratic form}
  \un{\Bquad}(\y',\z,\z')
  =\sum_{k=1}^n \un{a^k}(\y',\z)\ovl{\un{b^k}}(\y',\ovl{\z}'),
  \quad \z,\z'\in\C^m.
\end{equation}

We let $\mathscr W$ be an open conic set in $\R\times \R^{N-2} \times \R^{N-2} \times \R_+\times [1,+\infty)\times\in [0,1]$.
\begin{definition}
  Let $\Bquad$ be a boundary differential quadratic form of type $(m-1,r)$ with
  homogeneous principal symbol and 
  associated with the bilinear symbol $\un{\Bquad}(\y',\z,\z')$. We say that $\Bquad$ is
  positive definite in $\mathscr W$ if there exist $C>0$ and $R>0$ such that 
\begin{equation*}
  \Re \un{\Bquad}(\y'', x_N = 0^+, \z,\z)\geq
  C\sum_{j=0}^{m-1}  \lsct^{2(m-1-j+r)}|z_j|^2,
\end{equation*}
for $\y''=(s,x',\xi',\tau,\gamma,\eps) \in \mathscr W$, with $|(\xi', \tilde{\tau})| \geq
R$, and
$\z=(z_0,\dots,z_{m-1})\in \C^m$.
\end{definition}
We have the following G{\aa}rding inequality.
\begin{proposition}
  \label{prop: boundary form -Gaarding tangentiel}
  Let $\Bquad$ be a boundary differential quadratic form of type $(m-1,r)$,
  positive definite in $\mathscr W$, an open conic set in $\R\times \R^{N-2}
  \times \R^{N-2} \times [\tau_0,+\infty)\times [1,+\infty)\times\in [0,1]$, with bilinear symbol
  $\un{\Bquad}(\y',\z,\z')$. Let $\chi \in S(1,g_{\mathsf T})$ be homogeneous of
  degree 0, with $\supp(\chi\br)\subset \mathscr W$ and let $M\in \N$. Then
  there exist $\tau_0>0$, $\gamma\geq 1,$ $C>0$, $C_M >0$ such that 
  \begin{equation*}
    \Re \Bquad(\Opt(\chi) u)
    \geq C\normsc{\trace(\Opt(\chi) u)}{m-1,r}^2
    -C_M\normsc{\trace(u)}{m-1,-M}^2,
  \end{equation*}
  for $u\in \ovl{\mathscr{S}}(\R^{N}_+)$ and $\gamma\geq\gamma_0,$ $\tau\geq\tau_0,~\eps\in [0,1]$.
\end{proposition}
The following lemma which can be found in \cite[Lemma B.1]{JL}, gives a perfect elliptic estimate.\\
Let $(s^0,x^0)\in (0,S_0)\times\partial\Omega$ and $V$ denotes the neighborhood of $(s^0,x^0)$ in $\R^{N}.$ We set $\mathscr{O}_{\mathsf{T},V}=V\times\R^{N-1}\times [\tau_0,+\infty)\times [1,+\infty)\times [0,1].$\\
\indent Let $\ell(\y)\in S^{m,0}_{\tilde{\tau}}$ with $\y=(s,x, \xi, \tau, \gamma,\eps)$ and $m\geq 1$ be a polynomial in $\xi_N$ with homogeneous coefficients in $(\sigma,\xi', \hat{\tau})$ where $\hat{\tau}=\tilde{\tau}\gamma d_z\psi(s,x)\in\R^{N}$ and let $L=\ell(z,D_z, \tau,\gamma,\eps).$
\begin{lemma}\label{el1}
Let $\U$ be a conic open subset of $\mathscr{O}_{\mathsf{T},V}$. Assume that, for $\ell(\y',\zeta_N)$ viewed as a polynomial in $\zeta_N,$ for $\y'\in\U$
\begin{enumerate}
    \item[$\bullet$] the leading coefficient is $1;$
\item[$\bullet$] all roots of $\ell(\y',\zeta_N)=0$ have negative imaginary part.
\end{enumerate}
Let $\chi(\y')\in S(1,g_{\mathsf{T}})$ be homogeneous of degree zero and such that $\supp(\chi)\subset\U$. Then, for any $M\in\N,$ there exist $C>0,\tau_0>0,\gamma_0\geq 1$ such that
\begin{equation*}
  \Normsc{\Opt(\chi)w}{m,0}
  + \normsc{\trace(\Opt(\chi)w)}{m-1, 1/2}
  \leq C\left( 
  \Norm{L \Opt(\chi)w}{+}
  + \Normsc{w}{m,-M}\right),
\end{equation*}
for $w\in\SbarpN$ and $\tau\geq \tau_0,$ $\gamma\geq \gamma_0,$ $\eps\in [0,1].$ 
\end{lemma}
\section{The Lopatinski\u{\i}-\v{S}apiro condition}
\subsection{The \LS condition for the bi-Laplace operator}\label{LS:definition}
Let $P=\Delta^2$ with principal symbol $p(x,\omega)$ for
$(x,\omega)\in T^*\Omega$. One defines the following polynomial in
$\omega_d,$
$$p(x,\omega',\omega_d)=p(x,\xi'-\omega_d n_x),$$ for $x\in\partial\Omega,$ $\omega'\in T^*_x\partial\Omega,$ $\omega_d\in\R$ and  where $n_x$ denotes the outward pointing conormal vector at $x$, unitary. Here $x$ and $\omega'$ act as parameters. We denote by $\rho_j(x,\omega'),$ $1\leq j\leq 4$ the complex roots of $p.$ One sets $$p^+(x,\omega',\omega_d)=\prod\limits_{\Im\rho_j(x,\omega')\geq 0}(\omega_d-\rho_j(x,\omega')).$$ Given the boundary differential operators $B_1$ and $B_2$ in a neighborhood of $\partial\Omega,$ with principal symbols $b_j(x,\omega)$, $j=1,2,$ one also sets $b_j(x,\omega',\omega_d)=b_j(x,\omega'-\omega_dn_x).$
\begin{definition}[Lopatinski\u{\i}-\v{S}apiro boundary condition]\label{def: LS}
Let $(x,\omega')\in T^*\partial\Omega$ with $\omega'\neq 0.$ One says that the Lopatinski\u{\i}-\v{S}apiro condition holds for $(P, B_1, B_2)$ at $(x,\omega')$ if for any polynomial function $f(\omega_d)$ with complex coefficients, there exists $c_1,c_2\in\C$ and a polynomial function $h(\omega_d)$ with complex coefficients such that, for all $\omega_d\in\C,$
$$f(\omega_d)=\sum\limits_{1\leq j\leq 2}c_jb_j(x,\omega',\omega_d)+h(\omega_d)p^+(x,\omega',\omega_d).$$
We say that the Lopatinski\u{\i}-\v{S}apiro condition holds for $(P, B_1, B_2)$ at $x\in\partial\Omega$ if it holds at $(x,\omega')$ for all $\omega'\in T^*_x\partial\Omega$ with $\omega'\neq 0.$
\end{definition}
The general boundary differential operators $B_1$ and $B_2$ are then given by
\begin{equation*}
  B_\ell(x,D)
  = \sum\limits_{0\leq j\leq\min(3,k_\ell)}
  B_\ell^{k_\ell-j}(x,D') (i \partial_\nu) ^j,
  \qquad \ell=1,2,
\end{equation*}
with $B_\ell^{k_\ell-j}(x,D')$ differential operators acting in the
tangential variables.  
We denote by $b_1(x,\omega)$ and $b_2(x,\omega)$ the principal symbols of $B_1$ and $B_2$ respectively. For $(x,\omega')\in T^*\partial\Omega,$ we set 
\begin{equation}\label{symbol-au-bord}
  b_\ell(x,\omega',\omega_d)
  =\sum\limits_{0\leq j\leq \min(3,k_\ell)}
  b_\ell^{k_\ell-j}(x,\omega')\omega_d^j,
  \qquad \ell=1,2.
\end{equation}
We recall that the principal symbol of $P$ is given by
$p(x,\omega)=|\omega|^4_\mathsf{g}.$ One thus has 
\begin{equation*}
  p(x,\omega',\omega_d)=p(x,\omega'-\omega_dn_x) =  \big( \omega_d^2 + |\omega'|_\mathsf{g}^2\big)^2.
\end{equation*}
Therefore
$p(x,\omega',\omega_d)=(\omega_d-i|\omega'|_\mathsf{g})^2(\omega_d+i|\omega'|_\mathsf{g})^2$. According
to the above definition we set 
$p^+(x,\omega',\omega_d)=(\omega_d-i|\omega'|_\mathsf{g})^2.$ Thus, the
\LS condition holds at $(x,\omega')$ with $\omega'\neq 0$ if
and only if for any function $f(\omega_d)$ the complex number $i|\omega'|_\mathsf{g}$
is a root of the polynomial function $\omega_d\mapsto
f(\omega_d)-c_1b_1(x,\omega',\omega_d)-c_2b_2(x,\omega',\omega_d)$ and its
derivative for some $c_1,c_2\in\C$. This leads to the
following determinant condition
\begin{lemma}\label{lemma: Lopatinskii bi-laplacian op}  
  Let $P=\Delta^2$ on $\Omega$, $B_1$ and $B_2$ be two boundary differential
  operators. If $x \in \partial\Omega$, $\omega' \in T_x^* \partial\Omega$, with
  $\omega' \neq 0$, the \LS
   condition holds at $(x,\omega')$ if and only if
  \begin{equation}\label{cond1}
    \det
    \begin{pmatrix}
      b_1 &b_2 \\[2pt]
      \partial_{\omega_d} b_1& \partial_{\omega_d} b_2
    \end{pmatrix} (x,\omega',\omega_d=i|\omega'|_\mathsf{g})
    \neq 0.
\end{equation}
\end{lemma}
\begin{remark}\label{RR0}
With the determinant condition and homogeneity, we note that if the
\LS condition holds for $(P,B_1, B_2)$ at $(x,\omega')$ it also holds
in a conic neighborhood of $(x,\omega')$ by continuity. If it holds at $x\in\Omega,$ it also holds in a neighborhood of $x$.
\end{remark}
\subsection{Formulation in normal geodesic coordinates}\label{normal-geodesic}
Near a boundary point $x\in\partial\Omega$, we shall use normal
geodesic coordinates. Then the principal symbols of $\Delta$ and
$\Delta^2$ are given by $\xi_d^2+r(x,\xi')$ and
$(\xi_d^2+r(x,\xi'))^2$ respectively, where $r(x,\xi')$ is the
principal symbol of a tangential differential elliptic operator
$R(x,D')$ of order 2, with
$$r(x,\xi')=\sum\limits_{1\leq i,j\leq
  d-1}\mathsf{g}^{ij}(x)\xi'_i\xi'_j \ \ \text{and} \ \  r(x,\xi')\geq C|\xi'|^2.$$
Here $\mathsf{g}^{ij}$ is the inverse of the metric $\mathsf{g}_{ij}.$ Below, we shall
often write $|\xi'|^2_x=r(x,\xi')$ and we shall also write
$|\xi|^2_x=\xi_d^2+r(x,\xi'),$ for $\xi=(\xi',\xi_d).$

If $b_1(x,\xi)$ and $b_2(x,\xi)$ are the principal symbols of the
boundary operators $B_1$ and $B_2$ in the normal geodesic
coordinates then the \LS condition for $(P, B_1, B_2)$ with $P =
\Delta^2$ at $(x,\xi')$ reads
\begin{equation}\label{LS-bi}
    \det
    \begin{pmatrix}
      b_1&b_2\\[2pt]
      \partial_{\xi_d}b_1& \partial_{\xi_d}b_2
    \end{pmatrix}(x,\xi',\xi_d=i |\xi'|_x)
    \neq 0,
  \end{equation}
  if $|\xi'|_x \neq 0$ according to Lemma~\ref{lemma: Lopatinskii bi-laplacian op}. If the \LS condition holds at some $x^0$,
  because of homogeneity, there exists $C_0>0$ such that
  \begin{equation}\label{eq: homog formul LS condition}
    \left|\det
    \begin{pmatrix}
      b_1 &b_2\\[2pt]
      \partial_{\xi_d}b_1& \partial_{\xi_d}b_2 
    \end{pmatrix}\right|(x^0,\xi',\xi_d=i |\xi'|_x)
    \geq C_0 |\xi'|_x^{k_1+k_2 -1},
    \qquad \xi' \in \R^{d-1}.
  \end{equation}
  \subsection{Boundary operators yielding symmetry}
\label{sec: Examples  of boundary operators yielding symmetry}

We give some examples of pairs of boundary operators $B_1, B_2$ that
\begin{enumerate}
\item fulfill the \LS condition,
\item yield symmetry for the bi-Laplace operator $P = \Delta^2$, that is, 
\begin{align*}
  \inp{P u }{v}_{L^2(\Omega)} =  \inp{u }{P v}_{L^2(\Omega)} 
\end{align*}
for $u,v\in H^4(\Omega)$ such that $B_j u_{|\partial\Omega} = B_j
v_{|\partial\Omega} =0$, $j=1,2$.
\end{enumerate}

We first recall that following Green formula
\begin{align}
\label{eq: Green formula}
  \inp{\Delta u }{v}_{L^2(\Omega)} 
  =  \inp{u }{\Delta  v}_{L^2(\Omega)} 
  + \inp{\partial_{n} u\bd}{v\bd}_{L^2(\partial\Omega)}  
  - \inp{u\bd}{\partial_{n} v\bd}_{L^2(\partial\Omega)}, 
\end{align}
which applied twice gives
$\inp{P u }{v}_{L^2(\Omega)} 
  =  \inp{u }{P  v}_{L^2(\Omega)} + T(u,v)$
with \begin{align}
  \label{eq: Green formula biLaplace}
  T(u,v) &= \inp{\partial_{n} \Delta  u\bd}{v\bd}_{L^2(\partial\Omega)}  
  - \inp{\Delta  u\bd}{\partial_{n} v\bd}_{L^2(\partial\Omega)}\notag\\
  &\quad+ \inp{\partial_{n} u\bd}{\Delta  v\bd}_{L^2(\partial\Omega)}  
  - \inp{u\bd}{\partial_{n} \Delta  v\bd}_{L^2(\partial\Omega)}.
\end{align}
Using normal geodesic coordinates in a \nhd of the whole boundary
$\partial\Omega$ allows one to write $\Delta = \partial_n^2 + \Delta'$
where $\Delta'$ is the resulting Laplace operator on the boundary,
that is, associated with the trace of the metric on
$\partial\Omega$. Since $\Delta'$ is selfadjoint on $\partial\Omega$ 
this allows one to write formula \eqref{eq: Green formula biLaplace}
in the alternative forms
\begin{align}
  \label{eq: Green formula biLaplace-bis}
  T(u,v) &= \inp{\partial_{n}^3 u\bd}{v\bd}_{L^2(\partial\Omega)}  
  - \inp{(\partial_{n}^2 + 2 \Delta') u\bd}{\partial_{n}v\bd}_{L^2(\partial\Omega)}
\notag\\
  &\quad 
    + \inp{\partial_{n} u\bd}{ (\partial_{n}^2 + 2 \Delta')v\bd}_{L^2(\partial\Omega)}  
  - \inp{u\bd}{\partial_{n}^3  v\bd}_{L^2(\partial\Omega)},
\end{align}
or 
\begin{align}
  \label{eq: Green formula biLaplace-ter}
  T(u,v) &= 
  \inp{(\partial_{n}^3 
  + 2 \Delta'\partial_{n})  u\bd}{v\bd}_{L^2(\partial\Omega)}  
  - \inp{\partial_{n}^2  u\bd}{\partial_{n}v\bd}_{L^2(\partial\Omega)}
  \notag\\
  &\quad 
   + \inp{\partial_{n} u\bd}{ \partial_{n}^2 v\bd}_{L^2(\partial\Omega)}  
  - \inp{u\bd}{(\partial_{n}^3 
           + 2 \Delta'\partial_{n})  v\bd}_{L^2(\partial\Omega)}.
\end{align}
With this computations, one can provide a list of examples in which the boundary operators yield symmetry and fulfill the \LS condition. This can be found in \cite[Section 3.5]{Jer-Em}.\\
We start our list of examples with the most basics ones.
\begin{example}[Hinged boundary conditions]
\label{ex: hinged boundary conditions}
This type of conditions is given by $B_1 u\bd = u\bd$ and $B_2 u\bd
= \Delta u\bd$. With \eqref{eq: Green formula biLaplace} one finds
$T(u,v) =0$ in the case of homogeneous conditions, hence symmetry. We have
$b_1(x,\xi',\xi_d)=1$ and $b_2(x,\xi',\xi_d)=
-\xi_d^2$. It follows that 
\begin{equation*}
    \det
    \begin{pmatrix}
      b_1 &b_2 \\[2pt]
      \partial_{\xi_d}b_1& \partial_{\xi_d}b_2
    \end{pmatrix} (x,\xi',\xi_d=i|\xi'|_x)
    =
     \det
    \begin{pmatrix}
     1 &|\xi'|_x \\[2pt]
     0&-2i|\xi'|_x
    \end{pmatrix} =-2i|\xi'|_x \neq 0,
\end{equation*}
if $\xi'\neq 0$ and thus the \LS condition holds by Lemma~\ref{lemma: Lopatinskii bi-laplacian op}.
\end{example}
\begin{example}[Clamped  boundary conditions]
This type of conditions refers to $B_1 u\bd = u\bd$ and $B_2 u\bd
= \partial_n u\bd$. With \eqref{eq: Green formula biLaplace-bis}  one finds
$T(u,v) =0$ in the case of homogeneous conditions, hence symmetry. With the notation of Section~\ref{LS:definition} this gives $b_1(x,\xi',\xi_d)=1$ and $b_2(x,\xi',\xi_d)=-i\xi_d$. It follows that
We have
\begin{equation*}
    \det
    \begin{pmatrix}
      b_1 &b_2 \\[2pt]
      \partial_{\xi_d}b_1& \partial_{\xi_d}b_2
    \end{pmatrix} (x,\xi',\xi_d=i|\xi'|_x)
    =
     \det
    \begin{pmatrix}
     1 &|\xi'|_x \\[2pt]
     0&-i
    \end{pmatrix} =-i\neq 0.
\end{equation*}
Thus the \LS condition holds by Lemma~\ref{lemma: Lopatinskii bi-laplacian op}.
\end{example}

\begin{examples}[More examples]
$\phantom{-}$
\begin{enumerate}
\item Take  $B_1 u\bd =\partial_n  u\bd$ and $B_2 u\bd
= \partial_n \Delta u\bd$. With these boundary conditions the bi-Laplace operator is precisely
the square of the Neumann-Laplace operator. The symmetry property is
immediate and so is the \LS condition.

\item Take $B_1 u\bd = (\partial_n^2 + 2 \Delta') u\bd$ and $B_2 u\bd
= \partial_n^3 u\bd$. With \eqref{eq: Green formula biLaplace-bis} one finds
$T(u,v) =0$ in the case of homogeneous conditions, hence symmetry.

We have $b_1(x,\xi',\xi_d)=-\xi_d^2 - 2 |\xi'|_x^2$ and
$b_2(x,\xi',\xi_d)=i\xi_d^3$ and
\begin{align*}
   \det
    \begin{pmatrix}
      b_1 &b_2 \\[2pt]
      \partial_{\xi_d}b_1& \partial_{\xi_d}b_2
    \end{pmatrix} (x,\xi',\xi_d=i|\xi'|_x)
    &=
     \det
    \begin{pmatrix}
     -|\xi'|_x^2 &|\xi''|^3_x \\[2pt]
     -2 i |\xi'|_x&-3i|\xi'|^2_x
    \end{pmatrix} \\
    &=5 i |\xi'|^4_x \neq 0,
\end{align*}
if $\xi'\neq 0$ and thus the \LS condition holds by Lemma~\ref{lemma: Lopatinskii bi-laplacian op}.
\item Take $B_1 u\bd = \partial_n u\bd$ and $B_2 u\bd
= (\partial_n^3+ A')  u\bd$, with $A'$ a symmetric differential
operator of order less than or equal to three on
$\partial \Omega$, with homogeneous principal symbol $a'(x,\xi')$ such that
$a'(x,\xi') \neq 2 |\xi'|_x ^3$ for $\xi' \neq 0.$

With \eqref{eq: Green formula biLaplace-bis} one
finds
\begin{align*}
T(u,v)  =  \inp{-A' u\bd}{v\bd}_{L^2(\partial\Omega)}  
  + \inp{u\bd}{A' v\bd}_{L^2(\partial\Omega)} =0,
\end{align*}
in the case of homogeneous conditions, hence symmetry for $P$.

We have $b_1(x,\xi',\xi_d)=-i \xi_d$ and
$b_2(x,\xi',\xi_d)=i\xi_d^3 + a'(x,\xi')$ with $a'$ the
principal symbol of $A'$.
\begin{align*}
   \det
    \begin{pmatrix}
      b_1 &b_2 \\[2pt]
      \partial_{\xi_d}b_1& \partial_{\xi_d}b_2
    \end{pmatrix} (x,\xi',\xi_d=i|\xi'|_x)
    &=
     \det
    \begin{pmatrix}
     |\xi'|_x &|\xi'|^3_x + a'(x,\xi')\\[2pt]
    - i &-3i|\xi'|^2_x
    \end{pmatrix} \\
    &= i \big( a'(x,\xi') -2  |\xi'|^3_x\big)\neq 0,
\end{align*}
if $\xi'\neq 0$ since $a'(x,\xi') \neq 2 |\xi'|^3_x$  by
assumption implying
that the \LS condition holds by Lemma~\ref{lemma: Lopatinskii bi-laplacian op}.
\item 
 Take  $B_1 u\bd =u\bd$ and $B_2 u\bd
=(\partial_n^2+ A' \partial_n) u\bd$  with $A'$ a symmetric  differential
operator of order less than or equal to one on
$\partial \Omega$, with homogeneous principal symbol $a'(x,\xi')$ such that
$a'(x,\xi') \neq - 2 |\xi'|_x$ for $\omega' \neq 0.$
This is  a refinement of the  boundary conditions given
in Example~\ref{ex: hinged boundary conditions} above.

With \eqref{eq: Green formula biLaplace-bis} one
finds
\begin{align*}
  T(u,v)  
  =  \inp{A' \partial_n u\bd}{\partial_n v\bd}_{L^2(\partial\Omega)} 
  + \inp{\partial_n u\bd}{- A' \partial_n v\bd}_{L^2(\partial\Omega)}  
  =0,
\end{align*}
in the case of homogeneous conditions, hence symmetry for $P$.\\
We have $b_1(x,\xi', \xi_d)=1$ and
$b_2(x,\xi',\xi_d)=-\xi_d^2 -i \xi_d a'(x,\xi')$ with $a'$ the
principal symbol of $A'$.
\begin{align*}
    \det
    \begin{pmatrix}
      b_1 &b_2 \\[2pt]
      \partial_{\xi_d}b_1& \partial_{\xi_d}b_2
    \end{pmatrix} (x,\xi',\xi_d=i|\xi'|_x)
    &=
     \det
    \begin{pmatrix}
     1&|\xi'|^2_g + |\xi'|_x  a'(x,\xi')\\[2pt]
    0 &-2i|\xi'|_x - i a'(x,\xi')
    \end{pmatrix} \\
    &= - i \big( a'(x,\xi') + 2  |\xi'|_x\big)\neq 0,
\end{align*}
if $\xi'\neq 0$ since $a'(x,\xi') \neq - 2 |\xi'|_x$ by
assumption implying
that the \LS condition holds by Lemma~\ref{lemma: Lopatinskii bi-laplacian op}.
\item Take $B_1 u\bd = (\partial_n^2  + A' \partial_n) u\bd$ and $B_2 u\bd
= (\partial_n^3+ 2 \partial_n\Delta')  u\bd$,  with $A'$ a symmetric  differential
operator of order less than or equal to one on
$\partial \Omega$, with homogeneous principal symbol $a'(x,\xi')$ such that
$2 a'(x,\xi') \neq - 3 |\xi'|_x $ for $\xi' \neq 0.$
With \eqref{eq: Green formula biLaplace-ter} one
finds
\begin{align*}
  T(u,v)  
  =  \inp{A' \partial_n u\bd}{\partial_n v\bd}_{L^2(\partial\Omega)} 
  + \inp{\partial_n u\bd}{- A' \partial_n v\bd}_{L^2(\partial\Omega)}  
  =0,
\end{align*}
in the case of homogeneous conditions, hence symmetry for $P$.\\
We have $b_1(x,\xi',\xi_d)=-\xi_d^2 -i \xi_d a'(x,\xi')$ and
$b_2(x,\xi',\xi_d)=i\xi_d^3 + 2 i \xi_d |\xi'|_x^2$ and
\begin{align*}
    \det
    \begin{pmatrix}
      b_1 &b_2 \\[2pt]
      \partial_{\xi_d}b_1& \partial_{\xi_d}b_2
    \end{pmatrix} (x,\xi',\xi_d=i|\xi'|_x)
    &=
     \det
    \begin{pmatrix}
     |\xi'|^2_x + |\xi'|_x   a'(x,\xi')&-|\xi'|^3_x \\[2pt]
     -2i|\xi'|_x - i a'(x,\xi')&-i|\xi'|^2_x
    \end{pmatrix} \\
    & =- i |\xi'|^3_x \big( 2 a'(x,\xi') + 3 |\xi'|_x\big) \neq 0,
\end{align*}
if $\xi'\neq 0$ since $2 a'(x,\xi') + 3 |\xi'|_x  \neq 0$ by
assumption implying that  the \LS condition holds by Lemma~\ref{lemma: Lopatinskii bi-laplacian op}.
\end{enumerate}
\end{examples}

\subsection{Some properties of the bi-Laplace operator}
Let $(\Pell, D(\Pell))$ be the unbounded operator on $L^2(\Omega),$ with
\begin{align}
  \label{eq: domain P intro}
  D(\Pell) = \big\{ u \in H^4(\Omega); \ B_1
  u_{|\partial \Omega}
  = B_2 u_{|\partial \Omega}=0\big\},
\end{align}
and given by $\Pell u = \Delta^2 u$ for $u \in D(\Pell)$.
The boundary operators $B_1$
and $B_2$ of orders $k_j$, $j=1,2$, less than or equal to $3$ in the normal direction are chosen so that
\begin{enumerate}
\item[(i)] 
  the \LS condition of Definition~\ref{def: LS} is fulfilled for $(P,
  B_1, B_2)$ on $\partial\Omega$;
\item[(ii)] 
  the operator $P$ is symmetric under homogeneous boundary
  conditions, that is, 
  \begin{align}
    \label{eq: symmetry property}
    \inp{Pu }{v}_{L^2(\Omega)} = \inp{u}{P v}_{L^2(\Omega)}, 
  \end{align}
  for $u, v \in H^4(\Omega)$ such that $B_j u_{|\partial \Omega} = B_j v_{|\partial \Omega} = 
  0$ on $\partial\Omega$, $j=1,2$. 
 \end{enumerate}
With the assumed \LS condition the operator 
\begin{align}
  \label{eq: Fredholm operator}
  L: H^4 (\Omega) &\to L^2(\Omega) \oplus H^{7/2-k_1} (\partial \Omega)
  \oplus H^{7/2-k_2} (\partial \Omega), \notag\\
  u & \mapsto (Pu, B_1 u_{|\partial \Omega} , B_2 u_{|\partial \Omega}),
\end{align}
is Fredholm.

\begin{enumerate}
  \item[(iii)]
We shall further assume that the Fredholm  index of the operator $L$ is zero.
\end{enumerate}

The previous symmetry property gives
$\inp{Pu}{u}_{L^2(\Omega)} \in \R$. We further assume the following
nonnegativity property:
\begin{enumerate}
\item[(iv)] For $u \in H^4(\Omega)$ such that
  $B_j u_{|\partial \Omega} = 0$ on $\partial\Omega$, $j=1,2$ one has
  \begin{align}
    \label{eq: nonnegativity assumption P}
    \inp{Pu}{u}_{L^2(\Omega)}\geq 0.
    \end{align}
  \end{enumerate}
Associated with $P$ and the boundary operators $B_1$
and $B_2$ is the operator $(\Pell, D(\Pell))$ on $L^2(\Omega)$, with domain
\begin{align*}
  D(\Pell) = \big\{ u \in L^2(\Omega); \ P u \in L^2(\Omega), \ B_1
  u_{|\partial \Omega}
  = B_2 u_{|\partial \Omega}=0\big\},
\end{align*}
and given by $\Pell u = P u \in  L^2(\Omega)$ for $u \in  D(\Pell)$.
The definition of $D(\Pell)$ makes sense since having $P u \in
L^2(\Omega)$ for $u \in L^2(\Omega)$ implies that the traces $\partial_\nu^k
u_{|\partial \Omega}$ are well defined for $k=0, 1,2,3$. 

Since the \LS condition holds on $\partial \Omega$ one has $D(\Pell)
\subset H^4(\Omega)$ (see for instance Theorem~20.1.7 in
\cite{Hoermander:V3}) and thus one can also write $D(\Pell)$ as in
\eqref{eq: domain P intro}. 
From the assumed nonnegativity in \eqref{eq: nonnegativity assumption P} above one finds that
$\Pell + \id$ is injective. Since the operator
\begin{align*}
  L': H^4 (\Omega) &\to L^2(\Omega) \oplus H^{7/2-k_1} (\partial \Omega)
  \oplus H^{7/2-k_2} (\partial \Omega)\\
  u & \mapsto (Pu +u, B_1 u_{|\partial \Omega} , B_2 u_{|\partial \Omega})
\end{align*}
is Fredholm and has the same zero index as $L$ defined in \eqref{eq:
  Fredholm operator}, one finds that $L'$ is surjective. Thus
$\range(\Pell + \id) = L^2(\Omega)$.
One thus concludes that $\Pell$ is maximal monotone. 
From the assumed symmetry property \eqref{eq: symmetry property} and
one finds that   $\Pell$ is selfadjoint, using
that a symmetric maximal monotone operator is selfadjoint (see for
instance Proposition~7.6 in \cite{Brezis:11}).

The resolvent of $\Pell +\id$ being compact on $L^2(\Omega)$,  $\Pell$ has a sequence of
eigenvalues with finite multiplicities. With the assumed nonnegativity
\eqref{eq: nonnegativity assumption P}   they take the form of a sequence 
\begin{align*}
0\leq \mu_0 \leq \mu_1 \leq \cdots \leq \mu_k \leq \cdots,~~~\text{with}~~\Pell\phi_j=\mu_j\phi_j~~\text{and}~~\|\phi_j\|^2_{L^2(\Omega)}=1.
\end{align*}
that grows to $+\infty$. Associated with this sequence is
$(\phi_j)_{j \in \N}$ a Hilbert basis of $L^2(\Omega)$.\\
\\
Well-posedness for the parabolic system (\ref{eq: parabolic}) is shown in \cite[Corollary 1.10]{JL} and we recall it here.
\begin{corollary}
The operator $(\Pell,D(\Pell))$ generates an analytic $C_0-$semigroup $S(t)$ on $L^2(\Omega).$ For $T>0,$ $y_0\in L^2(\Omega)$, and $f\in L^2(0,T; H^{-2}(\Omega)),$ there exists a unique solution 
$$y\in L^2(0,T; D(\Pell))\cap \mathscr{C}([0,T]; L^2(\Omega))\cap H^1(0,T; H^{-2}(\Omega)),$$ given by $y(t)=S(t)y_0+\int_{0}^tS(t-s)f(s)ds,$ such that
$$\partial_ty+\Delta^2y=f~~~~\text{for a.e.}~t\in(0,T),~~~y_{|t=0}=y_0.$$
\end{corollary}

\subsection{The \LS condition for the augmented operator}
We consider the augmented operator $Q=D_s^4+\Delta ^2,$ with $s\in\R.$ Let $q$ denote the principal symbol of $Q.$ We have
 $$q(x,\sigma,\xi',\xi_d)=\sigma^4+(\xi^2_d+|\xi'|^2_x)^2=(\xi^2_d-i\sigma^2+|\xi'|^2_x)(\xi^2_d+i\sigma^2+|\xi'|^2_x).$$ 
  We look for the roots of $q(x,\sigma,\xi',\xi_d)$ as a polynomial in $\xi_d.$ We recall that for a given complex number $z=a+ib,$ the square root of $z$ is given by:
  \begin{equation*}
\pm \left(\sqrt{\frac{m+a}{2}}+i\sqrt{\frac{m-a}{2}}\right)~~\text{if}~b>0,~~\pm \left(\sqrt{\frac{m+a}{2}}-i\sqrt{\frac{m-a}{2}}\right)~~\text{if}~b<0,
  \end{equation*}
  where $m=\sqrt{a^2+b^2}=|z|.$ Consequently the roots of $q(x,\sigma,\xi',\xi_d)$ are
  \begin{equation*}\label{augmrac1}
      \pm \left(\sqrt{\frac{\sqrt{\sigma^4+|\xi'|^4_x}-|\xi'|^2_x}{2}}+i\sqrt{\frac{\sqrt{\sigma^4+|\xi'|^4_x}+|\xi'|^2_x}{2}}\right)
  \end{equation*}
  \begin{equation*}\label{augmrac2}
      \pm \left(\sqrt{\frac{\sqrt{\sigma^4+|\xi'|^4_x}-|\xi'|^2_x}{2}}-i\sqrt{\frac{\sqrt{\sigma^4+|\xi'|^4_x}+|\xi'|^2_x}{2}}\right).
  \end{equation*}
  Let $\rho_1$ and $\rho_2$ denote the roots with positive imaginary part, then
  \begin{equation}\label{rac1}
\rho_1 :=\rho_1(x,\sigma,\xi')= -\sqrt{\frac{\sqrt{\sigma^4+|\xi'|^4_x}-|\xi'|^2_x}{2}}+i\sqrt{\frac{\sqrt{\sigma^4+|\xi'|^4_x}+|\xi'|^2_x}{2}}
  \end{equation}
  \begin{equation}\label{rac2}
\rho_2:=\rho_2(x,\sigma,\xi')= \sqrt{\frac{\sqrt{\sigma^4+|\xi'|^4_x}-|\xi'|^2_x}{2}}+i\sqrt{\frac{\sqrt{\sigma^4+|\xi'|^4_x}+|\xi'|^2_x}{2}}.
  \end{equation}
 We observe that $\rho_j(x,\lambda\sigma,\lambda\xi')=\lambda\rho_j(x,\sigma,\xi')$ for $j=1$ or $2$, then the roots $\rho_j$ are homogeneous of degree one in the variable $(\sigma,\xi').$
  \begin{remark}
  The roots $\rho_1$ and $\rho_2$ are equals if and only if $\sigma=0.$
  \end{remark}
  Set $q^+(x,\sigma,\xi',\xi_d)=(\xi_d-\rho_1)(\xi_d-\rho_2).$ According to Definition \ref{def: LS}, the \LS condition is satisfies for $(Q, B_1, B_2)$ at the point $(x,\xi',\sigma)$ with $(\xi',\sigma)\neq( 0,0)$ if for every polynomial function $f(\xi_d)$, there exist $c_1,c_2\in\mathbb{C}$ and a polynomial function $h(\xi_d)$ such that $$f(\xi_d)=c_1b_1(x,\xi',\xi_d)+c_2b_2(x,\xi',\xi_d)+h(\xi_d)(\xi_d-\rho_1)(\xi_d-\rho_2).$$
  Therefore the \LS condition holds if the complex numbers $\rho_1$ and $\rho_2$ are roots of the polynomial function $f(\xi_d)-c_1b_1(x,\xi',\xi_d)-c_2b_2(x,\xi',\xi_d).$ This leads to the following determinant condition 
  \begin{equation}\label{Lopaugm1}
    \det
    \begin{pmatrix}
      b_1(x,\xi', \xi_d=\rho_1)&b_2(x, \xi',\xi_d=\rho_1)\\[2pt]
      b_1(x, \xi',\xi_d=\rho_2)& b_2(x, \xi',\xi_d=\rho_2)
    \end{pmatrix}
    \neq 0.
\end{equation}
\begin{remark}\label{homoge}
We have for $\ell=1$ or $2$ $$b_{\ell}(x,\lambda\xi',\rho_{\ell}(x,\lambda\sigma,\lambda\xi'))=\lambda^{k_{\ell}}b_{\ell}(x,\xi',\rho_{\ell}(x,\sigma,\xi'))$$ and 
$$b_{\ell}(x,\lambda\frac{\xi'}{\lambda},\rho_{\ell}(x,\lambda\frac{\sigma}{\lambda},\lambda\frac{\xi'}{\lambda}))=\lambda^{k_{\ell}}b_{\ell}(x,\frac{\xi'}{\lambda},\rho_{\ell}(x,\frac{\sigma}{\lambda},\frac{\xi'}{\lambda})).$$ But writing $$\frac{\xi'}{\lambda}+\frac{\sigma}{\lambda}=\underbrace{\frac{\xi'}{(|\xi'|_x^4+\sigma^4)^{1/4}}}_{:=X}+\underbrace{\frac{\sigma}{(|\xi'|_x^4+\sigma^4)^{1/4}}}_{:=Y}$$ with $\lambda=(|\xi'|_x^4+\sigma^4)^{1/4},$ we have $|X|^4+Y^4=1.$ 
\end{remark}
If the \LS condition holds at some $(x^0,\sigma^0,\xi^0)$ where $\sigma^0\neq 0$,
  because of homogeneity, there exists $C_0>0$ such that
  \begin{equation}\label{eq: homog formul LS condition 1}
    \left| \det
    \begin{pmatrix}
      b_1(x,\xi', \xi_d=\rho_1)&b_2(x, \xi',\xi_d=\rho_1)\\[2pt]
      b_1(x, \xi',\xi_d=\rho_2)& b_2(x, \xi',\xi_d=\rho_2)
    \end{pmatrix}\right|
    \geq C_0 \Lambda^{k_1+k_2 },
  \end{equation}
  with $\Lambda=(\sigma^4+|\xi'|^4_x)^{1/4}=\frac{1}{2}(|\rho_1|+|\rho_2|)$ in a conic neighborhood $\mathscr{U}$ of $(x^0,\sigma^0,\xi^0).$

\begin{proposition}\label{lien-LS}
Let $(x,\xi')\in T^*\partial\Omega\cong\partial\Omega\times\R^{d-1}$ and $\sigma\in \R$ with $(\sigma,\xi')\neq (0,0).$ Suppose that the \LS condition holds for $(P,B_1, B_2)$ at $(x,\xi').$ Then the \LS condition also holds true for $(Q, B_1, B_2)$ at $(x,\sigma,\xi')$ for $\sigma$ near zero.
\end{proposition}
Before the proof of this proposition we make the following \emph{key} observation.
\begin{observation}
Consider the following two boundary operators $$B_1(x,D)=D^2_d+\alpha R(x,D')~~~~\text{and}~~~~ B_2(x,D)=D_d,~~~\text{where}~~D_d=-i\partial_d,$$ and with respective principal symbols $$b_1(x,\xi',\xi_d)=\xi_d^2+\alpha|\xi'|^2_x~~ \text{and} ~~b_2(x,\xi',\xi_d)=\xi_d.$$ Here $R(x,D')$ denote a tangential differential elliptic operator of order $2$ with principal symbol $r(x,\xi')=|\xi'|^2_x$ and $\alpha\in\R.$ The \LS condition for $(P,B_1,B_2)$ (the case $\sigma=0$) reads as follows (with $\rho=i|\xi'|^2_x$):
\begin{equation*}
\begin{vmatrix}
\rho^2+\alpha|\xi'|^2_x &~~~~ \rho\\
\\
2\rho & 1
\end{vmatrix}=-\rho^2+\alpha|\xi'|^2_x=(1+\alpha)|\xi'|^2_x\neq 0,~~\text{if}~~\alpha\neq -1.
\end{equation*}
On the other hand, the \LS condition for $(Q, B_1, B_2)$ (the case $\sigma\neq 0$) reads as follows
\begin{equation}\label{LSQ}
\begin{vmatrix}
\rho^2_1+\alpha|\xi'|^2_x &~~~~ \rho_1\\
\\
\rho^2_2+\alpha|\xi'|^2_x &~~ \rho_2
\end{vmatrix}=\rho^2_1\rho_2+\alpha\rho_2|\xi'|^2_x-\rho_1\rho^2_2-\alpha\rho_1|\xi'|^2_x=(\rho_1-\rho_2)(\rho_1\rho_2-\alpha|\xi'|^2_x).
\end{equation}
We recall that $\rho_1=-\overline{\rho}_2$ and so $\rho_1\rho_2=-|\rho_1|^2=-\sqrt{\sigma^4+|\xi'|^4_x}$. Therefore equation (\ref{LSQ}) becomes
$$(\rho_1-\rho_2)(\rho_1\rho_2-\alpha|\xi'|^2_x)=(\rho_1-\rho_2)(-\sqrt{\sigma^4+|\xi'|^4_x}-\alpha|\xi'|^2_x).$$
If $\alpha<-1,$ then for $\sigma^4=(\alpha^2-1)|\xi'|^4_x,$ we have $-\sqrt{\sigma^4+|\xi'|^4_x}-\alpha|\xi'|^2_x=0$ while for $\sigma=0,$ $-\sqrt{\sigma^4+|\xi'|^4_x}-\alpha|\xi'|^2_x=-|\xi'|^2_x-\alpha|\xi'|^2_x=-(\alpha+1)|\xi'|^2_x>0.$ This means that having \LS condition for $(P, B_1, B_2)$ does not imply necessarily that the \LS condition also holds true for $(Q, B_1, B_2).$
\end{observation}
More generally, we set 
$$K=\det
    \begin{pmatrix}
      b_1(x,\xi', \xi_d=\rho_1)&b_2(x, \xi',\xi_d=\rho_1)\\[2pt]
      b_1(x, \xi',\xi_d=\rho_2)& b_2(x, \xi',\xi_d=\rho_2)
    \end{pmatrix}~~\text{and}~~ K'= \det\begin{pmatrix}
      b_1&b_2\\[2pt]
      \partial_{\xi_d}b_1& \partial_{\xi_d}b_2
    \end{pmatrix}(x,\xi',\xi_d=\rho).$$ Using the notation (\ref{symbol-au-bord}), straightforward computations gives:
    \begin{align*}
    K&=(\rho_2-\rho_1)\left[b_1^{k_1}b_2^{k_2-1}-b_1^{k_1-1}b_2^{k_2}\right]\\ 
    &+(\rho_2-\rho_1)\left[b_1^{k_1}b_2^{k_2-2}-b_1^{k_1-2}b_2^{k_2}\right](\rho_2+\rho_1)\\ 
    &+(\rho_2-\rho_1)\left[b_1^{k_1}b_2^{k_2-3}-b_1^{k_1-3}b_2^{k_2}\right](\rho^2_2+\rho_2\rho_1+\rho^2_1)\\ 
    &+(\rho_2-\rho_1)\left[b_1^{k_1-1}b_2^{k_2-2}-b_1^{k_1-2}b_2^{k_2-1}\right](\rho_1\rho_2)\\ 
    &+
    (\rho_2-\rho_1)\left[b_1^{k_1-1}b_2^{k_2-3}-b_1^{k_1-3}b_2^{k_2-1}\right](\rho_1\rho_2)(\rho_2+\rho_1)\\
    &+(\rho_2-\rho_1)\left[b_1^{k_1-2}b_2^{k_2-3}-b_1^{k_1-3}b_2^{k_2-2}\right](\rho_1\rho_2)^2
\end{align*}
and 
\begin{align*}
    K'&=\left[b_1^{k_1}b_2^{k_2-1}-b_1^{k_1-1}b_2^{k_2}\right]+2\left[b_1^{k_1}b_2^{k_2-2}-b_1^{k_1-2}b_2^{k_2}\right]\rho\\ 
    &+3\left[b_1^{k_1}b_2^{k_2-3}-b_1^{k_1-3}b_2^{k_2}\right]\rho^2+\left[b_1^{k_1-1}b_2^{k_2-2}-b_1^{k_1-2}b_2^{k_2-1}\right]\rho^2\\ 
    &+2\left[b_1^{k_1-1}b_2^{k_2-3}-b_1^{k_1-3}b_2^{k_2-1}\right]\rho^3+3\left[b_1^{k_1-2}b_2^{k_2-3}-b_1^{k_1-3}b_2^{k_2-2}\right]\rho^4.
\end{align*}
Here we recall that $\rho=i|\xi'|^2_x,$  $\rho_1$ and $\rho_2$ are given (\ref{rac1}) and (\ref{rac2}). We can see that having $K'\neq 0$ does not imply necessarily that $K\neq 0$ for $\sigma\neq 0.$\\
\\
For $\sigma$ near zero, we have that having the \LS condition for $(P, B_1, B_2)$ implies that the \LS condition holds true for $(Q, B_1, B_2).$ This is shown in the proof of Proposition \ref{lien-LS} below.
\begin{proof}[Proof of Proposition~\ref{lien-LS}]
Set $\rho_2=\rho_1+h$ with $h=2\sqrt{\frac{\sqrt{\sigma^4+|\xi'|^4_x}-|\xi'|^2_x}{2}}\in\mathbb{R}.$ We have $h\to 0$ when $\sigma\to 0$ and $\rho_j\to i|\xi'|_x$ when $\sigma\to 0$ for $j=1$ or $2.$ Relation (\ref{Lopaugm1}) becomes
\begin{equation*}
    \det
    \begin{pmatrix}
      b_1(x,\xi',\xi_d=\rho_1)&b_2(x, \xi',\xi_d=\rho_1)\\[2pt]
      b_1(x,\xi',\xi_d=\rho_1+h)& b_2(x, \xi',\xi_d=\rho_1+h)
    \end{pmatrix}
    \neq 0,
\end{equation*}
which is equivalent to
\begin{multline*}
    \det
    \begin{pmatrix}
      b_1(x,\xi', \xi_d=\rho_1)&b_2(x, \xi',\xi_d=\rho_1)\\[2pt]
      \frac{b_1(x, \xi',\xi_d=\rho_1+h)- b_1(x,\xi',\xi_d=\rho_1)}{h}& \frac{b_2(x, \xi',\xi_d=\rho_1+h)- b_2(x,\xi',\xi_d=\rho_1)}{h}
    \end{pmatrix}\\
 =   \det
    \begin{pmatrix}
      b_1(x,\xi', \xi_d=i|\xi'|_x)&b_2(x, \xi',\xi_d=i|\xi'|_x)\\[2pt]
      \partial_{\xi_d}b_1(x,\xi', \xi_d=i|\xi'|_x)& \partial_{\xi_d}b_2(x,\xi', \xi_d=i|\xi'|_x)
    \end{pmatrix}
    \neq 0,
    \end{multline*}
as $\sigma\rightarrow 0$. Similarly we can also set
$$\rho_1=-h+ir~~\text{and}~~\rho_2=h+ir$$ where
$h=\sqrt{\frac{\sqrt{\sigma^4+|\xi'|^4_x}-|\xi'|^2_x}{2}}\in\mathbb{R}$ and $r=\sqrt{\frac{\sqrt{\sigma^4+|\xi'|^4_x}+|\xi'|^2_x}{2}}\in\mathbb{R}.$ By using relation (\ref{Lopaugm1}), we have
\begin{equation*}
    \det
    \begin{pmatrix}
      b_1(x,\xi',\xi_d=-h+ir)&b_2(x, \xi',\xi_d=-h+ir)\\[2pt]
      b_1(x,\xi',\xi_d=h+ir)& b_2(x, \xi',\xi_d=h+ir)
    \end{pmatrix}
    \neq 0,
\end{equation*}
which is equivalent to
\begin{multline*}
    \det
    \begin{pmatrix}
      b_1(x,\xi',\xi_d=\rho_1)& b_2(x,\xi',\xi_d=\rho_1)\\[2pt]
      \frac{b_1(x, \xi',\xi_d=\theta)- b_1(x,\xi',\xi_d=\rho_1)}{2h}& \frac{b_2(x, \xi',\xi_d=\theta)- b_2(x,\xi',\xi_d=\rho_1)}{2h}
    \end{pmatrix}\\
=    \det
    \begin{pmatrix}
      b_1(x,\xi', \xi_d=ir)&b_2(x, \xi',\xi_d=ir)\\[2pt]
      \partial_{\xi_d}b_1(x,\xi', \xi_d=ir)& \partial_{\xi_d}b_2(x,\xi', \xi_d=ir)
    \end{pmatrix}
    \neq 0,
\end{multline*}
as $h\to 0$ and since $r=|\xi'|_x$. Hence, 
for $\sigma $ small, having \LS for $(P, B_1, B_2)$ implies \LS for $(Q, B_1, B_2).$
\end{proof}
\subsection{Some examples}
In connection with the examples listed in Section~\ref{sec: Examples  of boundary operators yielding symmetry}, we check the validity of the \LS condition for $(Q,B_1,B_2)$ for some boundary conditions $B_1$ and $B_2$ at $(\xi',\sigma)\neq (0,0).$
\begin{example}[Hinged boundary conditions]
With $B_1u=u$ and $B_2u=\Delta u,$ we then have the principal symbols
$b_1(x,\xi',\xi_d)=1$ and $b_2(x,\xi',\xi_d)=-\xi_d^2.$ The \LS condition for $(D^4_s+\Delta^2, B_1, B_2)$ at $(x,\sigma,\xi')$ with $(\sigma,\xi')\neq (0,0)$ is equivalent to have
\begin{equation*}
 \begin{vmatrix}
 1&-\rho^2_1\\
 \\
 1&-\rho^2_2
 \end{vmatrix}=
 -\rho^2_2+\rho^2_1=-(\rho_2-\rho_1)(\rho_2+\rho_1)\neq 0.
\end{equation*}
We note that $\rho_1=-\Bar{\rho_2}$ and therefore
\begin{equation*}
 \begin{vmatrix}
 1&-\rho^2_1\\
 \\
 1&-\rho^2_2
 \end{vmatrix}=
 \Bar{\rho}^2_2-\rho^2_2=(\Bar{\rho_2}-\rho_2)(\Bar{\rho_2}+\rho_2).
\end{equation*}
But $\Bar{\rho_2}^2-\rho^2_2=-4i\Re\rho_2\Im\rho_2=-4i\sqrt{\frac{\sqrt{\sigma^4+|\xi'|^4_x}-|\xi'|^2_x}{2}}\times\sqrt{\frac{\sqrt{\sigma^4+|\xi'|^4_x}+|\xi'|^2_x}{2}}=-2i\sigma^2\neq 0$ if $\sigma\neq 0.$
\end{example}
\begin{example}[Clamped boundary conditions]
With $B_1u=u$ and $B_2u=\partial_n u,$ we then have the principal symbols of $B_1$ and $B_2$ given respectively by
$b_1(x,\xi',\xi_d)=1$ and $b_2(x,\xi',\xi_d)=-i\xi_d.$ The \LS condition for $(D^4_s+\Delta^2, B_1, B_2)$ at $(x,\sigma,\xi')$ with $(\sigma,\xi')\neq (0,0)$ is equivalent to have
\begin{equation*}
 \begin{vmatrix}
 1&-i\rho_1\\
 1&-i\rho_2
 \end{vmatrix}=
 -i\rho_2+i\rho_1=-i(\rho_2-\rho_1)\neq 0.
\end{equation*}
But $\rho_2-\rho_1=2\sqrt{\frac{\sqrt{\sigma^4+|\xi'|^4_x}-|\xi'|^2_x}{2}}$ which is different from zero if $\sigma\neq 0.$
\end{example}
\begin{example}
Let $B_1u=\partial_n u$ and $B_2u=\partial_n\Delta u$. We have $b_1(x,\xi',\xi_d)=-i\xi_d$ and $b_2(x,\xi',\xi_d)=i\xi^3_d.$
The \LS condition for $(D^4_s+\Delta^2, B_1, B_2)$ at $(x,\sigma,\xi')$ with $(\sigma,\xi')\neq (0,0)$ is equivalent to have
\begin{equation*}
 \begin{vmatrix}
 -i\rho_1&i\rho^3_1\\
 \\
 -i\rho_2&i\rho^3_2
 \end{vmatrix}=
 \rho_1\rho^3_2-\rho_2\rho^3_1=(\rho_2-\rho_1)(\rho_2+\rho_1)(\rho_2\rho_1)\neq 0.
\end{equation*}
We note that $\rho_1=-\Bar{\rho_2}$ and then 
$$\rho_1\rho^3_2-\rho_2\rho^3_1=-4i|\rho_2|^2\Re\rho_2\Im\rho_2=-2i\sigma^2\sqrt{\sigma^4+|\xi'|^4_x}\neq0~~\text{if}~\sigma\neq 0.$$
\end{example}
\smallskip
We give an additional example which is totally different from those boundary conditions listed in Section~\ref{sec: Examples  of boundary operators yielding symmetry}.
\begin{example}\label{ex1}
Let us consider the two boundary differential operators $B_1$ and $B_2$ to be of order zero and one respectively with nonvanishing principal symbols
\begin{equation*}
b_1(x,\xi)=\langle\xi,\nu_x\rangle~~\text{and}~~b_2(x,\xi)=\langle\xi,t_x\rangle+i\langle\xi,v_x\rangle
\end{equation*}
where $t_x$ and $v_x$ are two real vector fields on $\Omega$ along $\partial\Omega,$ and $\langle\cdot,\cdot\rangle$ denotes the Euclidean scalar product. We write 
$$v_x=v^{\nu}_x\nu_x+v'_x~~\text{and}~~t_x=t^{\nu}_x\nu_x+t'_x,$$ with $v^{\nu}_x,t^{\nu}_x\in\R$ and $v'_x,t'_x\in T_x\partial\Omega\cong\R^{d-1},$ $\nu_x$ denote the unitary outward pointing vector on $\Omega$ along $\partial\Omega.$ We set $$b_j(x,\xi',\xi_d)=b_j(x,\xi'-\xi_d n_x),$$ where $n_x$ denotes the outward unit pointing conormal vector at $x.$ 
Therefore, we have
$$b_1(x,\xi',\xi_d)=-i\xi_d$$ and 
$$b_2(x,\xi',\xi_d)=\langle\xi'-\xi_d n_x,t_x\rangle+i\langle\xi'-\xi_d n_x,v_x\rangle=\langle\xi',t'_x+iv'_x\rangle-\xi_d(t^{\nu}_x+iv^{\nu}_x).$$
 For $\xi'\in T^*_x\partial\Omega$ such that $\xi'\neq 0,$ the \LS condition holds for $(P,B_1,B_2)$ if and only if
\begin{equation*}
    \det
    \begin{pmatrix}
      b_1 &b_2 \\[2pt]
      \partial_{\xi_d} b_1& \partial_{\xi_d} b_2
    \end{pmatrix} (x,\xi',\xi_d=i|\xi'|_x)
    \neq 0.
\end{equation*}
This is equivalent to have
\begin{equation*}
    \begin{vmatrix}
    |\xi'|_x & ~~~~\langle\xi',t'_x+iv'_x\rangle-i|\xi'|_x(t^{\nu}_x+iv^{\nu}_x)\\
    \\
      -i& -t^{\nu}_x-iv^{\nu}_x
    \end{vmatrix}\neq0.
\end{equation*}
That is,
$$-|\xi'|_x(t^{\nu}_x+iv^{\nu}_x)+i\langle\xi',t'_x+iv'_x\rangle+|\xi'|_x(t^{\nu}_x+iv^{\nu}_x)\neq 0\Longleftrightarrow \langle\xi',t'_x\rangle+i\langle\xi',v'_x\rangle\neq0.$$ Hence, this holds if and only if 
$$\langle\xi',t'_x\rangle\neq 0~~\text{and}~~\langle\xi',v'_x\rangle\neq 0.$$ Therefore the \LS condition holds for $(P,B_1,B_2)$ at $(x,\xi')$ with $\xi'\neq 0$ if and only if $(\langle\xi',t'_x\rangle,\langle\xi',v'_x\rangle)\neq (0,0).$\\
\\
We say that the \LS condition holds for $(Q, B_1,B_2)$ at $(x,\sigma,\xi')$ with $(\sigma,\xi')\neq (0,0)$ if and only if 
\begin{equation*}
    \det
    \begin{pmatrix}
      b_1(x,\xi', \xi_d=\rho_1)&b_2(x, \xi',\xi_d=\rho_1)\\[2pt]
      b_1(x, \xi',\xi_d=\rho_2)& b_2(x, \xi',\xi_d=\rho_2)
    \end{pmatrix}
    \neq 0.
\end{equation*}
Therefore we have
\begin{equation*}
\begin{vmatrix}
 b_1(x,\xi', \xi_d=\rho)&b_2(x, \xi',\xi_d=\rho)\\[2pt]
      b_1(x, \xi',\xi_d=\theta)& b_2(x, \xi',\xi_d=\theta)
\end{vmatrix}=
    \begin{vmatrix}
    -i\rho_1 & ~~~~\langle\xi',t'_x+iv'_x\rangle-\rho_1(t^{\nu}_x+iv^{\nu}_x)\\
    \\
      -i\rho_2&~~~~\langle\xi',t'_x+iv'_x\rangle-\rho_2(t^{\nu}_x+iv^{\nu}_x)
    \end{vmatrix}\neq0,
\end{equation*}
i.e.,
$$i(\rho_2-\rho_1)\langle\xi',t'_x+iv'_x\rangle\neq 0\Longleftrightarrow \langle\xi',t'_x\rangle+i\langle\xi',v'_x\rangle\neq 0,$$
if $\rho_1-\rho_2\neq 0,$ i.e., $\sigma\neq 0.$ Hence the \LS condition holds for $(Q, B_1,B_2)$ at $(x,\sigma,\xi)$ with $(\sigma,\xi')\neq (0,0)$ if and only if $(\langle\xi',t'_x\rangle,\langle\xi',v'_x\rangle)\neq (0,0).$

\end{example}

\subsection{Stability of the \LS condition}\label{sec: LS condition}
To prepare the study of how the \LS condition behave under conjugation with Carleman exponential weight function, we show that the \LS condition for $(Q, B_1, B_2)$ is stable under small perturbations.
\begin{lemma}
  \label{lemma: perturbation LS}
  Let $V^0$ be a compact set of $\partial\Omega$ such that the \LS
  condition holds for $(Q, B_1, B_2)$ at the point $(x^0,\sigma^0,\xi^0)$ of $V^0$, then the \LS condition remains valid for $(Q,B_1,B_2)$ at every point $(x,\sigma,\xi)$ of $V^0.$ That is to say that the \LS condition is stable under small perturbations of the coefficients.
\end{lemma}
\begin{proof}
According to Definition \ref{def: LS}, let 
$$f(\z)=c_1b_1(x,\xi',\z)+c_2b_2(x,\xi',\z)+g(z)h(x,\sigma,\xi',\z),$$ for all $\xi'\in \R^{d-1},$ $x\in V^0,$ $\sigma\in\R$ and $\z\in\C.$
We suppose that $h$ is a polynomial function (in $\z$) of order 2 and plays the same role as $q^+.$ 
In addition, we consider $f$ to be at most of degree one, and 
\begin{equation}\label{symb1}
b_1(x,\xi',\z)=b_1^1(x,\xi')\z+b_1^0(x,\xi'),~~~~b_2(x,\xi',\z)=b_2^1(x,\xi')\z+b_2^0(x,\xi'),
\end{equation}
where $\xi'\in\R^{d-1}$.

We distinguish two cases:\\
\\
\textbf{Case 1:} The polynomial function $h$ has two distinct roots with positive imaginary parts, let say $r_1$ and $r_2.$ Having the \LS for $(Q,B_1,B_2)$ holds at $x$ is equivalent to the following condition
\begin{equation}\label{EE1}
    \begin{vmatrix}
      b_1(x,\xi', \z=r_1)&~~b_2(x, \xi',\z=r_1)\\
      \\
      b_1(x, \xi',\z=r_2)&~~b_2(x, \xi',\z=r_2)
    \end{vmatrix}\neq 0.
\end{equation}
Using the multi-linearity of the determinant together with (\ref{symb1}), we find that condition (\ref{EE1}) is equivalent to have
\begin{align*}
(r_2-r_1)\begin{vmatrix}
b_1^0(x,\xi') &~~~~ b_2^0(x,\xi')\\
\\
b_1^1(x,\xi') &~~ b_2^1(x,\xi')
\end{vmatrix}\neq 0.
\end{align*}
Thus it suffices to have
\begin{equation}\label{lspini}
\begin{vmatrix}
b_1^0(x,\xi') &~~~~ b_2^0(x,\xi')\\
\\
b_1^1(x,\xi') &~~ b_2^1(x,\xi')
\end{vmatrix}\neq 0
\end{equation}
since $r_1$ and $r_2$ are distinct and different from zero for $\sigma\neq 0$.
Now we aim to show that
\begin{equation*}\label{EE2}
\mathcal{A}=
    \begin{vmatrix}
      b_1(x,\xi'+\zeta', \z=r_1+\delta)&~~b_2(x, \xi'+\zeta',\z=r_1+\delta)\\
      \\
      b_1(x, \xi'+\zeta',\z=r_2+\tilde{\delta})&~~b_2(x, \xi'+\zeta',\z=r_2+\tilde{\delta})
    \end{vmatrix}\neq 0,
\end{equation*}
where $\zeta'\in C^{d-1}$ and $\delta,\tilde{\delta}\in\C,$ with $|\delta|+|\zeta'|<|r_1|, |\tilde{\delta}|+|\zeta'|<|r_2|.$ Using again, the multi-linearity of the determinant, we find that
\begin{equation*}\label{lspapc}
\mathcal{A}=
[(r_2-r_1)+(\tilde{\delta}-\delta)]\begin{vmatrix}
b_1^0(x,\xi'+\zeta') &~~~~ b_2^0(x,\xi'+\zeta')\\
\\
b_1^1(x,\xi'+\zeta') &~~ b_2^1(x,\xi'+\zeta)
\end{vmatrix}.
\end{equation*}
So, it is suffices to have
\begin{equation*}
\begin{vmatrix}
b_1^0(x,\xi'+\zeta') &~~~~ b_2^0(x,\xi'+\zeta')\\
\\
b_1^1(x,\xi'+\zeta') &~~ b_2^1(x,\xi'+\zeta)
\end{vmatrix}\neq 0.
\end{equation*}
This follows from (\ref{lspini}).\\
\\
\textbf{Case 2:} The polynomial function $h$ has a double root, let say $r$ and we have
\begin{align*}
     &\begin{vmatrix}
    b_1(x,\xi'+\zeta',\z=r+\delta) & b_2(x,\xi'+\zeta', \z=r+\delta)\\
    \\
    \partial_{\z}b_1(x,\xi'+\zeta', \z=r+\delta) & \partial_{\z}b_2(x,\xi'+\zeta', \z=r+\delta)
\end{vmatrix}\\
    &\qquad \ \quad \quad \quad\quad = \begin{vmatrix}
b_1^1(x,\xi'+\zeta')(r+\delta)+b_1^0(x,\xi'+\zeta') &~~~~ b_2^1(x,\xi'+\zeta')(r+\delta)+b_2^0(x,\xi'+\zeta') \\
\\
b_1^1(x,\xi'+\zeta') &~~ b_2^1(x,\xi'+\zeta')
\end{vmatrix}\\
\\
    & \qquad \ \quad \quad \quad\quad =
     (r+\delta)\underbrace{\begin{vmatrix}
 b_1^1(x,\xi'+\zeta')&    b_2^1(x,\xi'+\zeta')\\
 \\
 b_1^1(x,\xi'+\zeta') &~~ b_2^1(x,\xi'+\zeta')
\end{vmatrix}}_{=0}+\underbrace{\begin{vmatrix}
b_1^0(x,\xi'+\zeta') &~~~~ b_2^0(x,\xi'+\zeta')\\
\\
b_1^1(x,\xi'+\zeta') &~~ b_2^1(x,\xi'+\zeta')
\end{vmatrix}}_{\neq 0~~\text{thanks to \textbf{Case 1}}}\neq 0.
\end{align*}
\end{proof}

\subsection{The \LS condition after conjugation}\label{sec: LS condition for conjugated
  bilaplace} Here, we study the \LS condition for the augmented operator $Q$ after conjugation.
\subsubsection{The augmented operator}
We conjugate the operator $Q=D_s^4+\Delta^2$ with $\varphi\in\Cinf(\R^N, \R)$. We shall refer to $\varphi$ as the Carleman weight function. We set $A=-\Delta$ and we denote by $A_{\varphi}$ its conjugate. We consider a weight function that depends  on the variable $s$ and $x$ with $x=(x_d,x').$ We also write $$Q=Q_2Q_1,~~\text{with}~~Q_j=(-1)^jiD_s^2+A.$$
In normal geodesic coordinates, setting $Q_{\varphi}=e^{\tau\varphi}Qe^{-\tau\varphi}$ (with $\tau>0$) we have
\begin{equation}\label{lsaugm1}
  Q_{\varphi}=L_2L_1~~\text{with}~~L_j=e^{\tau\varphi}Q_je^{-\tau\varphi}=(-1)^ji(D_s+i\tau\partial_s\varphi)^2+A_{\varphi}
\end{equation}
with $$A_{\varphi}=e^{\tau\varphi}Ae^{-\tau\varphi}=(D_d+i\tau\partial_d\varphi(x))^2+R(x,D'+i\tau d_{x'}\varphi(x)),~~x=(x_d,x').$$
In fact, for $j=1,2$, we write $L_j$ in the following form
\begin{equation}\label{lsaugmt2}
  L_j= (D_d+i\tau\partial_d \varphi)^2+\Gamma_j,~~~~\Gamma_j=(-1)^ji(D_s+i\tau\partial_s\varphi)^2+R(x,D'+i\tau d_{x'}\varphi).
\end{equation}
For $j=1,2,$ we denote the principal symbols of $L_j$ and $\Gamma_j$ by $\ell_j$ and $\gamma_j$ respectively, which gives with $\varrho=(s, x,\sigma,\xi,\tau)$
\begin{equation*}
    \ell_j(\varrho)=(\xi_d+i\tau\partial_d\varphi)^2+\gamma_j(\varrho'),~~~\varrho'=(s, x,\sigma,\xi',\tau)
\end{equation*}
where
\begin{equation*}
   \gamma_j(\varrho')=(-1)^ji(\sigma+i\tau\partial_s\varphi)^2+r(x,\xi'+i\tau d_{x'}\varphi).
\end{equation*}
We now study the roots of $\ell_j$ viewed as a polynomial in the variable $\xi_d,$ with the other variables acting as parameters. 

\subsubsection{Roots properties and configuration for each factor}
\label{sec: Root configuration for each factor}
We consider the factors $\xi_d\mapsto
\ell_j(\varrho)$ and we recall that 
\begin{align*}
  \ell_j(\varrho)
  =(\xi_d+i\tau\partial_d\varphi)^2
  +\gamma_j(\varrho'),
  \ \ j=1,2.
\end{align*}

First, we  consider the case  $\gamma_j(\varrho') \in \R^-$, that is, equal to $-
\beta^2$ with $\beta\in \R$ . 
Then, the roots of $\xi_d\mapsto
\ell_j(\varrho)$ are given by
\begin{align*}
  -i\tau\partial_d\varphi + \beta
  \quad \text{and} \quad
  -i\tau\partial_d\varphi - \beta.
\end{align*}
Both lie in the lower complex open half-plane.

Second, we consider the case $\gamma_j(\varrho') \in\C\setminus
\R^-$. There exists a unique $\alpha_j \in \C$ such
that 
\begin{equation*}
 \Re\alpha_j> 0~~\text{and}~~\alpha_j^2:=\alpha_j^2(\y')= \gamma_j(\varrho'),~~~~j=1,2. 
\end{equation*}
We have 
\begin{align}\label{Ee1}\notag
  \alpha_j^2 &= r(x,\xi'+i\tau d_{x'}\varphi) + (-1)^ji(\sigma+i\tau\partial_s\varphi)^2)\\ \notag
             &= r(x,\xi') - \tau^2 r(x, d_{x'}\varphi)
               +   2i \tau \tilde{r}(x,\xi',d_{x'}\varphi)+(-1)^ji(\sigma+i\tau\partial_s\varphi)^2,\\ 
               &=r(x,\xi') - \tau^2 r(x, d_{x'}\varphi)
               +   2i \tau \tilde{r}(x,\xi',d_{x'}\varphi)+(-1)^ji\sigma^2+2(-1)^{j+1}\sigma\tau\partial_s\varphi+(-1)^{j+1}i(\tau\partial_s\varphi)^2
\end{align}
where $\tilde{r}(x,.,.)$ denotes the symmetric bilinear form associated with the quadratic form $r(x,.)$.
We may then write 
\begin{equation*}
  \ell_j(\y)=(\xi_d+i\tau\partial_d\varphi)^2+\alpha_j^2=(\xi_d-\pi_{j,+}(\y'))(\xi_d-\pi_{j,-}(\y')),
\end{equation*}
with
\begin{equation}\label{lsaugmt7}
\pi_{j,\pm}(\y')=-i\tau\partial_d\varphi\pm i\alpha_j(\y'),~~j=1,2.
\end{equation}
If $\tau\partial_d\varphi>C_1>0,$ $\Im \pi_{j,-}$ remains always in the complex lower half-plane while $\pi_{j,+}$ may cross the real line.
\begin{remark}\label{doubleroot}
Suppose that $\tau\partial_d\varphi\geq 0.$ Let $j=1$ or $2$. We have $$\pi_{j,-}(\y')=\pi_{j,+}(\y')\Leftrightarrow\pi_{j,-}(\y')=\pi_{j,+}(\y')=-i\tau\partial_d\varphi\Leftrightarrow\gamma_j(\y')=0.$$
\end{remark}

One has $\Im \pi_{j,-} <0$ since $\partial_d\varphi \geq C_1>0$. 
With $\Im \pi_{j,+} = -\tau\partial_d\varphi + \Re \alpha_j$ one sees that
the sign of $\Im \pi_{j,+}$ may change. The following lemma gives an algebraic characterization of the sign of  $\Im \pi_{j,+}$.

\medskip
\noindent We recall the definition of $\rho_j$ for $j=1,2$ is given in \eqref{rac1} and \eqref{rac2}.
\begin{lemma}
  \label{lemma: caracterisation Im pi2 <0}
  Assume that $\partial_d \varphi>0$, there exists $K_0>0$  sufficiently small such that $|\partial_s\varphi|+|d_{x'}\varphi|_x\leq K_0|\partial_d\varphi|$. There exist $C_1$ and $C_2$ such that, for $j=1,2,$  if we have
  $$(\tau\partial_d\varphi)^4\geq C_1|\rho_j|^3(\tau\partial_d\varphi)+ C_2 (\tau \partial_d\varphi)^2|\rho_j|^2+\sigma^4+(\tau\partial_s\varphi)^4.$$ 
then 
  $\Im \pi_{j,+}<0$, for $j=1,2$. 
\end{lemma}
\begin{proof}
  From~(\ref{lsaugmt7}) one has $\Im \pi_{j,+} <0$ if and only if
$\Re \alpha_j < \tau \partial_d \varphi = | \tau \partial_d\varphi|$, that is, if
and only if 
\begin{align*}
  4 (\tau \partial_d\varphi)^2 \Re \alpha_j^2-4 (\tau \partial_d\varphi)^4+(\Im
  \alpha_j^2)^2 < 0,
\end{align*}
by Lemma~\ref{lem1} below. With \eqref{Ee1}, we have
\begin{equation*}
\begin{cases}
        \Re\alpha_j^2  &=r(x,\xi')-r(x,\tau d_{x'}\varphi)-2(-1)^j\sigma(\tau\partial_s\varphi) \\
        (\Im\alpha_j^2)^2&=4\tau^2\tilde{r}(x,\xi',d_{x'}\varphi)^2+(\tau\partial_s\varphi)^4+\sigma^4+4(-1)^j\sigma^2\tau \tilde{r}(x,\xi',d_{x'}\varphi)\\&-4(-1)^j(\tau\partial_s\varphi)^2\tau \tilde{r}(x,\xi',d_{x'}\varphi)-2\sigma^2(\tau\partial_s\varphi)^2.
     \end{cases}
\end{equation*}
Therefore the inequality $$ 4 (\tau \partial_d\varphi)^2 \Re \alpha_j^2-4 (\tau \partial_d\varphi)^4+(\Im
  \alpha_j^2)^2 < 0$$ is equivalent to
\begin{align}  \label{eqsimpl}   
    &  4(\tau\partial_d\varphi)^2r(x,\xi')+4\tau^2\tilde{r}(x,\xi',d_{x'}\varphi)^2+\sigma^4+(\tau\partial_s\varphi)^4  \notag  \\
   &\quad   <4(\tau\partial_d\varphi)^4+4(\tau\partial_d\varphi)^2r(x,\tau d_{x'}\varphi)+8(-1)^j(\tau\partial_d\varphi)^2\sigma(\tau\partial_s\varphi)+
      2\sigma^2(\tau\partial_s\varphi)^2 \notag\\
  &\qquad    -4(-1)^j\sigma^2\tau \tilde{r}(x,\xi',d_{x'}\varphi)+4(-1)^j(\tau\partial_s\varphi)^2\tau \tilde{r}(x,\xi',d_{x'}\varphi) \notag  \\
      &\quad< 4(\tau\partial_d\varphi)^2|\tau d_x\varphi|_x^2+8(-1)^j(\tau\partial_d\varphi)^2\sigma(\tau\partial_s\varphi) +2\sigma^2(\tau\partial_s\varphi)^2-4(-1)^j\sigma^2\tau \tilde{r}(x,\xi',d_{x'}\varphi) \notag \\
  &\qquad         +4(-1)^j(\tau\partial_s\varphi)^2\tau \tilde{r}(x,\xi',d_{x'}\varphi)  
  \end{align}   
using the fact that $|\tau d_x\varphi|_x^2=|\tau d_{x'}\varphi|_x^2+(\tau\partial_d\varphi)^2.$
 On the one hand, since $|d_{x'} \varphi| \leq K_0 |\partial_d
  \varphi|$ one has 
\begin{align}
	\label{est: 1}
  4(\tau\partial_d\varphi)^2r(x,\xi')+4\tilde{r}(x,\xi',\tau d_{x'}\varphi)^2+\sigma^4+(\tau\partial_s\varphi)^4
  &\leq (4+4K_0^2) (\tau\partial_d\varphi)^2 |\xi'|_x^2+\sigma^4+(\tau\partial_s\varphi)^4   \notag\\
  &\leq (4+4K_0^2) (\tau\partial_d\varphi)^2|\rho_j|^2+\sigma^4+(\tau\partial_s\varphi)^4 , 
\end{align}
since $|\rho_j|^2=(\sigma^4+|\xi'|_x^4)^{1/2}\geq |\xi'|_x^2$, $r(x,\xi')=|\xi'|^2_x$ and $|\tilde{r}(x,\xi',\tau d_{x'}\varphi)^2|\leq |\xi'|^2_x|\tau d_{x'}\varphi|_x^2.$
On the other hand one has
\begin{align*}
&4(\tau\partial_d\varphi)^2|\tau d_x\varphi|^2+8(-1)^j(\tau\partial_d\varphi)^2\sigma(\tau\partial_s\varphi)+2\sigma^2(\tau\partial_s\varphi)^2
-4(-1)^j\sigma^2\tau \tilde{r}(x,\xi',d_{x'}\varphi)\\
&\quad -4(-1)^j(\tau\partial_s\varphi)^2\tau \tilde{r}(x,\xi',d_{x'}\varphi) \\
&\qquad  \geq 4(\tau\partial_d\varphi)^4-8(\tau\partial_d\varphi)^2|\sigma||\tau\partial_s\varphi| 
    -4\sigma^2|\tau d_{x'}\varphi|_x|\xi'|_x-4|\tau d_{x'}\varphi|_x|\xi'|_x(\tau\partial_s\varphi)^2\\
    & \qquad \geq 4(\tau\partial_d\varphi)^4-8K_0(\tau\partial_d\varphi)^3|\rho_j|
    -4K_0|\rho_j|^3(\tau\partial_d\varphi)-4K^3_0|\rho_j|(\tau\partial_d\varphi)^3\\
    & \qquad \geq 4(\tau\partial_d\varphi)^4-(8K_0+4K^3_0)(\tau\partial_d\varphi)^3|\rho_j|
    -4K_0|\rho_j|^3(\tau\partial_d\varphi).
\end{align*}
Using $xy^2\leq \frac 13 x^3+\frac 23 y^3$, we then have
\begin{align} 
	\label{est: 2}
&4(\tau\partial_d\varphi)^4-(8K_0+4K^3_0)(\tau\partial_d\varphi)^3|\rho_j| 
-4K_0|\rho_j|^3(\tau\partial_d\varphi)  \notag \\
&\qquad \qquad \qquad\geq 4(\tau\partial_d\varphi)^4-(8K_0+4K^3_0)(\frac 13 |\rho_j|^3
+\frac 23 (\tau\partial_d\varphi)^3)(\tau\partial_d\varphi)  
-4K_0|\rho_j|^3(\tau\partial_d\varphi)  \notag  \\
&\qquad \qquad \qquad \geq (\tau\partial_d\varphi)^4-C K_0|\rho_j|^3(\tau\partial_d\varphi)
\end{align}
for $K_0$ sufficiently small  and some $C>0.$  From \eqref{est: 1} 
and \eqref{est: 2}, \eqref{eqsimpl} holds if one has
$$(\tau\partial_d\varphi)^4-C K_0|\rho_j|^3(\tau\partial_d\varphi)\geq (4+4K_0^2) (\tau \partial_d\varphi)^2|\rho_j|^2+\sigma^4+(\tau\partial_s\varphi)^4,$$ that is   
$$(\tau\partial_d\varphi)^4\geq C K_0|\rho_j|^3(\tau\partial_d\varphi)+ (4+4K_0^2) (\tau \partial_d\varphi)^2|\rho_j|^2+\sigma^4+(\tau\partial_s\varphi)^4.$$ This ends the proof.
\end{proof}
\begin{lemma}
  \label{lem1}
Let $z\in\C$ such that $m=z^2.$ For $x_0\in\R$ such
that $x_0\neq 0$, we have 
\begin{align*}
  |\Re z|\lesseqqgtr |x_0|
  \qquad \Longleftrightarrow
  \qquad 4x_0^2 \Re m-4x_0^4+(\Im m)^2
  \lesseqqgtr0.
  \end{align*}
\end{lemma}
\begin{proof}
Let $z=x+iy\in\C.$ On the one hand we have
$z^2=x^2-y^2+2ixy=m$ and $\Re m=x^2-y^2,$  $\Im m=2xy.$ On the other
hand we have 
\begin{align*}
  4x_0^2\Re m-4x_0^4+(\Im m)^2
  &=4x_0^2(x^2-y^2)-4x_0^4+4x^2y^2\\
  &=4(x_0^2+y^2)(x^2-x_0^2),
\end{align*}
thus with the same sign as $(x^2-x_0^2)$.
Since $|\Re z|\lesseqqgtr |x_0| \ \Leftrightarrow \ 
x^2-x_0^2 \lesseqqgtr 0$ the conclusion follows.
\end{proof}
\begin{lemma}\label{low-frequency}
 There exists $K_0>0$  and $C>0$ such that if  $|\partial_s\varphi|+|d_{x'}\varphi|\leq K_0|\partial_d\varphi|$ 
 and
  $C |\rho_j|\leq (\tau\partial_d\varphi),$ $j=1,2$ then
 $\Im\pi_{j,+}(x,s,\xi',\sigma,\tau)<0$.
\end{lemma}
\begin{proof}
     It is sufficient to prove the estimate of  Lemma~\ref{lemma: caracterisation Im pi2 <0}.
From assumption we have 
\[
5  (\tau\partial_d\varphi)^4\geq C|\rho_j|^3(\tau\partial_d\varphi)+ C^2 (\tau \partial_d\varphi)^2|\rho_j|^2  +C^4|\rho_j|^4
+C^4\sigma^4 + (\tau\partial_s\varphi)^4,
\]  
if $K_0$ is sufficiently small and as $|\rho_j|^4=\sigma^4+|\xi'|^4_x$.
This estimate implies the one of Lemma~\ref{lemma: caracterisation Im pi2 <0} as $K_0$ can be chosen sufficiently 
small and $C$ sufficiently large.
\end{proof}

We consider the boundary differential operators
$B_1$ and $B_2$ of order $k_1$ and $k_2$ with $b_j(x,\xi)$ their
principal symbol, $j=1,2$. The associated conjugated operators 
\begin{align*}
  B_{j,\varphi} = e^{\tau \varphi} B_j e^{-\tau \varphi},
\end{align*}
have respective principal symbols 
\begin{align*}
  b_{j,\varphi} (s,x,\xi, \tau) = b_j(x,\xi+i\tau d_{z}\varphi),\quad j=1,2,~~~z=(s,x).
\end{align*}
We assume that the \LS condition holds for $(Q, B_1, B_2)$ as in
Definition \ref{def: LS} for any point $(s, x,\sigma,\xi')\in\R\times
\partial\Omega\times\R\times\R^{d-1}.$ We wish to know if the \LS condition can hold
for $(Q , B_1, B_2,\varphi),$ as given by the following
definition (in local coordinates for simplicity).
\begin{definition}\label{def: LS after conjugation}
  Let
  $(s,x,\sigma, \xi',\tau)\in\times\R\times
  \partial\Omega\times\R\times\R^{d-1}\times[0,+\infty)$ with
  $(\sigma,\xi',\tau)\neq(0,0,0)$. One says that the \LS condition holds for
  $(Q ,B_1, B_2, \varphi)$ at $(s,x,\sigma, \xi',\tau)$ if for
  any polynomial function $f(\xi_d)$ with complex coefficients there
  exist $c_1, c_2\in\C$ and a polynomial function
  $g(\xi_d)$ with complex coefficients such that, for all
  $\xi_d\in\C$
  \begin{equation*}
    f(\xi_d)=c_1 b_{1,\varphi}(s, x,\xi',\xi_d,\tau)
    +c_2b_{2,\varphi}(s,x,\xi',\xi_d,\tau)
    +g(\xi_d) q_{\varphi}^+(s, x,\sigma,\xi',\xi_d,\tau),
\end{equation*}
with
\begin{equation*}
 q_{\varphi}^+(s,x,\sigma,\xi',\xi_d,\tau) 
  =\prod\limits_{\Im r_j(x, \sigma,\xi',\tau)\geq
    0}(\xi_d-r_j(x, \sigma,\xi',\tau)),
\end{equation*}
where $r_j(x,\sigma,\xi',\tau),$ $j=1,\cdots,4,$ denote the complex
roots of $q_{\varphi}(s,x,\sigma,\xi',\xi_d,\tau)$ viewed as a polynomial in
$\xi_d$.
\end{definition}
\subsubsection{Discussion on the Lopatinski\u{\i}-\v{S}apiro condition
  according to the position of the roots}
\label{sec: discussion LS root positions}
With the assumption that $\partial_d\varphi>0,$ for any point
$(s, x,\sigma,\xi',\tau)$ at most two roots lie in the upper complex closed half-plane. We then enumerate the following cases. Here, we drop the dependence of the roots of the polynomial $q_{\varphi}^+(s,x,\sigma,\xi',\xi_d,\tau) $ on the variables $(x,\sigma,\xi',\tau).$
\begin{enumerate}
\item[•] Case 1 : No root lying in the upper complex closed
  half-plane, then $q_{\varphi}^+ (s,x,\sigma,\xi',\xi_d,\tau)=1$ and the
  \LS condition of Definition~\ref{def: LS after conjugation} holds trivially at $(s, x,\sigma,\xi',\tau)$.
\item[•] Case  2 : One root lying in the upper complex closed
  half-plane. Let us denote by  $r^+$ that root, then
  $q_{\varphi}^+ (s,x,\sigma,\xi',\xi_d,\tau)=\xi_d-r^+.$ With
  Definition~\ref{def: LS after conjugation}, for any choice of $f$, the polynomial
  function $\xi_d\mapsto
  f(\xi_d)-c_1b_{1,\varphi}(s,x,\xi',\xi_d,\tau)-c_2b_{2,\varphi}(s,x,\xi',\xi_d,\tau)$
  admits $r^+$ as a root for $c_1,c_2\in\C$ well
  chosen. Hence, the \LS condition holds  at
  $(s, x,\sigma,\xi',\tau)$ if and only
  if
  \begin{align*}
    b_{1,\varphi}(s,x,\xi',\xi_d=r^+,\tau) \neq 0
    \ \ \text{or} \ \
    b_{2,\varphi}(s,x,\xi',\xi_d=r^+,\tau)\neq 0.
  \end{align*}
\item[•] Case 3 : Two different roots lying in the upper complex
  closed half-plane. Let denote by $r^+_1$ and $r^+_2$ these
  roots. One has
  $q_{\varphi}^+ (s,x,\sigma,\xi',\xi_d,\tau)=(\xi_d-r^+_1)(\xi_d-r^+_2).$
  The \LS condition holds at $(s, x,\sigma,\xi',\tau)$ if and only
  if the complex numbers $r^+_1$ and $r^+_2$ are the roots of
  the polynomial function
  $\xi_d\mapsto
  f(\xi_d)-c_1b_{1,\varphi}(s,x,\xi',\xi_d,\tau)-c_2b_{2,\varphi}(s, x,\xi',\xi_d,\tau),$
  for $c_1,c_2$ well chosen. This reads
\begin{equation*}
\begin{cases}
f(r^+_1)= c_1b_{1,\varphi}(s,x,\xi',\xi_d=r^+_1,\tau)+c_2b_{2,\varphi}(s,x,\xi',\xi_d=r^+_1,\tau),\\
\\
f(r^+_2)= c_1b_{1,\varphi}(s,x,\xi',\xi_d=r^+_2,\tau)+c_2b_{2,\varphi}(s,x,\xi',\xi_d=r^+_2,\tau).
\end{cases}
\end{equation*}
Being able to solve this system in $c_1$ and $c_2$ for any $f$ is equivalent to having
\begin{equation*}
  \det 
\begin{pmatrix}
b_{1,\varphi}(s,x,\xi',\xi_d=r^+_1,\tau) & b_{2,\varphi}(s,x,\xi',\xi_d=r^+_1,\tau) \\[4pt]
b_{1,\varphi}(s,x,\xi',\xi_d=r^+_2,\tau) & b_{2,\varphi}(s,x,\xi',\xi_d=r^+_2,\tau)
\end{pmatrix}\neq 0.
\end{equation*}
 \item[•] Case 4 : A double root lying in the upper complex closed
  half-plane. Denote by $r^+$ this root; one has
  $q_{\varphi}^+(s,x,\sigma,\xi',\xi_d,\tau))=(\xi_d-r^+)^2.$ The
  \LS condition holds at  at $(x,\sigma,\xi',\tau)$ if and only if for any choice of $f$, the
  complex number $r^+$ is a double root of the polynomial function
  $\xi_d\mapsto f(\xi_d)-c_1
  b_{1,\varphi}(s,x,\xi',\xi_d,\tau)-c_2b_{2,\varphi}(s,x,\xi',\xi_d,\tau)$
  for some $c_1,c_2\in\C.$ Thus having the
  \LS condition is equivalent of having the following
  determinant condition,
  \begin{equation*}
    \det 
\begin{pmatrix}
b_{1,\varphi}(s,x,\xi',\xi_d=r^+,\tau) & b_{2,\varphi}(s,x,\xi',\xi_d=r^+,\tau) \\
\\
\partial_{\xi_d}b_{1,\varphi}(s,x,\xi',\xi_d=r^+,\tau) & \partial_{\xi_d}b_{2,\varphi}(s,x,\xi',\xi_d=r^+,\tau)
\end{pmatrix}\neq 0.
\end{equation*}
\end{enumerate}
Observe that case 4 can only occur if $\sigma=0$ and $\tau\partial_s\varphi=0$. 
In this case Lopatinski\u{\i}-\v{S}apiro condition for $(Q,B_1,B_2,\varphi)$ is equivalent to 
Lopatinski\u{\i}-\v{S}apiro condition for $(P,B_1,B_2)$ which has been assumed.
\\
\indent To carry on with the proof of  Proposition~\ref{prop: LS after conjugation} below, we
now only have to consider having 
\begin{align}
  \label{eq: condition tau small}
  \tau \partial_d \varphi \leq C  |\rho_j|,
\end{align}
for $j=1,2$ and $C>0.$
\begin{proposition}
  \label{prop: LS after conjugation}
  Let $x^0\in\partial\Omega$ and $s^0\in\R$. Assume that the \LS condition holds for
  $(Q, B_1, B_2)$ at $(s^0,x^0)$ and thus in a compact neighborhood $V^0$
  of $(s^0,x^0)$. Assume also that
  $\partial_d\varphi\geq C_1> 0$ in $V^0$. There exist $\mu_0>0$ if
  $(s,x,\sigma,\xi',\tau)\in\R\times V^0\times\times\R^{d-1}\times[0,+\infty) $ with
  $(\sigma,\xi',\tau)\neq (0,0, 0)$,
  \begin{align*}
   |\partial_s\varphi(s,x)|+|d_{x'}\varphi(s,x)|\leq \mu_0\partial_d\varphi,
\ \   \text{and} \  \
\Im \pi_{j,+}(s,x,\sigma,\xi',\tau)\geq 0
  \end{align*}
    then the \LS condition holds for $(Q , B_1, B_2, \varphi)$ at
  $(s,x,\sigma,\xi',\tau).$
\end{proposition}
\indent The proof of Proposition \ref{prop: LS after conjugation} relies on the following lemma.
\begin{lemma}
  \label{lemma: small perturbation}
  There exists $C>0$ such that for  $0<\mu_0\leq 1$, one has
\begin{equation*}
  |\partial_s\varphi(s,x)|+|d_{x'}\varphi(s,x)|\leq \mu_0\partial_d\varphi,
\ \   \text{and} \  \
\Im \pi_{j,+}(s,x,\xi',\sigma,\tau)\geq 0  
\Rightarrow  
\big| i \alpha_j-\rho_j\big|+\tau|d_{x'}\varphi|   
\leq  (1+C)\mu_0 |\rho_j|.
\end{equation*}
\end{lemma}
We recall that $\alpha_j$ are defined at the beginning of Section~\ref{sec: Root configuration for each factor}, and
$\rho_j$ at \eqref{rac1} and  \eqref{rac2}.  
\begin{proof}
 For $j=1,2$
and $\ell=1,2$ we write 
\begin{align}
 \label{lop6}
  b_{\ell,\varphi} (s,x,\xi',\xi_d =\pi_{j,+},\tau)
  &=b_\ell(x,\xi'+i\tau d_{x'}\varphi, \pi_{j,+}+i\tau\partial_d\varphi) 
    = b_\ell(x,\xi'+i\tau d_{x'}\varphi, i\alpha_j)\notag\\
  &=b_\ell(x,\xi'+i\tau d_{x'}\varphi, i\tilde{\rho}_j+i(\alpha_j-\tilde{\rho}_j)),
\end{align}
where $\rho_j=i\tilde{\rho}_j.$
With the assumption that $\Im\pi_{j,+}\geq 0,$ we have $\Re(\alpha_j)\geq (\tau\partial_d\varphi)$ and $$4(\tau\partial_d\varphi)^2\Re\alpha_j^2-4(\tau\partial_d\varphi)^4+(\Im(\alpha^2_j))^2\geq 0.$$ 
With (\ref{eq: condition tau small}), one has 
\begin{align}
    \label{eq: est 1 small perturbation}
    \tau|d_{x'}\varphi| \leq \mu_0 \tau \partial_d\varphi
    \lesssim \mu_0  |\tilde{\rho}_j|.
  \end{align}
  With the first-order Taylor formula one has
  \begin{align*}
    \alpha_j^2 &= r(x,\xi'+i\tau d_{x'}\varphi)+(-1)^ji(\sigma+i\tau\partial_s\varphi)^2\\ 
               &= r(x,\xi')
                 +\int_{0}^1 d_{\xi'}r(x,\xi'+i \tau \theta\, d_{x'}\varphi)
                 (i \tau d_{x'}\varphi)d\theta
                 +\underbrace{(-1)^ji\sigma^2-2(-1)^j\sigma(\tau\partial_s\varphi)-(-1)^j(\tau\partial_s\varphi)^2}_{:=A}.
  \end{align*}
  With \eqref{eq: est 1 small perturbation} and 
  as $d_{\xi'} $ is linear with respect the second variable one  has
  \begin{align*}
    \big|
    d_{\xi'}r(x,\xi'+i \tau \theta\, d_{x'}\varphi)  (i \tau d_{x'}\varphi)
    \big|
    \lesssim \mu_0  |\xi'|_x|\Tilde{\rho}_j|+\mu_0^2|\Tilde{\rho}_j|^2 \lesssim \mu_0|\Tilde{\rho}_j|^2,
  \end{align*}
since $|\xi'|_x\leq |\Tilde{\rho}_j|.$ On the one hand, 
 since $r(x,\xi') = |\xi'|_x^2$ and $A\lesssim \mu_0|\rho|^2_j$, this yields
  $$\alpha_j^2  = \Tilde{\rho}_j^2+  O (\mu_0 |\Tilde{\rho}_j|^2)\Longleftrightarrow \alpha^2_j-\Tilde{\rho}_j^2= O (\mu_0 |\Tilde{\rho}_j|^2).$$ 
  On the other hand as $\Re \alpha_j>0$ and $\Re \Tilde{\rho}_j\ge 0$, 
 $|\alpha_j+\Tilde{\rho}_j|\geq |\Re(\alpha_j+\Tilde{\rho}_j)|\geq \Re(\Tilde{\rho}_j)\geq \frac{\sqrt 2}{2} |\Tilde{\rho}_j|$. 
 Therefore 
 $$\alpha^2_j-\Tilde{\rho}_j^2=(\alpha_j-\Tilde{\rho}_j)(\alpha_j+\Tilde{\rho}_j)=O (\mu_0 |\Tilde{\rho}_j|^2)~~\text{and}~~|\alpha_j-\Tilde{\rho}_j|=O (\mu_0 |\Tilde{\rho}_j|).$$ 
 This and (\ref{eq: est 1 small perturbation}) give the result.
\end{proof}
We now consider the root configurations that remain to consider according to the
discussion in Section \ref{sec: discussion LS root positions}.\\
 We choose $0< \mu_0\leq 1$ such that
  \begin{align}
    \label{eq: uniform condition mu0 mu1 small}
    \mu_0 (1+C) \leq \eps,
  \end{align}
with $C>0$ as given by Lemma~\ref{lemma: small perturbation}.
\\

\noindent 
{\bfseries Case 2.}\\
In this case, one root of $q_{\varphi}$ lies in the upper complex closed
half-plane. We denote this root by $\rho^+$. According to the
discussion in Section~\ref{sec: discussion LS root positions} it
suffices to prove that
\begin{align}\label{ee1}
    b_{1,\varphi}(s,x,\xi',\xi_d=r^+,\tau) \neq 0
    \ \ \text{or} \ \
    b_{2,\varphi}(s,x,\xi',\xi_d=r^+,\tau)\neq 0.
  \end{align}
  In fact, one has $r^+ = \pi_{j,+}$ with $j=1$ or $2$.  
As $(Q,B_1,B_2)$  satisfies the Lopatinski\u{\i}-\v{S}apiro,
and from Lemma~\ref{lemma: small perturbation} we have $|\rho_j-\pi_{j,+}|\lesssim \mu_0 |\rho_j|$.
 From Lemma~\ref{lemma: perturbation LS} Lopatinski\u{\i}-\v{S}apiro is also satisfied at $|\pi_{j,+}|$. 
 This means that \eqref{ee1} is satisfied.
  \\
 
  \noindent 
  {\bfseries Case 3.}\\
  In this case $\Im \pi_{1,+}>0$ and $\Im \pi_{2,+} >0$.  According to the
  discussion in Section~\ref{sec: discussion LS root positions} it
  suffices to prove that
  \begin{equation}\label{eq: LS case 3}
  \det 
\begin{pmatrix}
  b_{1,\varphi}(s,x,\xi',\xi_d=\pi_{1,+},\tau) & b_{2,\varphi}(s,x,\xi',\xi_d=\pi_{1,+},\tau) \\
  \\
  b_{1,\varphi}(s,x,\xi',\xi_d=\pi_{2,+},\tau) & b_{2,\varphi}(s,x,\xi',\xi_d=\pi_{2,+},\tau)
\end{pmatrix}\neq 0.
\end{equation}
We use again
  Lemma~\ref{lemma: perturbation LS} with
  $\zeta' = i \tau d_{x'} \varphi$, 
  $\delta = i(\alpha_1 -  \Tilde{\rho}_j)$, and $\tilde{\delta} = i (\alpha_2 -  \Tilde{\rho}_j)$. With \eqref{lop6} and \eqref{eq:
    uniform condition mu0 mu1 small} with Lemma~\ref{lemma:
    perturbation LS} one obtains \eqref{eq: LS case 3}.
    \\
    
  \noindent 
  {\bfseries Case 4.}\\
  In this case (that only occurs if $\sigma=0$) the \LS condition holds also if one has 
\begin{equation}
  \label{eq: LS case 2}
    \det
    \begin{pmatrix}
      b_{1,\varphi} &b_{2,\varphi}\\[2pt]
      \partial_{\xi_d}b_{1,\varphi}& \partial_{\xi_d}b_{2,\varphi} 
    \end{pmatrix}(s,x,\xi',\xi_d=\rho^+,\tau) \neq 0.
  \end{equation}
The proof uses again Lemma \ref{lemma: perturbation LS} (Case 2 in the proof of the Lemma) . This concludes the proof of
  Proposition~\ref{prop: LS after conjugation}.\hfill \qedsymbol \endproof
%
%
 \subsection{Local stability of the algebraic conditions}
\label{sec: Stability LS-conj}
In Section~\ref{sec: LS condition} we saw that the \LS condition for
$(Q, B_1, B_2)$ in Definition~\ref{def: LS} exhibits some stability
property.  This was used in the statement of Proposition~\ref{prop: LS
  after conjugation} that states how the \LS condition for
$(Q, B_1, B_2)$ can imply the \LS condition of
Definition~\ref{def: LS after conjugation} for
$(Q, B_1, B_2, \varphi)$, that is, the version of this condition
for the conjugated operators.

Below one
exploits the algebraic conditions listed in Section~\ref{sec: discussion LS root positions} once the
\LS condition is known to hold at a point $\y^{0\prime}  = (s^0,x^0,\sigma^0,
\xi^{0\prime}, \tau^0,\gamma^0,\eps^0)$  where $(s^0,x^0,\sigma^0,
\xi^{0\prime)}$ is in tangential
phase space. 
By abuse of notation we write $\gamma $ and $\eps$ in $\y^{\prime}$ to take account that 
$\varphi$ depends on $\gamma $ and $\eps$. 
One thus rather needs to know that these algebraic
conditions are stable. Here also the answer is positive and is the
subject of the present section.

\medskip

As in Definition~\ref{def: LS after conjugation} for $\y' =(s, x,\xi',\sigma,\tau,\gamma,\eps)$ one denotes by
$\pi_j (\y')$
the roots of $q_\varphi(\y)$ viewed as a polynomial in
$\xi_d$.

 Let $\y^{0\prime}  = (s^0,x^0,\sigma^0,
\xi^{0\prime}, \tau^0,\gamma^0,\eps^0)\in\R\times \partial
 \Omega \times\R\times \R^{d-1}\times [0,+\infty) \times [1,+\infty)\in (0,1]$.
 One sets
 \begin{align*}
   &\mathscr{E}^+= \big\{ j \in \{1,2,3,4\}; \ \Im \pi_j (\y^{0\prime}) \geq 0\}, \\
   &\mathscr{E}^-= \big\{ j \in \{1,2,3,4\}; \ \Im \pi_j (\y^{0\prime}) < 0\},
 \end{align*}
 and, for $\y' =(s, x,\xi',\sigma,\tau,\gamma,\eps)$,
 \begin{align*}
   \k^+_{\y^{0\prime}} (\y') 
   = \prod_{j \in \mathscr{E}^+} \big(\xi_d -\pi_j(\y')\big), 
   \qquad 
   \k^-_{\y^{0\prime}} (\y') 
   = \prod_{j \in \mathscr{E}^-} \big(\xi_d -\pi_j(\y')\big). 
 \end{align*}
 Naturally, one has $\k^+_{\y^{0\prime}} (\y^{0\prime},\xi_d) = q_{\varphi}^+(s^0,x^0,\sigma^0,
 \xi^{0\prime}, \xi_d, \tau^0)$ and $\k^-_{\y^{0\prime}} (\y^{0\prime},\xi_d) = q_{\varphi}^-(s^0,x^0,\sigma^0,
 \xi^{0\prime}, \xi_d, \tau^0)$. 
 Moreover, there exists a conic open \nhd $\U_0$ of
 $\y^{0\prime}$ where both $\k^+_{\y^{0\prime}}(\y')$ and
 $\k^-_{\y^{0\prime}}(\y')$ are smooth with respect to $\y'$.
One has 
\begin{align*}
  q_{\varphi} = q_{\varphi}^+ q_{\varphi}^- 
  = \k^+_{\y^{0\prime}} \k^-_{\y^{0\prime}} .
 \end{align*}
 Note however that for $\y' =(s, x,\xi',\sigma,\tau,\gamma,\eps)\in \U_0$ it may very well happen that
 $q_{\varphi}^+(\y) \neq \k^+_{\y^{0\prime}} (\y',\xi_d)$
 and  $q_{\varphi}^-(\y) \neq \k^-_{\y^{0\prime}} (\y',\xi_d)$.

 The following proposition can be found in \cite[proposition 1.8]{ML}.
 \begin{proposition}
   \label{prop: stability LS algegraic conditions}
   Let the \LS condition hold at $\y^{0\prime}  = (s^0,x^0,\sigma^0,
\xi^{0\prime}, \tau^0,\gamma^0,\eps^0)\in\R\times \partial
 \Omega \times\R\times \R^{d-1}\times [0,+\infty) \times [1,+\infty)\times (0,1]$ for
 $(Q, B_1, B_2, \varphi)$. Then,
 \begin{enumerate}
 \item The polynomial
 $\xi_d \mapsto q_{\varphi}^+(\y^{0\prime})$ is of degree less than or equal to
 two.
 \item There exists a conic open \nhd $\U$ of $\y^{0\prime}$ such that
   the family 
   $\{b_{1,\varphi}(\y',\xi_d),b_{2,\varphi}(\y',\xi_d)\}$ is complete
   modulo $\k^+_{\y^{0\prime}} (\y',\xi_d)$ at every point $\y' \in \U$,
   namely for
  any polynomial $f(\xi_d)$ with complex coefficients there
  exist $c_1, c_2\in\C$ and a polynomial
  $\ell(\xi_d)$ with complex coefficients such that, for all
  $\xi_d\in\C$
  \begin{equation}
    \label{eq: extended LS condition}
    f(\xi_d)=c_1 b_{1,\varphi}(s,x,\xi',\xi_d,\tau)
    +c_2b_{2,\varphi}(s,x,\xi',\xi_d,\tau)
    +\ell(\xi_d) \k^+_{\y^{0\prime}}(\y',\xi_d).
\end{equation}
\end{enumerate}
\end{proposition}
We emphasize again that the second property in Proposition~\ref{prop:
  stability LS algegraic conditions} looks very much like the
statement of \LS condition for $(Q, B_1, B_2, \varphi)$ at $\y'$
in Definition~\ref{def: LS after conjugation}. Yet, it differs by
having $q_{\varphi}^+(x,\xi',\xi_d,\tau)$ that only depends on the root configuration at $\y'$ replaced
 by $ \k^+_{\y^{0\prime}}(\y',\xi_d)$ whose structure is based on the
 root configuration at $\y^{0\prime}$. 
Actually $q_{\varphi}^+$ and $ \k^+_{\y^{0\prime}}$ are only different if at $\y^{0\prime}$ one root is real.
 The polynomial $ \k^+_{\y^{0\prime}}$ is smooth with respect $\y'$ in a \nhd of $\y^{0\prime}$ but the degree of $q_{\varphi}^+$ can be changed in a \nhd of $\y^{0\prime}$.

\medskip
  Let $m^+$ be the common degree of $q_{\varphi}^+(\y^{0\prime},\xi_d)$ and
$\k^+_{\y^{0\prime}} (\y',\xi_d)$ and 
$m^-$ be the common degree of $q_{\varphi}^-(\y^{0\prime},\xi_d)$ and
$\k^-_{\y^{0\prime}} (\y',\xi_d)$ for $\y'\in \U$. One has $m^+ + m^-
= 4$ and thus $m^- \geq 2$ for $\y'\in \U$ since $m^+ \leq 2$. 

Invoking the Euclidean division of polynomials, one sees that it is
sufficient to consider polynomials $f$ of degree less than or equal to
$m^+-1\leq 1$  in \eqref{eq: extended LS condition}.
Since the degree of $b_{j,\varphi} (\y',\xi_d)$ can be as high as
$3> m^+-1$ it however makes sense to consider $f$ of degree less than or equal to
$m=3$. Then, the second property in Proposition~\ref{prop:
  stability LS algegraic conditions} is equivalent to having 
\begin{align*}
  \{ b_{1,\varphi}(s,x,\xi',\xi_d,\tau),
  b_{2,\varphi}(s,x,\xi',\xi_d,\tau)\} \cup \bigcup_{\ell=0
  }^{3- m^+} \{
  \k^+_{\y^{0\prime}}(\y',\xi_d) \xi_d^\ell\} 
\end{align*}
be a complete in the set of polynomials of degree less than or equal to
$m=3$. Note that this family is made of $m' = 6- m^+ = 2 + m^-$
polynomials. 

\medskip
We now express an inequality that follows from Proposition~\ref{prop:
  stability LS algegraic conditions} that will be key in the boundary
estimation given in Proposition~\ref{prop: local boundary
  estimate} below.

\subsection{Symbol positivity at the boundary}
\label{sec: Symbol postivity at the boundary}
The symbols $b_{j,\varphi}$, $j=1,2$, are polynomial in $\xi_d$ of
degree $k_j \leq 3$ and we
may thus write them in the form
\begin{align*}
b_{j,\varphi}(\y',\xi_d)=\sum\limits_{\ell=0}^{k_j}b_{j,\varphi}^\ell (\y')\xi_d^\ell,
\end{align*}
with $b_{j,\varphi}^\ell$ homogeneous of degree $k_j-\ell$.

The polynomial $\xi_d \mapsto \k^+_{\y^{0\prime}} (\y',\xi_d)$ is of
degree $m^+ \leq 2$ for $\y'\in \U$ with $\U$ given by Proposition~\ref{prop: stability LS
  algegraic conditions}.
Similarly, we write 
\begin{align*}
\k^+_{\y^{0\prime}} (\y',\xi_d)=\sum\limits_{\ell=0}^{m^+}\k^{+,\ell}_{\y^{0\prime}} (\y')\xi_d^\ell,
\end{align*}
with $\k^{+,\ell}_{\y^{0\prime}}$  homogeneous of degree $m^+-\ell$.
We introduce 
\begin{align*}
  e_{j, \y^{0\prime}} (\y', \xi_d) = 
\begin{cases}
  b_{j,\varphi}(\y',\xi_d) & \text{if} \ j=1,2,\\
  \k^+_{\y^{0\prime}} (\y',\xi_d) \xi_d^{j-3} & \text{if} \ j=3,\dots, m'.
\end{cases}
\end{align*}
As explained above, all these polynomials are of degree less than or
equal to three. 
If we now write 
\begin{align*}
  e_{j, \y^{0\prime}} (\y', \xi_d) 
  = \sum_{\ell=0}^{3} e_{j,
  \y^{0\prime}}^\ell (\y') \xi_d^\ell,
\end{align*}
for $j=1,2$ one has $e_{j, \y^{0\prime}}^\ell (\y') =
b_{j,\varphi}^\ell (\y')$, with $\ell =0, \dots, k_j$ and  $e_{j,
  \y^{0\prime}}^\ell (\y') =0$ for $\ell > k_j$, and for
$j=3,\dots,m'$, 
\begin{align*}
   e_{j, \y^{0\prime}}^\ell (\y') = 
  \begin{cases}
   0 & \text{if} \ \ell <  j-3,\\
   \k^{+,\ell+3-j}_{\y^{0\prime}} (\y') & \text{if} \ \ell=j-3,\dots,
   m^+ + j- 3\leq m^+ + m' -3,\\
   0 & \text{if} \ \ell > m^+ + j- 3.
  \end{cases}
\end{align*}
In particular $e_{j, \y^{0\prime}}^\ell (\y')$ is homogeneous of
degree $m^+ +j - \ell -3$. Note that $m^+ + m' -3 = 3$.
We thus have the following result. 
\begin{lemma}\label{lemma: algebraic condition LS}
  Set the $m' \times (m+1) $ matrix $M (\y') = (M_{j, \ell} (\y')
  )_{\substack{1\leq j \leq m'\\0\leq \ell \leq m}}$ with $M_{j, \ell} (\y')  = e_{j, \y^{0\prime}}^\ell (\y')$. Then, the second point
  in Proposition~\ref{prop: stability LS algegraic conditions} states
  that $M (\y') $ is of rank $m+1=4$ for $\y ' \in \U$. 
\end{lemma}
Indeed,
we have 
\begin{equation*}
    M (\y')=
    \begin{pmatrix}
    e_{1, \y^{0\prime}}^0& e_{1, \y^{0\prime}}^1&e_{1, \y^{0\prime}}^2&e_{1, \y^{0\prime}}^3\\
    e_{2, \y^{0\prime}}^0& e_{2, \y^{0\prime}}^1&e_{2, \y^{0\prime}}^2&e_{2, \y^{0\prime}}^3\\
    e_{3, \y^{0\prime}}^0& e_{3, \y^{0\prime}}^1&e_{3, \y^{0\prime}}^2&e_{3, \y^{0\prime}}^3\\
    e_{4, \y^{0\prime}}^0& e_{4, \y^{0\prime}}^1&e_{4, \y^{0\prime}}^2&e_{4, \y^{0\prime}}^3\\
    e_{5, \y^{0\prime}}^0& e_{5, \y^{0\prime}}^1&e_{5, \y^{0\prime}}^2&e_{5, \y^{0\prime}}^3\\
    e_{6, \y^{0\prime}}^0& e_{6, \y^{0\prime}}^1&e_{6, \y^{0\prime}}^2&e_{6, \y^{0\prime}}^3
    \end{pmatrix}(\y').
\end{equation*}
\begin{enumerate}
    \item For $m^+=0$, we have $m'=6-m^+=6$ and 
    \begin{equation*}
    M (\y')=
        \begin{pmatrix}
            b^0_{1,\varphi}& b^1_{1,\varphi}& b^2_{1,\varphi}&b^3_{1,\varphi}\\
            
            b^0_{2,\varphi}& b^1_{2,\varphi}& b^2_{2,\varphi}&b^3_{2,\varphi}\\
            \k^{+,0}_{\y'^0}&0&0&0\\
            0&\k^{+,0}_{\y'^0}&0&0\\
            0&0&\k^{+,0}_{\y'^0}&0\\
            0&0&0&\k^{+,0}_{\y'^0}
        \end{pmatrix}(\y').
    \end{equation*}
 \item For $m^+=1$, we have $m'=6-m^+=5$ and 
    \begin{equation*}
    M (\y')=
        \begin{pmatrix}
            b^0_{1,\varphi}& b^1_{1,\varphi}& b^2_{1,\varphi}&b^3_{1,\varphi}\\
            
            b^0_{2,\varphi}& b^1_{2,\varphi}& b^2_{2,\varphi}&b^3_{2,\varphi}\\
            \k^{+,0}_{\y'^0}&\k^{+,1}_{\y'^0}&0&0\\
            0&\k^{+,0}_{\y'^0}&\k^{+,1}_{\y'^0}&0\\
            0&0&\k^{+,0}_{\y'^0}&\k^{+,1}_{\y'^0}
        \end{pmatrix}(\y').
    \end{equation*}
    \item For $m^+=2$, we have $m'=6-m^+=4$ and 
    \begin{equation*}
    M (\y')=
        \begin{pmatrix}
            b^0_{1,\varphi}& b^1_{1,\varphi}& b^2_{1,\varphi}&b^3_{1,\varphi}\\
            
            b^0_{2,\varphi}& b^1_{2,\varphi}& b^2_{2,\varphi}&b^3_{2,\varphi}\\
            \k^{+,0}_{\y'^0}&\k^{+,1}_{\y'^0}&\k^{+,2}_{\y'^0}&0\\
            0&\k^{+,0}_{\y'^0}&\k^{+,1}_{\y'^0}&\k^{+,2}_{\y'^0}
        \end{pmatrix}(\y').
    \end{equation*}
\end{enumerate}
There exists a $4\times 4$ sub-matrix such that $\det  M (\y')\neq 0.$ Since the rank of $M (\y')$ is the same as the maximum of the ranks of the largest square sub-matrices of $M (\y')$, then $\rank M (\y')=4.$\\
\\
Recall that  $m' = m^- + 2 \geq 4$.

\medskip
We now set 
\begin{align}
  \label{eq: boundary symbol}
  \un{e_{j, \y^{0\prime}}} (\y', \z) = \sum_{\ell=0}^{3} e_{j,
  \y^{0\prime}}^\ell (\y') z_\ell  = \sum_{\ell=0}^{3}  M_{j, \ell}
  (\y') z_\ell , \qquad \z = (z_0, \dots, z_3)
  \text{ and }j=1,\dots, m' .
\end{align}
in agreement with the notation introduced in \eqref{eq: symbol z} in Section~\ref{sec: Boundary  differential quadratic forms}.
One has the following positivity result. 
\begin{lemma}
  \label{lemma: postivity boundary form}
  Let the \LS condition hold at a point $\y^{0\prime}  = (s^0,x^0,\sigma^0,
\xi^{0\prime}, \tau^0,\gamma^0,\eps^0)\in\R\times \partial
 \Omega \times\R\times \R^{d-1}\times [0,+\infty) \times [1,+\infty)\times (0,1]$ for
 $(Q, B_1, B_2, \varphi)$ and let $\U$ be as given by
 Proposition~\ref{prop: stability LS algegraic conditions}.
 Then, if $\y' \in \U$ there exists $C>0$ such that 
 \begin{align*}
  \sum_{j=1}^{m'} \big|\un{e_{j, \y^{0\prime}}} (\y', \z)\big|^2 \geq
   C |\z|_{\C^{4}}^2, 
   \qquad \z = (z_0, \dots, z_3) \in \C^{4}.
\end{align*}
\end{lemma}
\begin{proof}
In $\C^4$ define  the bilinear form 
\begin{align*}
  \Sigma_{\mathscr B} (\z, \z') 
  = \sum_{j=1}^{m'} \un{e_{j,\y^{0\prime}}} (\y', \z)  
  \ovl{\un{e_{j,\y^{0\prime}}} (\y', \z')}.
\end{align*}
With \eqref{eq: boundary symbol} one has
\begin{align*}
  \Sigma_{\mathscr B} (\z, \z') 
  = \biginp{M (\y') \z}{M (\y') \z'}_{\C^{m'}}
  = \biginp{\transp \ovl{M (\y')}  M (\y') \z}{\z'}_{\C^{4}}.
\end{align*}
As $\rank \transp \ovl{M (\y')}  M (\y') = \rank  M (\y') =4$ by
Lemma~\ref{lemma: algebraic condition LS} one obtains
the result. 
\end{proof}

\section{Estimate for the boundary norm under
    \LS condition}\label{sec: boundary norm under LS}
Near $x^0 \in \partial \Omega$ we consider two boundary operators $B_1$
and $B_2$. As in Section~\ref{sec: LS condition for conjugated
  bilaplace} the associated conjugated operators are denoted by 
$B_{j,\varphi}$, $j=1,2$ with respective principal symbols $b_{j,\varphi}(s,x,\xi,\tau)$.

The main result of this section is the following proposition for the
fourth-order conjugated operator $Q_{\varphi}$. It is key in the final
result of the present article. It states that all traces are
controlled by norms of $B_{1,\varphi} v\br$ and $B_{2,\varphi} v\br$ if the
\LS condition holds for $(Q,B_1, B_2,\varphi)$.

\begin{proposition}
  \label{prop: local boundary estimate}
Let $Q=\Delta^2+D_s^4$ and $(s^0,x^0) \in (0,S_0)\times\partial \Omega$, with
  $\Omega$ locally given by $\{ x_d>0\}$ and $S_0>0$.
  Assume that  $(Q, B_1, P_2,\varphi)$ satisfies the
  \LS condition of Definition~\ref{def: LS after conjugation} at  $\y' =(s, x,\xi',\sigma,\tau,\gamma,\eps)$
  for all $(\sigma,\xi',\tau,\gamma,\eps) \in \R\times\R^{d-1} \times (0,+\infty) \times
  [1,+\infty)\times (0,1]$ such that $\tau >0$.

  Then, there exist $W^0$ a \nhd of $(s^0,x^0)$, $C>0$, $\tau_0>0$ such that
  \begin{align*}
    \normsc{ \trace(v)}{3,1/2} \leq C 
    \big(
    \Norm{Q_{\varphi}v}{+}
    +\sum\limits_{j=1}^{2}
    \normsc{B_{j, \varphi}v\br}{7/2-k_{j}}
    + \Normsc{v}{4,-1}
    \big),
  \end{align*}
  for $\tau\geq \tau_0, \gamma\geq 1$  and $v\in \Cbarc(W^0_+)$.
\end{proposition}
For the proof of Proposition~\ref{prop: local boundary estimate} we
start with a microlocal version of the result.

\subsection{A microlocal estimate}
\begin{proposition}
  \label{prop: microlocal boundary estimate}
Let $(s^0,x^0) \in (0,S_0)\times\partial \Omega$, with
  $\Omega$ locally given by $\{ x_d>0\}$ and $S_0>0$ and let $W$ be a bounded open
  \nhd of $(s^0,x^0)$ in $\R\times\R^d$. Let $(\sigma^0,\xi^{0\prime},\tau^0,\gamma^0,\eps^0)\in\R\times \R^{d-1}\times (0,+\infty)\times[1,\infty)\times(0,1]$ nonvanishing with $\tau^0 >0$
  and such that $(P, B_1, B_2,\varphi)$ satisfies the
  \LS condition of Definition~\ref{def: LS after conjugation} at $\y^{0\prime}
  = (s^0,x^0,\sigma^0, \xi^{0\prime},\tau^0,\gamma^0,\eps^0)$.
 
  Then, there exists $\U$ a
  conic neighborhood of $\y^{0\prime}$ in
  $W\times\R\times \R^{d-1}\times (0,+\infty) \times [1,+\infty)\times (0,1]$ where
  $\tau \geq\tau^0$  such that
  if  $\chi \in S(1,g_{\mathsf T})$, homogeneous of degree $0$ in $(\sigma,\xi',\tau)$ with
  $\supp (\chi)\subset \U$, there exist $C>0$ and $\tau_0>0$ such that
  \begin{multline*}
    \normsc{\trace(\Opt(\chi) v)}{3,1/2} \leq C \Big( 
    \sum\limits_{j=1}^{2}
    \normsc{B_{j, \varphi}v\br}{7/2-k_{j}}
    + \Norm{Q_{\varphi} v}{+}
    + \Normsc{v}{4,-1}
    + \normsc{\trace(v)}{3,-1/2} 
    \Big),
  \end{multline*}
  for  $\tau\geq \tau_0$  and $v\in \Cbarc(W_+)$.
\end{proposition}
\begin{proof}
We choose a conic open \nhd $\U_0$ of $\y^{0\prime}$ according to
Proposition~\ref{prop: stability LS algegraic conditions} and such that $\U_0 \subset W\times\R\times \R^{d-1}\times (0,+\infty) \times [1,+\infty)\times (0,1]$ . Assume
moreover that $\tau \geq \tau_0$ in $\U_0$.

In Section~\ref{sec: Symbol postivity at the boundary} we introduced
the symbols $e_{j, \y^{0\prime}} (\y', \xi_d)$, $j=1, \dots, m' = m^-
+2 = 6 - m^+$. 
For a conic set $\mathscr{W}$ denote $\mathbb S_{\mathscr{W}} = \{ \y' =(s, x,\sigma,\xi',\tau,\gamma,\eps) \in \mathscr{W}; \ |(\sigma,\xi', \tilde{\tau})|=1\}$. 

Consequence of the \LS condition holding at
$\y^{0\prime}$, by
Lemma~\ref{lemma: postivity boundary form} for all $\y' \in \mathbb S_{\U_0}$ there exists $C>0$ such
that  
\begin{align*}
  \sum_{j=1}^{m'} \big|\un{e_{j, \y^{0\prime}}} (\y', \z)\big|^2 \geq
   C |\z|_{\C^{4}}^2, 
   \qquad \z = (z_0, \dots, z_3) \in \C^{4}.
\end{align*}
Let $\U_1$ be a second conic open \nhd of $\y^{0\prime}$ such that
$\ovl{\U_1} \subset \U_0$. 
Since $\mathbb S_{\ovl{\U_1}}$ is compact (recall that $W$ is bounded), there exists $C_0>0$ such that 
\begin{align*}
  \sum_{j=1}^{m'} \big|\un{e_{j, \y^{0\prime}}} (\y', \z)\big|^2 \geq
   C_0 |\z|_{\C^{4}}^2, 
   \qquad \z = (z_0, \dots, z_3) \in \C^{4}, \ \y' \in \mathbb S_{\ovl{\U_1}}.
\end{align*}
Introducing the map $N_{t}\y'=(s,x, t\sigma, t \xi', t \tilde{\tau})$, for
$\y' =(s, x,\xi',\sigma,\tau,\gamma,\eps)$ with $t =\lsct^{-1}$ 
one has 
\begin{align}
  \label{eq: symbol positivity 1}
  \sum_{j=1}^{m'} \big|\un{e_{j, \y^{0\prime}}} (N_t \y', \z)\big|^2 \geq
   C_0 |\z|_{\C^{4}}^2, 
   \qquad \z = (z_0, \dots, z_3) \in \C^{4}, \ \y' \in \U_1, 
\end{align}
since $N_t \y' \in \mathbb S_{\ovl{\U_1}}$.
Now, for $j=1,2$  one has
\begin{align*}
  \un{e_{j, \y^{0\prime}}} (\y', \z)= \sum_{\ell=0}^{k_j} e_{j,
  \y^{0\prime}}^\ell (\y') z_\ell,
\end{align*}
with $ e_{j, \y^{0\prime}}^\ell (\y')$ homogeneous of degree
$k_j-\ell$, 
and for $3\leq j \leq m'$ one has 
\begin{align*}
  \un{e_{j, \y^{0\prime}}} (\y', \z)= \sum_{\ell=0}^{3} e_{j,
  \y^{0\prime}}^\ell (\y') z_\ell,
\end{align*}
with $ e_{j, \y^{0\prime}}^\ell (\y')$ homogeneous of degree
$m^++j-\ell-3$.
We define $\z' \in \C^4$ by $z'_\ell = t^{\ell-7/2} z_\ell$, $\ell=0,
\dots, 3$. 
One has 
\begin{align*}
  \un{e_{j, \y^{0\prime}}} (N_t \y', \z')= t^{k_j-7/2} \un{e_{j,
  \y^{0\prime}}} (\y', \z), \qquad j=1,2,
\end{align*}
and 
\begin{align*}
  \un{e_{j, \y^{0\prime}}} (N_t \y', \z')= t^{m^++j-13/2} \un{e_{j,
  \y^{0\prime}}} (\y', \z), \qquad j=3,\dots,m'.
\end{align*}
Thus, from \eqref{eq: symbol positivity 1} we deduce
\begin{align}
  \label{eq: symbol positivity 2}
  \sum_{j=1}^{2} \lsct^{2(7/2-k_j)} \big|\un{e_{j, \y^{0\prime}}} (\y', \z)\big|^2 
  + \sum_{j=3}^{m'} \lsct^{2(13/2-m^+-j)} 
  \big|\un{e_{j,\y^{0\prime}}} (\y', \z)\big|^2\notag
  \\
  \geq
   C_0 \sum_{\ell=0}^{3} \lsct^{2(7/2-\ell)}| z_\ell|^2, 
\end{align}
for $\z = (z_0, \dots, z_3) \in \C^{4}$, and $\y' \in \U_1$,
since $t = \lsct^{-1}.$ 

We now choose $\U$ a conic open neighborhood of $\y^{0\prime}$, such that
$\ovl{\U}\subset\U_1$.  Let $\chi\in S(1,g_{\mathsf T})$ be as in the statement and let  $\tchi\in
S(1,g_{\mathsf T})$ be homogeneous of degree $0$, with
$\supp(\tchi)\subset\U_1$ and $\tchi\equiv 1$ in a neighborhood of
$\ovl{\U}$, and thus in a \nhd of $\supp(\chi)$.

For $j=3, \dots, m'$ one has $e_{j, \y^{0\prime}}(\y',\xi_d) =
\k^+_{\y^{0\prime}} (\y',\xi_d) \xi_d^{j-3} \in \Ssc^{m^++j-3,0}$. 
Set $E_{j} = \Op ( \tchi e_{j, \y^{0\prime}})$. The introduction
of $\tchi$ is made such that $\tchi e_{j, \y^{0\prime}}$ is defined on
the whole tangential phase-space. Observe that 
\begin{align*}
  \mathscr B (w) &= \sum\limits_{j=1}^{2} \normsc{B_{j, \varphi}w\br}{7/2-k_{j}}^2
  +\sum\limits_{j=3}^{m'} \normsc{E_{j}w\br}{13/2-m^+-j}^2\\
  &=\sum\limits_{j=1}^{2} 
    \norm{\Lsct^{7/2-k_{j}} B_{j, \varphi}w\br}{\partial}^2
    + \sum\limits_{j=3}^{m'} 
    \norm{\Lsct^{13/2-m^+-j} E_{j}w\br}{\partial}^2
\end{align*}
is a boundary quadratic form of type $(3,1/2)$ as in
Definition~\ref{def: boundary quadratic form}.
From Proposition~\ref{prop: boundary form -Gaarding tangentiel} (G{\aa}rding inequality for a boundary differential form of type $(m-1,r)$) and inequality
\eqref{eq: symbol positivity 2} we have 
\begin{align}
  \label{eq: boundary inequality - est traces}
  \normsc{\trace(u)}{3,1/2}^2 
  &\lesssim  \sum\limits_{j=1}^{2} \normsc{B_{j, \varphi}u\br}{7/2-k_{j}}^2
  +\sum\limits_{j=3}^{m'} \normsc{E_{j}u\br}{13/2-m^+-j}^2
  +  \normsc{\trace(v)}{3,-M}^2.
\end{align}
for $u = \Opt(\chi) v$ and $\tau \geq \tau_0$ chosen \suff
large.

\bigskip
In $\U_1$ one can write 
\begin{align*}
  q_{\varphi} = q_{\varphi}^+ q_{\varphi}^- 
  = \k^+_{\y^{0\prime}} \k^-_{\y^{0\prime}},
 \end{align*}
with $\k^+_{\y^{0\prime}}$ of degree $m^+$ and $\k^-_{\y^{0\prime}}$
of degree $m^-$.
In fact we set 
\begin{align*}
   \tk^+_{\y^{0\prime}} (\y') 
   = \prod_{j \in \mathscr{E}^+} \big(\xi_d -\tchi \pi_j(\y')\big), 
   \qquad 
   \tk^-_{\y^{0\prime}} (\y') 
   = \prod_{j \in \mathscr{E}^-} \big(\xi_d -\tchi \pi_j(\y')\big),
 \end{align*}
with the notation of Section~\ref{sec: Stability LS-conj},
thus making the two symbols defined on
the whole tangential phase-space.
In $\U$, one has also 
\begin{align*}
  q_{\varphi} 
  = \tk^+_{\y^{0\prime}} \tk^-_{\y^{0\prime}}.
 \end{align*}
The factor $\tk^-_{\y^{0\prime}}$ is associated with roots with
negative imaginary part. With Lemma~\ref{el1} given in
Section~\ref{SB} one has the following microlocal elliptic estimate
\begin{equation*}
  \Normsc{\Opt(\chi)w}{m^-}
  + \normsc{\trace(\Opt(\chi)w)}{m^--1, 1/2}
  \lesssim 
  \Norm{\Opt(\tk^-_{\y^{0\prime}}) \Opt(\chi)w}{+}
  + \Normsc{w}{m^-,-M},
\end{equation*}
for $w \in \ovl{\scrS}(\R^N_+)$ and $\tau \geq  \tau_0$ chosen \suff large. 
We apply this inequality to $w=\Opt(\tk^+_{\y^{0\prime}} )v$. 
Since 
\begin{align*}
  \Opt(\tk^-_{\y^{0\prime}}) \Opt(\chi)\Opt(\tk^+_{\y^{0\prime}} ) =
  \Opt(\chi) Q_{\varphi} \mod \Psisc^{4,-1},
\end{align*}
one obtains
\begin{align*}
  \normsc{\trace(\Opt(\chi) \Opt(\tk^+_{\y^{0\prime}})v)}{m^--1, 1/2}
  \lesssim 
  \Norm{Q_{\varphi} v}{+}
  + \Normsc{v}{4,-1}.
\end{align*}
With $[\Opt(\chi), \Opt(\tk^+_{\y^{0\prime}})] \in \Psisc^{m^+,-1}$ one
then has
\begin{align*}
  \normsc{\trace(\Opt(\tk^+_{\y^{0\prime}}) u)}{m^--1, 1/2}
  \lesssim 
  \Norm{Q_{\varphi} v}{+}
  + \Normsc{v}{4,-1}
  + \normsc{\trace(v)}{3,-1/2},
\end{align*}
with $u= \Opt(\chi) v$ as above, using that $m^+ + m^-=4$. 
Note that 
\begin{align*}
  &\normsc{\trace(\Opt(\tk^+_{\y^{0\prime}}) u)}{m^--1, 1/2}\\
  & \qquad \asymp \sum_{j=0}^{m^--1} 
  \normsc{D_d^j \Opt(\tk^+_{\y^{0\prime}}) u\br}{m^--j- 1/2}\\
  &\qquad \gtrsim \sum_{j=3}^{m'} 
  \normsc{E_j u\br}{5/2+ m^--j} - \normsc{\trace(v)\br}{3, - 1/2} ,
\end{align*}
using that $\xi_d^j \tk^+_{\y^{0\prime}} = \tchi e_{j+3,
  \y^{0\prime}}$ in a conic \nhd of $\supp(\chi)$ and using that 
$m^- = m'-2$.
We thus obtain 
\begin{align*}
   \sum_{j=3}^{m'}  \normsc{E_{j}u\br}{13/2-m^+-j}
  \lesssim 
  \Norm{Q_{\varphi} v}{+}
  + \Normsc{v}{4,-1}
  + \normsc{\trace(v)}{3,-1/2},
\end{align*}
since $13/2-m^+ = 5/2+ m^-$.
With \eqref{eq: boundary inequality - est traces} then one finds
\begin{align*}
  \normsc{\trace(u)}{3,1/2} 
  &\lesssim  \sum \limits_{j=1}^{2} \normsc{B_{j, \varphi}u\br}{7/2-k_{j}}^2
  + \Norm{Q_{\varphi} v}{+}
  + \Normsc{v}{4,-1} + \normsc{\trace(v)}{3,-1/2}.
\end{align*}
In addition, observing that 
\begin{align*}
  \normsc{B_{j, \varphi}u\br}{7/2-k_{j}}\lesssim
  \normsc{B_{j, \varphi}v\br}{7/2-k_{j}} + \normsc{\trace(v)}{3,-1/2},
\end{align*}
the result of Proposition~\ref{prop: microlocal boundary estimate}
follows.
\end{proof}

\subsection{Proof of Proposition~\ref{prop: local boundary estimate}}
\label{sec: proof prop: local boundary estimate}
  As mentioned above the proof relies on a patching procedure of
  microlocal estimates given by Proposition~\ref{prop: microlocal
boundary estimate}. 

  We set 
  \begin{align*}
    \Gamma^{d} = \{ (\sigma, \xi',\tau,\gamma,\eps)\in\R\times \R^{d-1}\times (0,+\infty)\times (0,1] ;\ \tau \geq \tau_0,\gamma\geq \gamma_0 \},
  \end{align*}
  and 
  \begin{align*}
    \mathbb{S}^{d} = \{(\sigma, \xi',\tau,\gamma,\eps)\in
    \Gamma^{d} ;\ |(\sigma,\xi',\tilde{\tau})|=1\}.
  \end{align*}

 Consider $(\sigma^0,\xi^{0\prime},\tau^0,\gamma^0,\eps^0) \in \mathbb{S}^{d}$.
  Since the \LS condition holds at \\
  $\y^{0\prime} = (\sigma^0,x^0, \xi^{0\prime},\tau^0,\gamma^0,\eps^0)$, we can invoke 
  Proposition~\ref{prop: microlocal boundary estimate}:
  \begin{enumerate}
  \item There exists a conic open neighborhood
  $\U_{\y^{0\prime}}$ of $\y^{0\prime}$ in
  $W \times \R^{d-1}\times (0,+\infty) \times [1,+\infty)\times (0,1]$ where
  $\tau \geq \tau_0, \gamma\geq \gamma_0$;
  \item For any $\chi_{\y^{0\prime}} \in S(1,g_{\mathsf T})$ homogeneous of
  degree $0$ supported in $\U_{\y^{0\prime}}$ the estimate of
  Proposition~\ref{prop: microlocal boundary estimate} applies to $\Opt(\chi_{\y^{0\prime}} ) v$ for $\tau \geq \tau_0, \gamma\geq \gamma_0$. 
  \end{enumerate}
  Without any loss of generality we may choose $\U_{\y^{0\prime}}$ of the form $\U_{\y^{0\prime}} = \scrO_{\y^{0\prime}} \times
  \Gamma_{\y^{0\prime}}$, with\\ $\scrO_{\y^{0\prime}} \subset W$ an open \nhd of
  $(s^0,x^0)$ and $\Gamma_{\y^{0\prime}}$ a conic open \nhd of
  $(\sigma^0,\xi^{0\prime},\tau^0,\gamma^0,\eps^0)$ in $\R\times\R^{d-1} \times
  (0,+\infty)\times[1,+\infty)\times (0,1]$ where $\tau \geq \tau^0, \gamma\geq \gamma_0$. 

  Since $\{ (s^0,x^0)\} \times \mathbb{S}^{d}$ is compact we can
  extract a finite covering of it by open sets of the form of
  $\U_{\y^{0\prime}}$. We denote by
  $\tilde{\U}_i$, $i \in I$ with $|I| < \infty$, such a finite covering. This is
  also a finite covering of 
  $\{(s^0,x^0)\} \times \Gamma^{d}$.

  Each $\tilde{\U}_i$ has the form $\tilde{\U}_i =  \scrO_i \times
  \Gamma_i$, with $\scrO_i$ an open \nhd of
  $(s^0,x^0)$ and $\Gamma_i$ a conic open set in $\R^{d-1}\times [0,+\infty) \times
  [0,+\infty)$ where $\tau \geq \tau^0$. We set $\scrO = \cap_{i \in I} \scrO_i$ and 
  $\U_i = \scrO \times \Gamma_i$, $i \in I$.

  Let $W^0$ be an open \nhd of $(s^0,x^0)$ such that $W^0 \Subset
  \scrO$. The open sets $\U_i$
  give also an open covering of $\ovl{W^0} \times \mathbb{S}^{d}$ and  $\ovl{W^0} \times \Gamma^{d}$.
  With this second covering we associate a partition of unity
  $\chi_i$, $i \in I$, of $\ovl{W^0}  \times \mathbb{S}^{d}$,
where each $\chi_i$ is chosen smooth and homogeneous
  of degree one for $|(\sigma, \xi', \tilde{\tau}, )| \geq 1$, that is:
  \begin{align*}
    \sum\limits_{i \in I}\chi_i(\y')=1 
  \end{align*}
   for $\y' = (s, x,\sigma, \xi', \tau, \gamma,\eps)$ in a \nhd
    of $\ovl{W^0} \times \Gamma^{d}$,
    and $|(\sigma,\xi', \tilde{\tau})| \geq 1$.

  Let $u \in \Cbarc(W^0_+) $.
 Since each $\chi_i$ is in $S(1,g_{\mathsf T})$ and 
  supported in $\U_i$,  Proposition~\ref{prop: microlocal boundary
    estimate} applies:
  \begin{multline}
     \label{eq: local boundary estimate-1}
    \normsc{\trace(\Opt(\chi_i) v)}{3,1/2} \leq C_i
    \Big( 
    \sum\limits_{j=1}^{2}
    \normsc{B_{j, \varphi}v\br}{7/2-k_{j}}
    + \Norm{Q_{\varphi} v}{+}
    + \Normsc{v}{4,-1}
    + \normsc{\trace(v)}{3,-1/2} 
    \Big),
    \end{multline}
 for some $C_i>0$, for $\tau\geq \max\limits_i\tau_i$ for some $\tau_i>0,\gamma\geq 1$. 

 We set $\tchi=1-\sum\limits_{i \in I}\chi_i$. One has $\tchi \in
 \cap_{M\in\N}S(\lsct^{-M},g_{\mathsf T})$ microlocally in a \nhd of $\ovl{W^0} \times
 \Gamma^{d}$.
 Thus, considering the definition of $\Gamma^{d}$,  we then have $\tchi \in
 \cap_{M\in\N}S(\lsct^{-M},g_{\mathsf T})$ locally in a \nhd of $\ovl{W^0}$.
  
 For any $M\in \N$ using that $\supp(v) \subset W^0$ one has
\begin{align*}
  \normsc{\trace (v)}{3,1/2}
  &\leq \sum\limits_{i \in I}  \normsc{\trace (\Opt (\chi_i)v)}{3,1/2}
  + \normsc{\trace (\Opt (\tchi)v)}{3,1/2}\\
  &\lesssim 
    \sum\limits_{i \in I}  \normsc{\trace (\Opt (\chi_i)v)}{3,1/2}
  + \normsc{\trace (v)}{3,-M}\\
  &\lesssim 
    \sum\limits_{i \in I}  \normsc{\trace (\Opt (\chi_i)v)}{3,1/2}
  + \Normsc{v}{4,-M}.
\end{align*}
Summing estimates \eqref{eq: local boundary estimate-1} together for $i
\in I$ we thus obtain
\begin{align*}
    \normsc{\trace(v)}{3,1/2} 
  \lesssim
    \sum\limits_{j=1}^{2}
    \normsc{B_{j, \varphi}v\br}{7/2-k_{j}}
    + \Norm{Q_{\varphi} v}{+} + \Normsc{v}{4,-1}
    + \normsc{\trace(v)}{3,-1/2} ,
\end{align*}
for $\tau \geq \max\limits_{i\in I} \tau_i$ for some $\tau_i>0$. 
Therefore, by choosing $\tau \geq \tau_0$ \suff large and $\gamma\geq 1$ one obtains
the result of Proposition~\ref{prop: local boundary estimate}.
\hfill \qedsymbol \endproof
\section{Local Carleman estimate for the augmented conjugated operator}
We recall that $Q=D^4_s+\Delta^2=Q_1Q_2$ with $Q_k=(-1)^kiD^2_s+A$ and $A=-\Delta.$ Setting $Q_{\varphi}=e^{\tau\varphi}Qe^{-\tau\varphi},$
\begin{equation}\label{micro1}
Q_{\varphi}=P_1P_2~~~\text{with}~~~P_k=e^{\tau\varphi}Q_ke^{-\tau\varphi}=(-1)^ki(D_s+i\tau\partial_s\varphi(z))^2+A_\varphi,
\end{equation}
with, in the selected normal geodesic coordinates,
$$A_\varphi=e^{\tau\varphi}Ae^{-\tau\varphi}=(D_d+i\tau\partial_d\varphi(z))^2+R(x,D_x'+i\tau d_{x'}\varphi(z)),~~~z=(s,x).$$
In \cite[Section 4.5]{JL}, a local Carleman estimate for the augmented conjugated operator $\Q_\varphi$ is obtained by combining microlocal estimates in three different microlocal regions.\\

Let $V$ be an open bounded neighborhood of $z_0=(s_0,x_0)$ in $\R^N.$ We use the following result due to J. Le Rousseau and L. Robbiano.

\begin{proposition}[\cite{JL}]\label{carl-fact}
Let $\varphi(z)=\varphi_{\gamma,\eps}(z)$ be define as in section \ref{SB}. There exists an open neighborhood $W$ of $z_0$ in $(0,S_0)\times\R^d$, $W\subset V$, and there exist $\tau_0\geq \tau_*,\gamma\geq 1,\eps_0\in (0,1]$ such that
$$\gamma\Normsc{\tilde{\tau}^{-1}v}{4,0}+\normsc{\trace(v)}{3,1/2} \lesssim \Norm{Q_\varphi v}{+}+\normsc{\trace(v)}{1,5/2}, $$
for $\tau\geq \tau_0,$ $\gamma\geq \gamma_0,$ $\eps\geq \eps_0,$ and $v=w_{|\ovl{\R^N_+}}$ with $w\in\Cinfc(\R^N)$ and $\supp(w)\subset W.$
\end{proposition}
Putting together the estimates of Propositions \ref{prop: local boundary estimate} and \ref{carl-fact}, we produce the following final local estimate for the augmented operator $Q.$
\begin{theorem}\label{carleman-local}
Let $Q=\Delta^2+D_s^4$ and $(s^0,x^0) \in (0,S_0)\times\partial \Omega$, with
  $\Omega$ locally given by $\{ x_d>0\}$ and $S_0>0$. Assume that $(Q,B_1,B_2,\varphi)$ satisfies the \LS condition of Definition \ref{def: LS after conjugation} at $\y'=(s,x,\sigma,\xi',\tau,\gamma,\eps)$ for all $(\sigma,\xi',\tau,\gamma,\eps)\in\R\times\R^{d-1}\times (0,+\infty)\times [1,+\infty)\times(0,1].$
  Let $\varphi(z)=\varphi_{\gamma,\eps}(z)$ be define as in section \ref{SB}. There exists an open neighborhood $W$ of $z_0$ in $(0,S_0)\times\R^d$, $W\subset V$, and there exist $\tau_0\geq \tau_*,\gamma\geq 1,\eps_0\in (0,1]$ such that
$$\gamma\Normsc{\tilde{\tau}^{-1}e^{\tau\varphi}u}{4,0}+\normsc{\trace(e^{\tau\varphi}u)}{3,1/2} \lesssim \Norm{e^{\tau\varphi}Qu}{+}+\sum\limits_{j=1}^{2}
\normsc{e^{\tau\varphi}B_{j}u\br}{7/2-k_{j}}, $$
for $\tau\geq \tau_0,$ $\gamma\geq \gamma_0,$ $\eps\geq \eps_0,$ and  $u\in\Cbarc(W_+).$
\end{theorem}
We recall that the notion of the function space is introduced in (\ref{eq: notation Cbarc}).\\
\\
Using the notations of Section \ref{tangential-semi} for semi-classical norms, we can write the Carleman inequality of Theorem \ref{carleman-local} as follows:
$$\gamma\sum_{|\alpha|\leq 4}\|\tilde{\tau}^{3-|\alpha|}e^{\tau\varphi}D^{\alpha}_{s,x}u\|_{+}+\sum\limits_{r=0}^3| e^{\tau\varphi}D^r_{x_N}u_{|\partial Z}|_{7/2-r,\tilde{\tau}}\lesssim \Norm{e^{\tau\varphi}Q u}{+}+\sum\limits_{j=1}^{2}
    \normsc{e^{\tau\varphi}B_{j}u\br}{7/2-k_{j}}.$$
    
\section{Spectral inequality and an application}
Before we state and prove the spectral inequality, we start by stating an interpolation-type inequality which follows from the Carleman inequality of Theorem  \ref{carleman-local}. Finally, as an application of such a spectral inequality, we deduce a null-controllability result of the fourth order parabolic equation (\ref{eq: parabolic}).
\subsection{An interpolation inequality}
Let $S_0>0$ and $\alpha\in (0,S_0/2)$. We introduce the cylinder $Z=(0,S_0)\times\Omega$ and we define $Y=(\alpha,S_0-\alpha)\times\Omega$ for some $\alpha>0.$ We recall that $Q$ denotes the augmented elliptic operator $Q:=D^4_s+\Delta^2.$\\
 Following the proof the the interpolation inequality of Theorem 5.1 in \cite[Section 5]{JL}, one can derive the following interpolation inequality.
\begin{theorem}(interpolation inequality)\label{Interp}.
Let $\scrO$ be nomempty subset of $\Omega.$ There exist $C>0$ and $\delta\in (0,1)$ such that for $u\in H^4(Z)$ that satisfies
$$B_1u_{|{x\in\partial \Omega}}=0,~~~~B_2u_{|{x\in\partial \Omega}}=0,~~~~~~~~~~~~~~s\in (0,S_0),$$ we have
\begin{equation}\label{interpolation inequality}
\|u\|_{H^3(Y)}\leq  C\|u\|^{1-\delta}_{H^3(Z)}\left(\|Qu\|_{L^2(Z)}+\sum\limits_{j=1}^{2}|B_{j}u_{|s=0}|_{H^{3-k_{j}}(\scrO)}\right)^\delta,
\end{equation}
where $H^m(\Omega)$ denote the classical Sobolev spaces in $\Omega.$
\end{theorem}
\subsection{Spectral inequality}\label{spectral}
Let $\phi_j$ and $\mu_j$ be eigenfunctions and associated eigenvalues of the bi-Laplace operator $P=\Delta^2$ with the boundary operators $B_1$ and $B_2$, which form a Hilbert basis for $L^2(\Omega)$ (see section \ref{sec: Examples  of boundary operators yielding symmetry} ), with \begin{align*}
0< \mu_0 \leq \mu_1 \leq \cdots \leq \mu_k \leq \cdots. 
\end{align*}
We now prove the spectral inequality of Theorem \ref{spectineg}, viz., 
for some $K>0,$
\begin{equation}\label{prfspec}
\|y\|_{L^2(\Omega)}\leq K e^{K\mu^{1/4}}\|y\|_{L^2(\omega)},~~\mu>0,~~
y\in\Span\{\phi_j;~\mu_j\leq \mu\}.
\end{equation}
We use the following corollary, whose proof can be adapted to the proof of Corollary 3.3  in \cite[Section 3.1]{JGL-vol2}.
\begin{corollary}\label{lsineq}
Assume that $(P,B_1,B_2)$ fulfills the \LS condition on $\partial\Omega.$ Then there exists $C>0$ such that 
$$\|u\|_{H^4(\Omega)}\leq C\left(\|Pu\|_{L^2(\Omega)}+|B_1u_{|\partial\Omega}|_{H^{7/2-k_1}(\partial\Omega)}+|B_1u_{|\partial\Omega}|_{H^{7/2-k_2}(\partial\Omega)}\right)$$ for all $u\in H^4(\Omega).$
\end{corollary}
We recall that $k_1$ and $k_2$ denote respectively the orders of $B_1$ and $B_2.$
\begin{proof}
Let $\mu>0$ and $y\in\Span\{\phi_j;~\mu_j\leq \mu\},$ meaning that $y=\sum\limits_{\mu_j\leq\mu}\alpha_j\phi_j.$ We set 
$$u(s,x)=\sum\limits_{\mu_j\leq\mu}\alpha_j\mu_j^{-3/4}f(s\mu_j^{1/4})\phi_j(x).$$ Here $\alpha_j\in\C$ for $j=0,\dots, k$ with $k\in\N,$ where $f(s)=\beta\sin(\beta s)\cosh(\beta s)-\beta\cos(\beta s)\sinh(\beta s)$ with $\beta=\sqrt{2}/2.$ Straightforward computation shows that $D_s^4f=\partial^4_sf=-f.$ Therefore, we have $Qu=0$ with $Q=D^4_s+\Delta^2$. Moreover, we note that
$$f(0)=f'(0)=f^{(2)}(0)=0, f^{(3)}(0)=1$$
and $$f(s)=h(\beta s)=\frac{1}{2}\left(e^{-s}\cos(s-\frac{\pi}{4})-e^s\cos(s+\frac{\pi}{4})\right).$$
Applying inequality (\ref{interpolation inequality}) of Theorem \ref{Interp} to $u$ reads $$\|u\|_{H^3(Y)}\leq C \|u\|^{1-\delta}_{H^3(Z)}\left(\sum\limits_{j=1}^{2}|B_{j}u_{|s=0}|_{H^{3-k_{j}}(\scrO)}\right)^\delta.$$
As $\partial_s^3u_{|s=0}=\sum\limits_{\mu_j\leq\mu}\alpha_j\phi_j(x)=y,$ we have $\sum\limits_{j=1}^{2}|B_{j}u_{|s=0}|_{H^{3-k_{j}}(\scrO)}\asymp\|y\|_{L^2(\scrO)}.$ Then
$$\|u\|_{H^3(Y)}\leq C \|u\|^{1-\delta}_{H^3(Z)}\|y\|_{L^2(\scrO)}^\delta.$$ We also have $\|u\|_{L^2(Y)}\lesssim\|u\|_{H^3(Y)}$ with
\begin{align*}
\|u\|_{L^2(Y)}&=\sum\limits_{\mu_j\leq\mu}\mu_j^{-3/2}|\alpha_j|^2\int_{\alpha}^{S_0-\alpha}f(\mu_j^{1/4}s)^2ds.
\end{align*}
Using the change of variable $t=\beta\mu_j^{1/4}s,$ it follows that
\begin{align*}
\|u\|_{L^2(Y)}&=\sum\limits_{\mu_j\leq\mu}\mu_j^{-7/4}|\alpha_j|^2\int_{\alpha\gamma\mu_j^{1/4}}^{(S_0-\alpha)\gamma\mu_j^{1/4}}f(\gamma^{-1}t)^2dt\\
&=\sum\limits_{\mu_j\leq\mu}\mu_j^{-7/4}|\alpha_j|^2\int_{\alpha\gamma\mu_j^{1/4}}^{(S_0-\alpha)\gamma\mu_j^{1/4}}h(t)^2dt\\
&\gtrsim \mu^{-7/4}\sum\limits_{\mu_j\leq\mu}|\alpha_j|^2=\mu^{-7/4}\|y\|_{L^2(\Omega)};
\end{align*}
we have used the following lemma, whose proof can be found in \cite[Lemma 5.7]{JL}.
\begin{lemma}
    Let $0<a<b$ and $t_0>0.$ There exists $C>0$ such that $\int_{at}^{bt}h(s)^2ds\geq C_0$ for $t\geq t_0.$
\end{lemma}
Hence we obtain
\begin{equation}\label{syl1}
    \|y\|_{L^2(\Omega)}\lesssim \mu^{7/4}\|u\|^{1-\delta}_{H^3(Z)}\|y\|_{L^2(\scrO)}^\delta.
\end{equation}
In order to complete the proof, we now estimate $\|u\|_{H^3(Z)}$ with the use of the following lemma, which by estimate (\ref{syl1}), allows to conclude the proof of Theorem \ref{spectineg}.
\end{proof}
\begin{lemma}
There exists $K>0$ such that $\|u\|_{H^3(Z)}\leq Ke^{K\mu^{1/4}} \|y\|_{L^2(\Omega)}.$
\end{lemma}
\begin{proof}
We have
$$\|u\|_{H^3(Z)}\asymp\sum\limits_{\ell=0}^3\int_{0}^{S_0}\|\partial^{\ell}_su(s,\cdot)\|^2_{H^{3-\ell}(\Omega)}ds\lesssim \sum\limits_{\ell=0}^3\int_{0}^{S_0}\|\partial^{\ell}_su(s,\cdot)\|^2_{H^{4}(\Omega)}ds.$$
Recalling that if $B_1u_{|{\partial\Omega}}=B_2u_{|{\partial\Omega}}=0,$ by Corollary \ref{lsineq} we find
\begin{align*}
    \|\partial^{\ell}_su(s,\cdot)\|^2_{H^{4}(\Omega)}&\lesssim\| \Delta^2\sum\limits_{\mu_j\leq\mu}\alpha_j\mu_j^{(\ell-3)/4}f^{(\ell)}(\mu_j^{1/4}s)\phi_j\|^2_{L^2(\Omega)}\\
    &=\|\sum\limits_{\mu_j\leq\mu}\alpha_j\mu_j^{(\ell+1)/4}f^{(\ell)}(\mu_j^{1/4}s)\phi_j\|^2_{L^2(\Omega)}\\
    &=\sum\limits_{\mu_j\leq\mu}|\alpha_j|^2\mu_j^{(\ell+1)/2}f^{(\ell)}(\mu_j^{1/4}s)^2\lesssim\mu^2 e^{S_0\mu^{1/4}}\sum\limits_{\mu_j\leq\mu}|\alpha_j|^2.
\end{align*}
Integrating this estimate over $(0,S_0)$ and summing over $\ell$ yields
$$\|u\|_{H^3(Z)}\leq Ke^{K\mu^{1/4}} \|y\|_{L^2(\Omega)}.$$
\end{proof}
\subsection{A null-controllability result}
Let $T>0,$ we consider the controlled evolution equation on $(0,T)\times\Omega$ under general boundary conditions:
\begin{equation}\label{evoleq}
    \begin{cases}
\partial_t y +\Delta^2y =\mathbbm{1}_{\Sigma}v
~~~~\text{in}~(t,x) \in  (0,T) \times\Omega,\\
B_1 y_{|(0,T) \times\partial \Omega}= B_2y_{|(0,T) \times\partial \Omega}=0,\\
y_{|t=0}=y^0\in L^2(\Omega),
\end{cases}
\end{equation}
where $\Sigma$ is an open subset of $\Omega$ and $\mathbbm{1}_{\Sigma}\in L^{\infty}$ is such that $\mathbbm{1}_{\Sigma}>0$ on $\Sigma.$ We recall that he function $v\in L^2((0,T)\times\Sigma)$ is the control function. One may ask the following question: can one choose $v$ to drive the solution from its initial condition $y^0$ to zero in final time $T?$ Thanks to the spectral inequality of Theorem \ref{spectineg} one can answer to this null-controllability question in the affirmative.
\begin{theorem}(Null-controllability)
There exists $C>0$ such that for any $y^0\in L^2(\Omega),$ there exists $v\in L^2((0,T)\times\Omega)$ such that the solution to (\ref{evoleq}) vanishes at time $T$ and moreover the control $v$ satisfies the bound $\|v\|_{L^2((0,T)\times\Omega)}\leq C\|y^0\|_{L^2(\Omega)}.$
\end{theorem}
The proof can be adapted in a straightforward manner from the proof scheme in section 6.1 of \cite{JLG} developed for the heat equation.
\addcontentsline{toc}{section}{References}
\bibliography{refs}

\begin{thebibliography}{12}
\providecommand{\natexlab}[1]{#1}
\providecommand{\url}[1]{\texttt{#1}}
\expandafter\ifx\csname urlstyle\endcsname\relax
  \providecommand{\doi}[1]{doi: #1}\else
  \providecommand{\doi}{doi: \begingroup \urlstyle{rm}\Url}\fi

\bibitem[Bellassoued and Le~Rousseau(2015)]{ML}
Mourad Bellassoued and J{\'e}r{\^o}me Le~Rousseau.
\newblock Carleman estimates for elliptic operators with complex coefficients.
  {P}art {I}: Boundary value problems.
\newblock \emph{Journal de Math{\'e}matiques Pures et Appliqu{\'e}es},
  104\penalty0 (4):\penalty0 657--728, 2015.

\bibitem[Br{\'e}zis(2011)]{Brezis:11}
Haim Br{\'e}zis.
\newblock \emph{Functional analysis, Sobolev spaces and partial differential
  equations}.
\newblock Springer, 2011.

\bibitem[Guerrero and Kassab(2019)]{SGK}
Sergio Guerrero and Karim Kassab.
\newblock Carleman estimate and null controllability of a fourth order
  parabolic equation in dimension ${N}\geq 2$.
\newblock \emph{Journal de Math{\'e}matiques Pures et Appliqu{\'e}es},
  121:\penalty0 135--161, 2019.

\bibitem[H{\"o}rmander(1963)]{Hoermander:63}
L.~H{\"o}rmander.
\newblock \emph{Linear Partial Differential Operators}.
\newblock Springer-Verlag, 1963.

\bibitem[H{\"o}rmander(1985)]{LarsH1}
L.~H{\"o}rmander.
\newblock \emph{The Analysis of Linear Partial Differential Operators}, volume
  III.
\newblock Springer-Verlag, 1985.

\bibitem[H{\"o}rmander(1994)]{Hoermander:V3}
L.~H{\"o}rmander.
\newblock \emph{The Analysis of Linear Partial Differential Operators}, volume
  III.
\newblock Springer-Verlag, 1994.

\bibitem[King et~al.(2003)King, Stein, and Winkler]{KBB}
Belinda~B King, Oliver Stein, and Michael Winkler.
\newblock A fourth-order parabolic equation modeling epitaxial thin film
  growth.
\newblock \emph{Journal of mathematical analysis and applications},
  286\penalty0 (2):\penalty0 459--490, 2003.

\bibitem[Le~Rousseau and Lebeau(2012)]{JLG}
J{\'e}r{\^o}me Le~Rousseau and Gilles Lebeau.
\newblock On carleman estimates for elliptic and parabolic operators.
  applications to unique continuation and control of parabolic equations.
\newblock \emph{ESAIM: Control, Optimisation and Calculus of Variations},
  18\penalty0 (3):\penalty0 712--747, 2012.

\bibitem[Le~Rousseau and Robbiano(2019)]{JL}
J{\'e}r{\^o}me Le~Rousseau and Luc Robbiano.
\newblock Spectral inequality and resolvent estimate for the bi-laplace
  operator.
\newblock \emph{Journal of the European Mathematical Society}, 22\penalty0
  (4):\penalty0 1003--1094, 2019.

\bibitem[Le~Rousseau and Zongo(2023)]{Jer-Em}
J{\'e}r{\^o}me Le~Rousseau and Emmanuel Wend-Benedo Zongo.
\newblock Stabilization of the damped plate equation under general boundary
  conditions.
\newblock \emph{Journal de l’{\'E}cole polytechnique—Math{\'e}matiques},
  10:\penalty0 1--65, 2023.

\bibitem[Le~Rousseau et~al.(2022)Le~Rousseau, Lebeau, and Robbiano]{JGL-vol2}
J{\'e}r{\^o}me Le~Rousseau, Gilles Lebeau, and Luc Robbiano.
\newblock \emph{Elliptic Carleman Estimates and Applications to Stabilization
  and Controllability, Volume II: General Boundary Conditions on Riemannian
  Manifolds}, volume~98.
\newblock Birkh{\"a}user, 2022.

\bibitem[Yu(2009)]{HY}
Hang Yu.
\newblock Null controllability for a fourth order parabolic equation.
\newblock \emph{Science in China Series F: Information Sciences}, 52\penalty0
  (11):\penalty0 2127--2132, 2009.

\end{thebibliography}
\end{document}